\documentclass[reqno]{amsart}
\usepackage{amssymb}
\usepackage{amsfonts}
\usepackage{mathrsfs}
\usepackage{xcolor}
\theoremstyle{plain}
    \newtheorem{theorem}{Theorem}[section]
    \newtheorem*{theorem*}{Theorem}
    \newtheorem{lemma}[theorem]{Lemma}
    \newtheorem{proposition}[theorem]{Proposition}
    \newtheorem{corollary}[theorem]{Corollary}

\theoremstyle{definition}
    \newtheorem{definition}{Definition}[section]
    \newtheorem{remark}{Remark}[section]
    
\numberwithin{equation}{section}

\renewcommand{\r}{\right}
\begin{document}
%%%%%%%%%%%%%%%%%%%%%%%%%%%%%%%%%%%%%%%%%%%%%%%%%%%%%%%%

\title[]{Classification and long-time asymptotics of $(2,2)$-sign solitons for the damped nonlinear Klein-Gordon equation}
\author[K. Ishizuka]{Kenjiro Ishizuka
% \thanks{ORCID: 0000-0001-6308-8984}
}
\address[K. Ishizuka]{Azabu Junior and Senior High School, Tokyo, Japan}
\email{k-ishizuka@azabu-jh.net}
\date{}

\begin{abstract}
We consider the damped nonlinear Klein-Gordon equation
\begin{align*}
\partial_t^2u-\Delta u+2\alpha\partial_tu+u-|u|^{p-1}u=0
\end{align*}
on $\mathbb{R}^d$, where $\alpha>0$, $2\leq d\leq5$, and $p>2$ is in the energy-subcritical range. We classify global solutions that converge to a superposition of two positive and two negative translates of the ground state, without imposing any symmetry, coplanarity, or a priori geometric condition on their centers. We prove that every such four-soliton configuration is asymptotically either an alternating collinear configuration or an expanding rhombus with alternating signs. We further determine the precise long-time asymptotics of all four centers.
\end{abstract}

\maketitle

\tableofcontents
%%%%%%%%%%%%%%%%%%%%%%%%%%%%%%%%%%%%%%%%%%

\section{Introduction}
\subsection{Setting of the problem}
We consider the following damped nonlinear Klein-Gordon equation
\begin{align}
\label{DNKG}
\tag{DNKG}
	\left\{
	\begin{aligned}
		&\partial_{t}^2u-\Delta u+2\alpha \partial_{t}u+u-f(u)=0, & (t,x) \in \mathbb{R} \times \mathbb{R}^d,
		\\
		&\left(u(0,x),{\partial}_tu(0,x)\right)=\left(u_{0}(x), u_1(x)\right)\in \mathcal{H},
	\end{aligned}
	\r.
\end{align}
where $\alpha>0$, $2\leq d\leq 5$,  $\mathcal{H}=H^1(\mathbb{R}^d)\times L^2(\mathbb{R}^d)$, $f(u)=|u|^{p-1}u$, with a power $p$ in  the energy-subcritical range, namely
\begin{align*}
2<p<\infty \ \mbox{for}\ d=2\ \mbox{and}\ 2<p<\frac{d+2}{d-2}\ \mbox{for}\ d=3,4,5.
\end{align*}
It follows from \cite{BRS} that the Cauchy problem for \eqref{DNKG} is locally well-posed in the energy space $\mathcal{H}$. Moreover, defining the energy of a solution $\vec{u}=(u, {\partial}_tu)$ by
\begin{align*}
E(\vec{u}(t))=\frac{1}{2}\|\vec{u}(t)\|_{\mathcal{H}}^2-\frac{1}{p+1}\| u(t)\|_{L^{p+1}}^{p+1},
\end{align*}
we have
\begin{align}\label{energydecay}
E(\vec{u}(t_2))-E(\vec{u}(t_1))=-2\alpha \int_{t_1}^{t_2} \| {\partial}_tu(t)\|_{L^2}^2dt.
\end{align}
In addition, the equation \eqref{DNKG} enjoys several invariances:
\begin{align}\label{invariance}
\begin{aligned}
u(t,x)&\mapsto -u(t,x),
\\
u(t,x)&\mapsto u(t,x+y),
\\
u(t,x)&\mapsto u(t+s,x),
\\
u(t,x)&\mapsto u(t,Rx),
\end{aligned}
\end{align}
where $y\in \mathbb{R}^d$ and $s\in \mathbb{R}$, and 
\begin{align*}
R\in SO(d):=\left\{A\in\mathbb{R}^{d\times d}:A^{\mathsf T}A=I_d,\ \det A=1\right\}.
\end{align*}

\subsection{Long-time asymptotics of multi-soliton solutions}

In this paper, we are interested in the dynamics of multi-soliton solutions related to the ground state $Q$, which is the unique positive radial solution in $H^1(\mathbb{R}^d)$ of
\begin{align}\label{Qode}
-\Delta Q+Q-f(Q)=0.
\end{align}
See \cite{K2}. The ground state generates the stationary solution $(Q,0)$. By the invariances in \eqref{invariance}, $(-Q,0)$ and every spatial translate $(Q(\cdot+y),0)$ are also stationary solutions of \eqref{DNKG}.

There has been intensive study of the global behavior of general, possibly large, solutions to nonlinear dispersive equations. A guiding principle is the \textit{soliton resolution conjecture}, which predicts that generic global solutions are asymptotic to a superposition of solitons for large time. For \eqref{DNKG}, Feireisl \cite{F} established such a decomposition along a sequence of times for every bounded global solution. In one space dimension, where all nonzero stationary solutions are translates of $\pm Q$, C\^{o}te, Martel, and Yuan \cite{CMY} proved soliton resolution in the full-time limit and showed that the solitons must have alternating signs.

Even when a soliton decomposition is known, the decomposition alone does not determine which collections of solitons can actually arise or what constraints the dynamics imposes on their parameters. This naturally leads to the problem of classifying the admissible asymptotic configurations of the constituent solitons. For equations with a scaling symmetry, such a classification may involve both the relative scales and the spatial arrangement of the solitons. In the present setting, however, we consider only superpositions of translates of the fixed ground state $Q$. Thus, the relevant parameters are the fixed signs and the time-dependent centers of the solitons.

The analysis of such superpositions generally requires controlling both the mutual interactions among the solitons and the radiative remainder; see, for example, \cite{CJ,IN}. For \eqref{DNKG}, however, the energy dissipation and the decay of the linear damped flow provide control of the radiative component, while the damping causes the soliton velocities to decay. Consequently, the long-time evolution of the centers is governed, to leading order, by the mutual soliton interactions and is described by an effective first-order system.

This damping-induced reduction has enabled considerable progress in this direction for \eqref{DNKG}. C\^{o}te, Martel, Yuan, and Zhao \cite{CMYZ} classified the case of two solitons in general space dimension: the two solitons must have opposite signs, and their centers separate along a fixed straight line. Their analysis also makes explicit the leading-order interaction law, according to which like-signed solitons attract and oppositely signed solitons repel. Combining the sign restrictions obtained in \cite{CD} with the results of the author \cite{I2}, the case of three solitons is also classified: exactly one soliton has the sign opposite to that of the other two, and the three centers are necessarily asymptotically collinear. More recently, the author \cite{I3} classified the case of four solitons in which one soliton has the sign opposite to that of the other three, proving that the uniquely signed soliton converges to a point while the other three separate in an expanding equilateral-triangle configuration centered at that point. Configurations consisting entirely of solitons of the same sign are excluded by \cite{CD}. The analyses of the three-soliton and $(1,3)$-sign four-soliton cases share a common feature: they exploit the uniquely signed soliton as a natural reference and formulate the geometry in terms of the relative position vectors from its center to the other centers.

The only remaining case among four-soliton superpositions consisting of translates of $Q$ and $-Q$ is that of two copies of $Q$ and two copies of $-Q$. In this case, however, no soliton is distinguished by its sign, so no analogous reference center is available. To overcome these difficulties, we develop an analysis of the full interaction system that does not rely on a distinguished reference soliton. This allows us to complete the classification of four-soliton configurations by determining the asymptotic geometry and precise long-time behavior in this remaining case.

\subsection{Main result}
First, we define solutions that behave as the superposition of ground states.
\begin{definition}\label{defKsoli}
A global solution $\vec{u}$ of \eqref{DNKG} is called a \textit{$K$-soliton solution} for $K\in\mathbb{N}$ if there exist $\sigma \in \{-1,1\}^K$ and $z:[0,\infty)\to (\mathbb{R}^d)^K$ such that 
\begin{align}
\label{Ksoli}
\begin{aligned}
\lim_{t\to \infty} \left(\|u(t)-\sum_{k=1}^K\sigma_k Q(\cdot-z_k(t))\|_{H^1}+\|{\partial}_tu(t)\|_{L^2}\right)&=0,
\\
\lim_{t\to\infty}\left( \min_{j\neq k}|z_j(t)-z_k(t)|\right)&=\infty.
\end{aligned}
\end{align}
\end{definition}

\begin{remark}
Hereafter, unless stated otherwise, whenever we refer to a $K$-soliton solution, we mean a solution of \eqref{DNKG}.
\end{remark}

In this paper, we focus on $4$-soliton solutions consisting of two positive and two negative solitons. Accordingly, in order to distinguish not only the number of solitary waves but also whether each component is a copy of $Q$ or of $-Q$, we introduce the notion of $(m,n)$-sign soliton solutions as follows.
\begin{definition}\label{defm,nsoli}
A global solution $\vec{u}$ of \eqref{DNKG} is called a \textit{$(m,n)$-sign soliton solution} for $m,n\in\mathbb{N}\cup\{0\}$ if there exist $z_+:[0,\infty)\to (\mathbb{R}^d)^m$ and $z_-:[0,\infty)\to(\mathbb{R}^d)^n$ such that 
\begin{align}
\label{mnsoli}
\begin{aligned}
\lim_{t\to \infty} \left(\|u(t)-\sum_{k=1}^mQ(\cdot-z_{+,k}(t))+\sum_{k=1}^n Q(\cdot-z_{-,k}(t))\|_{H^1}+\|{\partial}_tu(t)\|_{L^2}\right)&=0,
\\
\lim_{t\to\infty}\left( \min_{(\circ,j)\neq (\bullet,k)}|z_{\circ,j}(t)-z_{\bullet,k}(t)|\right)&=\infty.
\end{aligned}
\end{align}
\end{definition}

\begin{remark}
Since \eqref{DNKG} is invariant under the sign change $u\mapsto -u$, any property established for $(m,n)$-sign soliton  solutions immediately yields the corresponding property for $(n,m)$-sign soliton solutions.
\end{remark}
\begin{remark}
Hereafter, unless stated otherwise, whenever we refer to a $(m,n)$-sign soliton solution, we mean a solution of \eqref{DNKG}.
\end{remark}

As mentioned above, the results of \cite{CD,CMYZ,I2,I3} determine the long-time behavior of solutions asymptotic to superpositions of between two and four solitons in the following cases.

\begin{theorem}[Known results; \cite{CD,CMYZ,I2,I3}]\label{previoustheorem}
Let $\vec{u}$ be a solution of \eqref{DNKG}. Then, the following statements hold.
\begin{enumerate}
\item If $\vec{u}$ is a $2$-soliton solution, there exist a constant $c_1\in \mathbb{R}$ depending only on $\alpha,d,p$, a sign $\sigma=\pm 1$, functions $z_1,z_2:[0,\infty)\to\mathbb{R}^d$, $z_{\infty}\in \mathbb{R}^d$, and $\omega_{\infty}\in S^{d-1}$ such that
\begin{align*}
&\left\| u(t)-\sigma \sum_{k=1}^2 (-1)^kQ(\cdot-z_k(t))\right\|_{H^1}+\|{\partial}_tu(t)\|_{L^2}\lesssim t^{-1},
\\
&z_1(t)=z_{\infty}+\frac{1}{2}\left(\log{t}-\frac{d-1}{2}\log{(\log{t})}+c_1\right)\omega_{\infty}+O\left(\frac{\log{(\log{t})}}{\log{t}}\right),
\\
&z_2(t)=z_{\infty}-\frac{1}{2}\left(\log{t}-\frac{d-1}{2}\log{(\log{t})}+c_1\right)\omega_{\infty}+O\left(\frac{\log{(\log{t})}}{\log{t}}\right).
\end{align*}

\item If $\vec{u}$ is a $3$-soliton solution, there exist a constant $c_2\in \mathbb{R}$ depending only on $\alpha,d,p$, a sign $\sigma=\pm 1$, functions $z_1,z_2,z_3:[0,\infty)\to\mathbb{R}^d$, $z_{\infty}\in \mathbb{R}^d$, and $\omega_{\infty}\in S^{d-1}$ such that
\begin{align*}
&\left\| u(t)-\sigma \sum_{k=1}^3 (-1)^kQ(\cdot-z_k(t))\right\|_{H^1}+\|{\partial}_tu(t)\|_{L^2}\lesssim t^{-1},
\\
&z_1(t)=z_{\infty}+\left(\log{t}-\frac{d-1}{2}\log{(\log{t})}+c_2\right)\omega_{\infty}+O\left(\frac{\log{(\log{t})}}{\log{t}}\right),
\\
&z_2(t)=z_{\infty}+O\left(\frac{\log{(\log{t})}}{\log{t}}\right),
\\
&z_3(t)=z_{\infty}-\left(\log{t}-\frac{d-1}{2}\log{(\log{t})}+c_2\right)\omega_{\infty}+O\left(\frac{\log{(\log{t})}}{\log{t}}\right).
\end{align*}

\item If $\vec{u}$ is a $(1,3)$-sign soliton solution or a $(3,1)$-sign soliton solution, there exist a constant $c_3\in \mathbb{R}$ depending only on $\alpha,d,p$, a sign $\sigma=\pm 1$, functions $z_0,z_1,z_2,z_3:[0,\infty)\to\mathbb{R}^d$, $z_{\infty}\in \mathbb{R}^d$, and $\omega_1,\omega_2,\omega_3\in S^{d-1}$ such that
\begin{align*}
&\left\| u(t)-\sigma\left(Q(\cdot-z_0(t))-\sum_{k=1}^3Q(\cdot-z_k(t))\right)\right\|_{H^1}+\|{\partial}_tu(t)\|_{L^2}\lesssim t^{-1},
\\
&z_0(t)=z_{\infty}+O\left(\frac{\log{(\log{t})}}{\log{t}}\right),
\\
&z_1(t)=z_{\infty}+\left(\log{t}-\frac{d-1}{2}\log{(\log{t})}+c_3\right)\omega_1+O\left(\frac{\log{(\log{t})}}{\log{t}}\right),
\\
&z_2(t)=z_{\infty}+\left(\log{t}-\frac{d-1}{2}\log{(\log{t})}+c_3\right)\omega_2+O\left(\frac{\log{(\log{t})}}{\log{t}}\right),
\\
&z_3(t)=z_{\infty}+\left(\log{t}-\frac{d-1}{2}\log{(\log{t})}+c_3\right)\omega_3+O\left(\frac{\log{(\log{t})}}{\log{t}}\right),
\\
&\omega_1+\omega_2+\omega_3=0.
\end{align*}

\item There are no $(4,0)$-sign or $(0,4)$-sign soliton solutions.
\end{enumerate}
\end{theorem}

In this paper, we also classify the long-time behavior of $(2,2)$-sign soliton solutions, which constitute the remaining sign pattern for $4$-soliton solutions. We prove that every $(2,2)$-sign soliton solution is asymptotically either an alternating collinear configuration or an expanding rhombus with alternating signs.

Before stating the main theorem, we introduce the following terminology.

\begin{definition}
We say that $u,v\in S^{d-1}$ form an \textit{orthogonal pair} if
\begin{align*}
u\cdot v=0.
\end{align*}
\end{definition}

\begin{theorem}\label{maintheorem}
If a solution $\vec{u}$ of \eqref{DNKG} is a $(2,2)$-sign soliton solution, then there exists a constant $c_{\star}\in \mathbb{R}$, depending only on $d,\alpha,p$, such that one of the following alternatives holds.
\begin{enumerate}
\item There exist $\sigma=\pm 1$, $z_1,z_2,z_3,z_4:[0,\infty)\to\mathbb{R}^d$ and $z_{\infty}\in\mathbb{R}^d$ and  $\omega_{\infty}\in S^{d-1}$ such that 
\begin{align}\label{4linetheorem1}
\left\| u(t)-\sigma \sum_{k=1}^4 (-1)^kQ(\cdot-z_k(t))\right\|_{H^1}+\|{\partial}_tu(t)\|_{L^2}\lesssim t^{-1},
\end{align}
\begin{align}\label{4linetheorem2}
\begin{aligned}
z_1(t)&=z_{\infty}-\left\{\frac{3}{2}\left(\log{t}-\frac{d-1}{2}\log{(\log{t})}+c_{\star}\right)-\log{3}+\frac{1}{2}\log{2} \right\}\omega_{\infty}+O\left(\frac{\log{(\log{t})}}{\log{t}}\right),
\\
z_2(t)&=z_{\infty}-\left\{\frac{1}{2}\left(\log{t}-\frac{d-1}{2}\log{(\log{t})}+c_{\star}\right)-\frac{1}{2}\log{2} \right\}\omega_{\infty}+O\left(\frac{\log{(\log{t})}}{\log{t}}\right),
\\
z_3(t)&=z_{\infty}+\left\{\frac{1}{2}\left(\log{t}-\frac{d-1}{2}\log{(\log{t})}+c_{\star}\right)-\frac{1}{2}\log{2} \right\}\omega_{\infty}+O\left(\frac{\log{(\log{t})}}{\log{t}}\right),
\\
z_4(t)&=z_{\infty}+\left\{\frac{3}{2}\left(\log{t}-\frac{d-1}{2}\log{(\log{t})}+c_{\star}\right)-\log{3}+\frac{1}{2}\log{2} \right\}\omega_{\infty}+O\left(\frac{\log{(\log{t})}}{\log{t}}\right).
\end{aligned}
\end{align}

\item There exist $\sigma=\pm 1$, $z_1,z_2,z_3,z_4:[0,\infty)\to\mathbb{R}^d$, $z_{\infty}\in\mathbb{R}^d$, an orthogonal pair $u_{\infty},v_{\infty}\in S^{d-1}$, and a constant $\varphi$ satisfying $\frac{\pi}{6}<\varphi<\frac{\pi}{3}$ such that
\begin{align}\label{4rhombustheorem1}
\left\| u(t)-\sigma \sum_{k=1}^4 (-1)^kQ(\cdot-z_k(t))\right\|_{H^1}+\|{\partial}_tu(t)\|_{L^2}\lesssim t^{-1},
\end{align}
\begin{align}\label{4rhombustheorem2}
\begin{aligned}
z_1(t)&=z_{\infty}+\cos{\varphi}\left\{\log{t}-\frac{d-1}{2}\log{(\log{t})}+c_{\star}+\log{2} \right\}u_{\infty}+O\left(\frac{\log{(\log{t})}}{\log{t}}\right),
\\
z_2(t)&=z_{\infty}+\sin{\varphi}\left\{\log{t}-\frac{d-1}{2}\log{(\log{t})}+c_{\star}+\log{2} \right\}v_{\infty}+O\left(\frac{\log{(\log{t})}}{\log{t}}\right),
\\
z_3(t)&=z_{\infty}-\cos{\varphi}\left\{\log{t}-\frac{d-1}{2}\log{(\log{t})}+c_{\star}+\log{2} \right\}u_{\infty}+O\left(\frac{\log{(\log{t})}}{\log{t}}\right),
\\
z_4(t)&=z_{\infty}-\sin{\varphi}\left\{\log{t}-\frac{d-1}{2}\log{(\log{t})}+c_{\star}+\log{2} \right\}v_{\infty}+O\left(\frac{\log{(\log{t})}}{\log{t}}\right).
\end{aligned}
\end{align}

\item There exist $\sigma=\pm 1$, $z_1,z_2,z_3,z_4:[0,\infty)\to\mathbb{R}^d$, $d_1,d_2:[0,\infty)\to\mathbb{R}$, $z_{\infty}\in\mathbb{R}^d$, and an orthogonal pair $u_{\infty},v_{\infty}\in S^{d-1}$ such that
\begin{align}\label{4rhombustheorem3}
\left\| u(t)-\sigma \sum_{k=1}^4 (-1)^kQ(\cdot-z_k(t))\right\|_{H^1}+\|{\partial}_tu(t)\|_{L^2}\lesssim t^{-1},
\end{align}
\begin{align}\label{4rhombustheorem4}
\begin{aligned}
z_1(t)&=z_{\infty}+d_1(t)u_{\infty}+o(1),
\\
z_2(t)&=z_{\infty}+d_2(t)v_{\infty}+o(1),
\\
z_3(t)&=z_{\infty}-d_1(t)u_{\infty}+o(1),
\\
z_4(t)&=z_{\infty}-d_2(t)v_{\infty}+o(1),
\\
d_1(t)&=\frac{1}{2}\log{t}-\frac{d-3}{4}\log{(\log{t})}-\frac{1}{2}\log{(\log{(\log{t})})}+\frac{c_{\star}}{2}+\frac{1}{2}\log{\frac{3}{2}}+o(1),
\\
d_2(t)&=\frac{\sqrt{3}}{2}\log{t}-\frac{(3d-1)\sqrt{3}}{12}\log{(\log{t})}+\frac{\sqrt{3}}{6}\log{(\log{(\log{t})})}+\frac{\sqrt{3}}{2}c_{\star}
\\
&\quad+\frac{\sqrt{3}}{6}\log{\frac{32}{3}}+o(1).
\end{aligned}
\end{align}

\end{enumerate}
\end{theorem}

\begin{remark}
In the interior rhombus regime, the ratio of the two diagonals satisfies
\begin{align*}
\lim_{t\to\infty}\frac{|z_2(t)-z_4(t)|}{|z_1(t)-z_3(t)|}=\tan{\varphi}.
\end{align*}
Thus, $\varphi=\frac{\pi}{4}$ corresponds to the square configuration, whereas $\varphi=\frac{\pi}{6}$ and $\varphi=\frac{\pi}{3}$ correspond to the two endpoint aspect ratios $\frac{1}{\sqrt3}$ and $\sqrt3$, respectively.
\end{remark}

\begin{remark}
The third alternative in Theorem \ref{maintheorem} describes both endpoint regimes of the rhombus family appearing in the second alternative. Indeed, after a cyclic relabeling of the centers and, if necessary, interchanging $u_{\infty}$ and $v_{\infty}$, the asymptotic formulas in the third alternative correspond to both endpoints.
\end{remark}

\begin{remark}
The constant $c_{\star}$ is universal: it depends only on $d$, $\alpha$, and $p$, and is independent of the particular multi-soliton solution. It is the constant term in the large-time expansion of the scalar reference equation built from the interaction function $\mathcal F$. Its precise definition is given in Subsection 2.2.
\end{remark}

\begin{remark}
C\^{o}te and Du \cite{CD} constructed $(2,2)$-sign multi-soliton solutions whose centers form an expanding square with alternating signs. Hence, the rhombus regime in Theorem \ref{maintheorem} is nonempty. It remains open whether there exist solutions realizing non-square rhombi or either of the endpoint regimes.
\end{remark}

\subsection{Main difficulties and strategy}

The classifications obtained in \cite{I2,I3} rely on relative position vectors based at a dynamically distinguished soliton, namely the unique soliton whose sign differs from all the others. No such distinguished center exists in the $(2,2)$-sign case. Consequently, the interaction system cannot be reduced from the outset to a family of radial variables measured from a fixed reference center.

Our first step is to show that the interactions between like-signed solitons are lower-order compared with the four dominant interactions between oppositely signed solitons. We then introduce a normalized moment of inertia for the center configuration and derive an almost-monotonicity formula whose leading dissipative term is the normalized Cauchy-Schwarz deficit associated with the centered positions and the interaction forces. The resulting weighted integrability allows us to select times along which this deficit tends to zero. Passing to the equality case in the Cauchy--Schwarz inequality then yields an algebraic force-balance relation for every subsequential limit of the normalized center configuration.

The resulting finite-dimensional rigidity problem has precisely two classes of solutions: alternating collinear configurations and rhombus configurations with alternating signs. Since these two classes are separated in the normalized configuration space, a continuity and topological separation argument upgrades the subsequential classification to a full-time classification.

A further difficulty is to rule out slow rotation of the entire configuration as it becomes asymptotically collinear or rhombic. To rule this out, we apply the effective center equations to the relative position vectors between the solitons, as in \cite{I3}, and exploit the fact that their directions change more slowly than their lengths to refine the long-time behavior of the centers. As a result, in the line regime, the centers approach a fixed line, whereas in the rhombus regime, the configuration becomes asymptotically coplanar and its two diagonal directions converge to a fixed orthogonal pair. In particular, neither a persistently rotating line nor a persistently rotating rhombus can occur.

Having established this directional rigidity, we analyze the remaining radial dynamics through separate effective systems for the line and rhombus regimes. The interior aspect ratios of the rhombus family lead to a nondegenerate asymptotic system, whereas the two endpoint ratios exhibit an additional cancellation. Resolving this degeneracy yields the $\log\log\log t$ correction in the endpoint expansions.

\subsection{Previous results}
The present work lies at the intersection of the soliton resolution conjecture and the study of strongly interacting multi-soliton dynamics. For the focusing energy-critical wave equation, soliton resolution has been established along sequences of times for general bounded solutions in dimensions three to five; see \cite{DJKM}. Full-time soliton resolution for radial solutions is now known in a broad range of dimensions; see, in particular, \cite{DKM2,CDKM,JL4}. Related full-time results have also been obtained for equivariant wave maps \cite{DKMM,JL3}, for radial energy-critical wave equations with a potential \cite{JLX,LMZ}, for radial energy-critical wave equations with damping \cite{GZ}, and for the radial energy-critical heat equation \cite{A1}.

More recently, the rigidity of geometric parameters within already-resolved multi-bubble regimes has also been investigated. Kim and Merle \cite{KM2} classified several non-radial multi-bubble regimes for the energy-critical nonlinear heat equation in dimensions $N\geq7$, in which the scales, centers, and signs of the bubbles enter the effective dynamics in a nontrivial manner. For the five-dimensional focusing energy-critical wave equation, Jendrej, Zhang, and Zhao \cite{JZZ} proved that, for spatially separated multi-bubble solutions with comparable scales, all scaling parameters are of order $t^{-2}$ and the renormalized modulation parameters converge to a finite-dimensional algebraic set determined by the limiting centers.

For the nonlinear Klein-Gordon equation without damping, the soliton resolution conjecture remains widely open. Nevertheless, Nakanishi and Schlag \cite{NS} obtained a complete classification of the global dynamics for the three-dimensional focusing cubic equation when the energy is only slightly above that of the ground state. Chen and Jendrej \cite{CJ} also established conditional asymptotic stability for well-separated multi-solitons and classified pure multi-soliton solutions. In the presence of damping, \eqref{energydecay} provides additional compactness and dissipation. Keller \cite{K1} constructed stable and unstable manifolds near stationary solutions, while Feireisl \cite{F} established an asymptotic decomposition along a sequence of times for bounded global solutions. Related subsequential decomposition results for non-radial bounded solutions, under suitable assumptions on the dimension and the nonlinearity, were obtained by Li and Zhao \cite{LZ}. Burq, Raugel, and Schlag \cite{BRS} obtained full-time convergence results for radial solutions. In one space dimension, C\^{o}te, Martel, and Yuan \cite{CMY} proved full-time soliton resolution without any size restriction and showed that the solitons occurring in the decomposition necessarily have alternating signs.

Beyond the existence of a soliton decomposition, one may ask which sign patterns and spatial configurations are compatible with the interaction dynamics. For the damped nonlinear Klein-Gordon equation, C\^{o}te, Martel, Yuan, and Zhao \cite{CMYZ} completely described two-soliton solutions associated with the ground state. They proved that the two solitons must have opposite signs, determined their universal logarithmic separation law, and constructed the corresponding family of solutions as a codimension-two Lipschitz manifold in the energy space. Subsequently, the author and Nakanishi \cite{IN} classified the global dynamics of all solutions starting sufficiently close to a superposition of two well-separated ground states with opposite signs. Results concerning solutions associated with excited stationary states can be found in \cite{CY}.

Multi-soliton solutions have also been constructed for several nonlinear dispersive equations. For the nonlinear Klein-Gordon equation without damping, C\^{o}te and Mu\~{n}oz \cite{CM} constructed multi-solitons based on the ground state, and C\^{o}te and Martel \cite{CM1} extended this construction to more general stationary profiles. These solutions are asymptotic to solitary waves traveling with pairwise distinct velocities of magnitude strictly less than one. Strongly interacting multi-solitons with logarithmic relative distances have been obtained for several models; see, for example, \cite{A,GI,MN,N1,N2}. In the damped setting considered here, the soliton velocities decay because of the damping, and the logarithmic separation is instead generated by the balance between the damping and the exponentially small interactions.

Most directly related to the present work is the construction of C\^{o}te and Du \cite{CD}. By imposing suitable symmetry assumptions, they constructed multi-soliton solutions for \eqref{DNKG} whose centers form expanding regular polygons, regular polyhedra, higher-dimensional regular polytopes, or collinear configurations. In particular, their result provides $(2,2)$-sign soliton solutions whose centers form an expanding square with alternating signs. They also proved that a multi-soliton solution cannot consist entirely of solitons having the same sign.

In contrast, we start with an arbitrary $(2,2)$-sign soliton solution and impose no symmetry, coplanarity, or a priori geometric condition on its centers. We prove that the interaction dynamics themselves force the centers into one of two asymptotic geometries: an alternating collinear configuration or an expanding rhombus with alternating signs. We further determine the precise long-time behavior in both regimes, including the distinct asymptotic law arising at the endpoint aspect ratios. 

\subsection{Notation}

Let $\{e_1,\ldots,e_d\}$ denote the canonical basis of $\mathbb{R}^d$, and let ${\partial}_k$ denote the partial derivative
with respect to $x_k$. For $x,y\in\mathbb{R}^d$, we denote the Euclidean inner product by
\begin{align}\label{Euclideanproduct}
x\cdot y=\sum_{k=1}^d x_k y_k.
\end{align}
For a point $x$ and nonempty sets $A$ and $B$ in a Euclidean space, we define
\begin{align*}
\operatorname{dist}(x,A)
&=\inf_{y\in A}|x-y|,
\\
\operatorname{dist}(A,B)
&=\inf_{\substack{x\in A\\y\in B}}|x-y|.
\end{align*}
Throughout this paper, the comparison notation is used as follows. For nonnegative quantities $X$ and $Y$, we write $X\lesssim Y$ if there exists a constant $C>0$ such that $X\leq CY$. We write $X\ll Y$ if $X\leq cY$ for a sufficiently small constant $c>0$. The precise smallness requirement will be clear from the context or specified when needed.

For real-valued quantities $X$ and $Y$ with $Y\neq0$, we write $X\sim Y$ if there exist constants $0<c\leq C$ such that
\begin{align}\label{signedcomparison}
c\leq\frac{X}{Y}\leq C.
\end{align}
Thus, $X\sim Y$ means that $X$ and $Y$ have the same sign and comparable absolute values; in particular, this notation is also used when both quantities are negative. For nonnegative $X$ and $Y$, this
is equivalent to $X\lesssim Y$ and $Y\lesssim X$. Unless otherwise stated, implicit constants are independent of the variables under consideration and may change from line to line.

\section{Preliminaries}
In this section, we collect several preliminary results on the long-time dynamics of $K$-soliton solutions. Most of the material is taken from \cite[Section 2 and Section 3]{I2} and is included here for the reader's
convenience.

\subsection{Basic properties of the ground state}

Since $Q$ is radial and positive, there exists $q:[0,\infty)\to (0,\infty)$ such that
\begin{align*}
Q(x)=q(|x|).
\end{align*}
Furthermore, $q$ satisfies
\begin{align}\label{qode2}
q^{\prime \prime}+\frac{d-1}{r}q^{\prime}-q+q^p=0.
\end{align}
In addition, there exists a constant $c_q>0$ depending only on $d$ and $p$ such that, for $r\geq 1$,
\begin{align}\label{qes}
|q(r)-c_qr^{-\frac{d-1}{2}}e^{-r}|+|q^{\prime}(r)+c_qr^{-\frac{d-1}{2}}e^{-r}|
\lesssim r^{-\frac{d+1}{2}}e^{-r}.
\end{align}
Next, in order to make the leading interaction term more explicit, we introduce the vector field
$H:\mathbb{R}^d\setminus\{0\}\to\mathbb{R}^d$ defined by
\begin{align}\label{Hdef}
H(z)=-\int_{\mathbb{R}^d}(\nabla Q^p(x))\,Q(x-z)\,dx.
\end{align}
Moreover, we define $g:(0,\infty)\to\mathbb{R}$ as
\begin{align*}
H(z)=\frac{z}{|z|}\,g(|z|).
\end{align*}
We note that this definition is well-defined. Furthermore, we have 
\begin{align*}
g(r+h)-g(r)
=-\int_{\mathbb{R}^d} {\partial}_1(Q^p)(x)\Bigl( Q\bigl(x-(r+h)e_1\bigr)-Q(x-re_1)\Bigr)\,dx.
\end{align*}
Therefore, we have 
\begin{align*}
g^{\prime}(r)
=\int_{\mathbb{R}^d} {\partial}_1(Q^p)(x)\,{\partial}_1Q(x-re_1)\,dx
=-\int_{\mathbb{R}^d}Q^p(x)\,{\partial}_1^2Q(x-re_1)\,dx.
\end{align*}
In addition, by direct computation,
\begin{align*}
{\partial}_1^2Q(x-re_1)
=\frac{\sum_{k=2}^dx_k^2}{|x-re_1|^3}q^{\prime}(|x-re_1|)
+\frac{(x_1-r)^2}{|x-re_1|^2}q^{\prime \prime}(|x-re_1|).
\end{align*}
Using \eqref{qode2} and \eqref{qes}, we obtain, for $r\geq1$,
\begin{align}\label{gnozenkin1}
|g(r)-c_gr^{-\frac{d-1}{2}}e^{-r}|+|g^{\prime}(r)+c_gr^{-\frac{d-1}{2}}e^{-r}|
\lesssim r^{-\frac{d+1}{2}}e^{-r},
\end{align}
where
\begin{align*}
c_g=c_q\int_{\mathbb{R}^d} Q^p(x)e^{-x_1}\,dx.
\end{align*}
Moreover, by \eqref{gnozenkin1}, we have for $R\gg 1$ and $|a|\lesssim 1$
\begin{align}\label{gnozenkin2}
 \left|g(R+a)-e^{-a}g(R)\right|\lesssim \frac{|a|g(R)}{R},
 \end{align}
 and we have for $R\gg 1$ and $0< a\ll R$
 \begin{align}\label{gnozenkin4}
\left|g(R+a)-e^{-a}g(R)\right|\lesssim \frac{(a+1)e^{-a}g(R)}{R},
\end{align}
 and we have for $|a|\ll 1$ and $1\ll R$
 \begin{align}\label{gnozenkin3}
\left|g(R+a)-(1-a)g(R)\right|\lesssim g(R)\left(a^2+\frac{|a|}{R}\right).
\end{align}
These estimates will be used below through the interaction function $\mathcal F$.

\subsection{Basic properties of the dynamics of $K$-soliton solutions}

We now summarize several results concerning the long-time dynamics of $K$-soliton solutions.
Let $\vec{u}$ be a $K$-soliton solution. Then there exist $\sigma\in\{-1,1\}^K$ and a function $z:[0,\infty)\to(\mathbb{R}^d)^K$ such that \eqref{Ksoli} holds.
For such $z$ and $\sigma$, we define
\begin{align*}
Q_k=\sigma_k Q(\cdot-z_k),\qquad
\vec{Q}_k=
\begin{pmatrix}
Q_k\\
0
\end{pmatrix},
\qquad
Q_{\sum}=\sum_{k=1}^KQ_k,\qquad
\vec{Q}_{\sum}=
\begin{pmatrix}
Q_{\sum}\\
0
\end{pmatrix}.
\end{align*}
Furthermore, for a solution $\vec{u}$ of \eqref{DNKG} and $z$ and $\sigma$, we define
\begin{align*}
\vec{\varepsilon}=
\begin{pmatrix}
\varepsilon\\
\eta
\end{pmatrix}
=\vec{u}-\vec{Q}_{\sum}.
\end{align*}
We also introduce the minimal separation
\begin{align*}
D(t)=\min_{i\neq j}|z_i(t)-z_j(t)|.
\end{align*}

In a multi-soliton decomposition, the choice of the center parameters $z_k(t)$ is not unique due to translation invariance.
Following \cite[Lemma 3.2]{I2}, we fix this freedom by imposing orthogonality conditions that eliminate the contribution of the neutral directions associated with translations of the linearized operator around each soliton.

\begin{lemma}\label{modKsoli}
Let $\vec{u}$ be a $K$-soliton solution. Then there exist $\sigma\in\{-1,1\}^K$ and a $C^1$ function
$z:[0,\infty)\to (\mathbb{R}^d)^K$ such that
\begin{align}
\lim_{t\to\infty}\Bigl(\left\| u(t)-\sum_{k=1}^K \sigma_kQ(\cdot-z_k(t))\right\|_{H^1}
+\|{\partial}_tu(t)\|_{L^2}\Bigr)&=0,\label{modlim1yowa}
\\
\lim_{t\to\infty} D(t)&=\infty, \label{modlim2yowa}
\end{align}
and for $1\leq l\leq d$, $1\leq i\leq K$, and $t\gg 1$,
\begin{align}\label{modeqyowa}
\int_{\mathbb{R}^d} \Bigl\{ {\partial}_tu(t)+2\alpha\Bigl(u(t)-\sum_{k=1}^K \sigma_k Q(\cdot-z_k(t))\Bigr)\Bigr\}
\, {\partial}_lQ(\cdot-z_i(t))=0.
\end{align}
\end{lemma}
\begin{proof}
See \cite[Lemma 3.2]{I2}.
\end{proof}
Throughout the remainder of this section, $\sigma$ and $z$ are chosen as in Lemma \ref{modKsoli}, and $Q_k$, $\vec{\varepsilon}$, and $D$ are defined using this same choice. We then define the scalar function $\mathcal{F}:[1,\infty)\to [0,\infty)$ by
\begin{align}\label{Fdef}
\mathcal{F}(r)=\frac{g(r)}{2\alpha\|{\partial}_1Q\|_{L^2}^2}.
\end{align}
Then, by \eqref{gnozenkin1}, there exists
\begin{align*}
c_{\mathcal{F}}=\frac{c_g}{2\alpha\|{\partial}_1Q\|_{L^2}^2}>0
\end{align*}
such that 
\begin{align}\label{mathcalFhyouka1}
\left|\mathcal{F}(r)-c_{\mathcal{F}}r^{-\frac{d-1}{2}}e^{-r}\right|+\left|\mathcal{F}^{\prime}(r)+c_{\mathcal{F}}r^{-\frac{d-1}{2}}e^{-r}\right|
\lesssim r^{-\frac{d+1}{2}}e^{-r}.
\end{align}
Furthermore, by \eqref{gnozenkin2}, \eqref{gnozenkin4}, and \eqref{gnozenkin3}, we have for $R\gg 1$ and $|a|\lesssim 1$
\begin{align}\label{tukaeruF1}
\left|\mathcal{F}(R+a)-e^{-a}\mathcal{F}(R)\right|\lesssim \frac{|a|\mathcal{F}(R)}{R},
\end{align}
and we have for $R\gg 1$ and $0<a\ll R$
\begin{align}\label{tukaeruF3}
\left|\mathcal{F}(R+a)-e^{-a}\mathcal{F}(R)\right|\lesssim \frac{(a+1)e^{-a}\mathcal{F}(R)}{R},
\end{align}
and we have for $R\gg1$ and $|a|\ll 1$
\begin{align}\label{tukaeruF2}
\left|\mathcal{F}(R+a)-(1-a)\mathcal{F}(R)\right|\lesssim \mathcal{F}(R)\left(a^2+\frac{|a|}{R}\right).
\end{align}
As shown in \cite[Section 3]{I2}, the function $\mathcal{F}$ describes the leading-order interaction force between
well-separated solitons.
\begin{lemma}\label{centerdynamics}
Let $\vec{u}$ be a $K$-soliton solution. Furthermore, let $\sigma\in\{-1,1\}^K$ and $z\in C^1([0,\infty);(\mathbb{R}^d)^K)$ satisfy \eqref{modlim1yowa}, \eqref{modlim2yowa}, and \eqref{modeqyowa}. Then, for every $1\leq k\leq K$ and $t\gg1$,
\begin{align}\label{zmodvol2}
\dot{z}_k
=-\sum_{\substack{1\leq i\leq K\\ i\neq k}}\sigma_i\sigma_k\mathcal{F}(|z_k-z_i|)
\frac{z_k-z_i}{|z_k-z_i|}
+o(\mathcal{F}(D)).
\end{align}
\end{lemma}
\begin{proof}
For this choice of $\sigma$ and $z$, by \cite[Lemma 3.3]{I2}, for $1<\theta<\min\{p-1,2\}$, we have
\begin{align}
\left|\dot{z}_k+\sum_{\substack{1\leq i\leq K\\ i\neq k}}\sigma_i\sigma_k\mathcal{F}(|z_k-z_i|)
\frac{z_k-z_i}{|z_k-z_i|}\right|
&\lesssim e^{-\theta D}+\|\vec{\varepsilon}\|_{\mathcal{H}}^2.\label{zes}
\end{align}
Moreover, \cite[Lemma 3.8]{I2} gives
\begin{align*}
\|\vec{\varepsilon}(t)\|_{\mathcal H}^2=o(\mathcal F(D(t))).
\end{align*}
Since we have $\theta>1$, \eqref{mathcalFhyouka1} also yields
\begin{align*}
e^{-\theta D}\lesssim \frac{\mathcal F(D)}{D}=o(\mathcal F(D)).
\end{align*}
Substituting these estimates into \eqref{zes}, we obtain \eqref{zmodvol2}.
\end{proof}

As is apparent from Lemma \ref{centerdynamics}, the interaction
function $\mathcal{F}$ determines the leading-order dynamics of the
soliton centers. We therefore introduce a scalar reference function
associated with $\mathcal{F}$.

Fix $L_0\gg1$ sufficiently large. We define $\Phi:[L_0,\infty)\to[0,\infty)$ as 
\begin{align}\label{Phidef}
\Phi(r)=\int_{L_0}^r\frac{ds}{\mathcal{F}(s)}.
\end{align}
By \eqref{mathcalFhyouka1}, we have
\begin{align}\label{Phiasymptotics}
\Phi(r)=\frac{1}{c_{\mathcal{F}}}r^{\frac{d-1}{2}} e^r\left(1+O\left(\frac{1}{r}\right)\right).
\end{align}
In particular, $\Phi$ is a strictly increasing bijection from $[L_0,\infty)$ onto $[0,\infty)$. For $t\geq0$, we define
\begin{align}\label{LdefPhi}
L(t)=\Phi^{-1}(t).
\end{align}
Equivalently, $L$ is the unique solution of
\begin{align}\label{Lnoode}
\begin{aligned}
L(0)&=L_0,
\\
\dot{L}&=\mathcal{F}(L).
\end{aligned}
\end{align}
Taking the logarithm of \eqref{Phiasymptotics}, we obtain, for $t\gg1$,
\begin{align}\label{Lnagasa}
L(t)=\log{t}-\frac{d-1}{2}\log{(\log{t})}+c_{\star}+O\left(\frac{\log{(\log{t})}}{\log{t}}\right),
\end{align}
where
\begin{align*}
c_{\star}=\log c_{\mathcal{F}}.
\end{align*}

\begin{remark}
The constant $c_{\star}$ in \eqref{Lnagasa} is precisely the constant appearing in Theorem \ref{maintheorem}. In particular, it depends only on $d$, $\alpha$, and $p$, and is independent of the particular multi-soliton solution.
\end{remark}

We next record two elementary comparison properties of the reference equation.

\begin{lemma}\label{mathcalFyouode}
Let $T\gg 1$ and let
\begin{align*}
\tilde{L}\in C^1([T,\infty);[L_0,\infty)).
\end{align*}

\begin{enumerate}
\item If $\tilde{L}$ satisfies
\begin{align}\label{Ltildecomparable}
\dot{\tilde{L}}\sim\mathcal{F}(\tilde{L}),
\end{align}
then, we have for $t\gg 1$
\begin{align}\label{LtildeLookisaonaji}
\left|\tilde{L}(t)-L(t)\right|\lesssim1.
\end{align}
\item If there exists a constant $c>0$ such that 
\begin{align}\label{Ltildecass}
\dot{\tilde{L}}=\left(c+o(1)\right)\mathcal{F}(\tilde{L}),
\end{align}
then we have
\begin{align}\label{LtildeLsaga0}
\lim_{t\to\infty}\left(\tilde{L}(t)-L(t)-\log{c}\right)=0.
\end{align}
\end{enumerate}
\end{lemma}

\begin{proof}
By \eqref{Phidef} and \eqref{Ltildecomparable}, we have
\begin{align}\label{tildeLode0822}
\frac{d}{dt}\Phi(\tilde{L}(t))=\frac{\dot{\tilde{L}}(t)}{\mathcal{F}(\tilde{L}(t))}\sim1.
\end{align}
By \eqref{tildeLode0822}, we have
\begin{align}\label{Phinosim}
\Phi(\tilde{L}(t))\sim t\sim \Phi(L(t)).
\end{align}
Moreover, by \eqref{Phiasymptotics}, we have
\begin{align}\label{logPhiasymptotics}
\log{\Phi(r)}=r+\frac{d-1}{2}\log{r}-\log{c_{\mathcal{F}}}+O\left(\frac{1}{r}\right).
\end{align}
Furthermore, by \eqref{mathcalFhyouka1}, \eqref{Phidef}, and \eqref{Phiasymptotics}, we have
\begin{align}\label{logPhiderivative}
\frac{d}{dr}\log{\Phi(r)}=\frac{1}{\mathcal{F}(r)\Phi(r)}=1+O\left(\frac{1}{r}\right).
\end{align}
By \eqref{Phinosim} and \eqref{logPhiderivative}, we have
\begin{align*}
\left|\log{\Phi(\tilde{L}(t))}-\log{\Phi(L(t))}\right|\lesssim1.
\end{align*}
Since $L(t),\tilde{L}(t)\to\infty$, by \eqref{logPhiasymptotics}, we obtain \eqref{LtildeLookisaonaji}.

Next, we prove \eqref{LtildeLsaga0}. By \eqref{Phidef} and \eqref{Ltildecass}, we have
\begin{align*}
\frac{d}{dt}\Phi(\tilde{L}(t))=\frac{\dot{\tilde{L}}(t)}{\mathcal{F}(\tilde{L}(t))}=c+o(1),
\end{align*}
and hence
\begin{align}\label{(2)notameLes}
\frac{\Phi(\tilde{L}(t))}{\Phi(L(t))}=\frac{ct+o(t)}{t}=c+o(1).
\end{align}
Since \eqref{Ltildecass} implies \eqref{Ltildecomparable}, we obtain \eqref{LtildeLookisaonaji}. Therefore, we obtain
\begin{align}\label{tildeL/L}
\lim_{t\to\infty}\frac{\tilde{L}(t)}{L(t)}=1.
\end{align}
Moreover, by \eqref{LtildeLookisaonaji} and \eqref{logPhiasymptotics}, we have
\begin{align}\label{L-tildeLganbaru}
\tilde{L}(t)-L(t)=\log{{\frac{\Phi(\tilde{L}(t))}{\Phi(L(t))}}}-\frac{d-1}{2}\left(\log{\tilde{L}(t)}-\log{L(t)}\right)+O\left(\frac{1}{L(t)}\right).
\end{align}
By \eqref{(2)notameLes}, \eqref{tildeL/L} and \eqref{L-tildeLganbaru}, we obtain
\begin{align*}
\lim_{t\to\infty}\left(\tilde{L}(t)-L(t)\right)=\lim_{t\to\infty}\log{\frac{\Phi(\tilde{L}(t))}{\Phi(L(t))}}=\log{c},
\end{align*}
which implies \eqref{LtildeLsaga0}.
\end{proof}

Moreover, by exploiting \eqref{energydecay} together with a Taylor expansion of the
energy around $\vec{Q}_{\sum}$, one obtains quantitative estimates both on the interaction terms and on the decay of the
remainder $\vec{\varepsilon}$. We define $V$ as 
\begin{align}\label{Vdef}
V=-\sum_{1\leq i<j\leq K}\sigma_i\sigma_j\mathcal{F}(|z_i-z_j|).
\end{align}
We next record a useful estimate of $V$.

\begin{lemma}\label{limeplem}
Let $\vec{u}$ be a $K$-soliton solution. Let 
\begin{align*}
\sigma\in\{-1,1\}^K\  \mbox{and}\  z\in C^1([0,\infty);(\mathbb{R}^d)^K)
\end{align*}
satisfy \eqref{modlim1yowa}, \eqref{modlim2yowa}, and \eqref{modeqyowa}. Furthermore, let $D$ and $V$ be defined using this same
choice of $\sigma$ and $z$. Then, we have
\begin{align}\label{Ves1}
\liminf_{t\to\infty} \frac{V(t)}{\mathcal{F}(D(t))}\geq 0.
\end{align}
Furthermore, we have for $t\gg 1$
\begin{align}
D(t)-\log{t}+\frac{d-1}{2}\log{(\log{t})}\lesssim1, \label{Dupper}
\\
\mathcal{F}(D)\gtrsim \frac{1}{t}. \label{F(D)lower}
\end{align}
\end{lemma}
\begin{proof}
For this choice of the modulation parameters, \eqref{Ves1} follows from \cite[Lemma 3.8]{I2}, while \eqref{Dupper} and \eqref{F(D)lower} follow from \cite[Lemma 3.9]{I2}.
\end{proof}
If, in addition, the interaction functional $V$ dominates $\mathcal{F}(D)$ in the sense that
\begin{align}\label{Vrep}
\liminf_{t\to \infty}\frac{V(t)}{\mathcal{F}(D(t))}> 0,
\end{align}
then one can further improve the control on the remainder and the modulation parameters by \cite[Lemma 3.11]{I2}.
\begin{lemma}\label{epzes}
Let $\vec{u}$ be a $K$-soliton solution and satisfy \eqref{modlim1yowa}, \eqref{modlim2yowa}, and \eqref{modeqyowa} for some $\sigma$ and a $C^1$ function $z$. Furthermore, we assume \eqref{Vrep}. Then there exists $C>0$ depending on $\vec{u}$ such that
\begin{align}\label{newepes}
\|\vec{\varepsilon}(t)\|_{\mathcal{H}}\leq C \mathcal{F}(D(t)).
\end{align}
Furthermore, for any $1<\theta<\min{(p-1,2)}$, there exists $\tilde{C}>0$ such that for any $1\leq k\leq K$,
\begin{align}\label{newzes}
\left|\dot{z}_k+\sum_{\substack{1\leq i\leq K\\ i\neq k}}\sigma_i\sigma_k\mathcal{F}(|z_k-z_i|)
\frac{z_k-z_i}{|z_k-z_i|}\right|
\leq \tilde{C} e^{-\theta D}.
\end{align}
\end{lemma}
\begin{proof}
This is a direct consequence of \cite[Lemma 3.11]{I2}.
\end{proof}

\section{Soliton interactions for $(2,2)$-sign soliton solutions}
In this section, we consider $(2,2)$-sign soliton solutions. First, to simplify the notation, we rewrite Lemma \ref{modKsoli} in the following form.
\begin{lemma}\label{Ksolikakikae}
Let $\vec{u}$ be a $(2,2)$-sign soliton solution. Then there exist $C^1$ functions $z_1,z_2,z_3,z_4:[0,\infty)\to \mathbb{R}^d$ such that 
\begin{align}
\lim_{t\to\infty}(\left\| u(t)-\sum_{k=1}^4 (-1)^kQ(\cdot-z_k(t))\right\|_{H^1}+\|{\partial}_tu(t)\|_{L^2})&=0,\label{modlim1}
\\
\lim_{t\to\infty}D(t)=\lim_{t\to\infty} \min_{1\leq i\neq j\leq 4}|z_i(t)-z_j(t)|&=\infty. \label{modlim2}
\end{align}
Furthermore, we have for $1\leq l\leq d$, $1\leq i\leq 4$, and $t\gg 1$,
\begin{align}\label{modeq}
\int_{\mathbb{R}^d} \{ {\partial}_tu(t)+2\alpha(u(t)-\sum_{k=1}^4 (-1)^kQ(\cdot-z_k(t)))\} {\partial}_lQ(\cdot-z_i(t))=0.
\end{align}
\end{lemma}
\begin{proof}
By Definition \ref{defm,nsoli}, a $(2,2)$-sign soliton solution is a
$4$-soliton solution with two positive signs and two negative signs.
Applying Lemma \ref{modKsoli} with $K=4$ and labeling the two positive
centers by $z_2, z_4$ and the two negative centers by $z_1, z_3$, gives
\eqref{modlim1}, \eqref{modlim2}, and \eqref{modeq}.
\end{proof}

\begin{remark}
In particular, by Lemma \ref{Ksolikakikae}, throughout the rest of this paper, whenever we assume that $\vec{u}$ is a
$(2,2)$-sign soliton solution, we choose $C^1$ functions $z_1,z_2,z_3,z_4$ so that \eqref{modlim1}, \eqref{modlim2},
and \eqref{modeq} hold. In particular, setting $\sigma_k=(-1)^k$, the pair $(\sigma,z)$ satisfies the
assumptions on the modulation parameters in Lemmas \ref{centerdynamics} and \ref{limeplem}. Whenever these lemmas
are used below, they are applied to this fixed choice of $\sigma$ and $z$.
\end{remark}

We now analyze the long-time dynamics of the centers provided by Lemma \ref{Ksolikakikae}. To exploit their cyclic structure, we regard the indices as elements of $\mathbb{Z}/4\mathbb{Z}$; thus, for instance, $z_5=z_1$ and $z_6=z_2$. Furthermore, we define for $j,k\in\mathbb{Z}/4\mathbb{Z}$,
\begin{align}\label{hendef}
Z_{jk}=z_k-z_j,\qquad
\rho_{jk}=|Z_{jk}|,\qquad
u_{jk}=\frac{Z_{jk}}{\rho_{jk}}.
\end{align}
We also introduce, for later use, the following quantities for $i\in\mathbb{Z}/4\mathbb{Z}$:
\begin{align}\label{iroirodef}
\rho_i=|z_{i+1}-z_i|,\qquad
v_i=\frac{z_{i+1}-z_i}{\rho_i},\qquad
R_i=|z_{i+2}-z_i|,\qquad
w_i=\frac{z_{i+2}-z_i}{R_i}.
\end{align}
With this notation, Lemma \ref{centerdynamics} can be reformulated as follows.
\begin{lemma}\label{2,2centerlem}
Let $\vec{u}$ be a $(2,2)$-sign soliton solution. Then the following statements hold:
\begin{enumerate}
\item For every $i\in \mathbb{Z}/4\mathbb{Z}$, we have
\begin{align}\label{2,2solicenterode}
\dot{z}_i&=-\mathcal{F}(\rho_i)v_i+\mathcal{F}(\rho_{i-1})v_{i-1}+\mathcal{F}(R_i)w_i+o\left(\mathcal{F}(D)\right).
\end{align}
\item For every $i\in \mathbb{Z}/4\mathbb{Z}$, we have
\begin{align}\label{saisho2,2nagasaode}
\begin{aligned}
\dot{\rho}_i&=2\mathcal{F}(\rho_i)-\mathcal{F}(\rho_{i+1})v_i\cdot v_{i+1}-\mathcal{F}(\rho_{i-1})v_i\cdot v_{i-1}
\\
&\quad-\mathcal{F}(R_i)v_i\cdot w_i+\mathcal{F}(R_{i+1})v_i\cdot w_{i+1}+o\left(\mathcal{F}(D)\right),
\\
\dot{R}_i&=-2\mathcal{F}(R_i)+\mathcal{F}(\rho_i)w_i\cdot v_i+\mathcal{F}(\rho_{i+1})w_i\cdot v_{i+1}
\\
&\quad -\mathcal{F}(\rho_{i+2})w_i\cdot v_{i+2}-\mathcal{F}(\rho_{i-1})w_i\cdot v_{i-1}+o\left(\mathcal{F}(D)\right).
\end{aligned}
\end{align}
\end{enumerate}
\end{lemma}
\begin{proof}
The centers fixed in Lemma \ref{Ksolikakikae}, together with $\sigma_i=(-1)^i$, satisfy the assumptions of Lemma
\ref{centerdynamics}. Therefore, applying Lemma \ref{centerdynamics} with $K=4$ gives \eqref{2,2solicenterode}. By \eqref{2,2solicenterode} and
\begin{align*}
\dot{\rho}_i=\left(\dot{z}_{i+1}-\dot{z}_i\right)\cdot v_i,\ \dot{R}_i=\left(\dot{z}_{i+2}-\dot{z}_i\right)\cdot w_i
\end{align*}
for $i\in \mathbb{Z}/4\mathbb{Z}$, we obtain \eqref{saisho2,2nagasaode}.
\end{proof}

\subsection{Estimates on the distances between like-signed solitons}

Our aim in this subsection is to prove the following proposition. First, we define $\rho$ and $R$ as 
\begin{align}\label{rhoRdef}
\rho=\min_{i=1,2,3,4}\rho_i,\ R=\min_{i=1,2}R_i.
\end{align}

By \eqref{mathcalFhyouka1}, after increasing the initial time if necessary, we may assume throughout this subsection that
$\mathcal{F}$ is positive and strictly decreasing at all distances under consideration.

\begin{proposition}\label{samesolitonhanare}
Let $\vec{u}$ be a $(2,2)$-sign soliton solution. Then, we have
\begin{align}\label{taikakusengomi}
\lim_{t\to\infty}\left(R(t)-\rho(t)\right)=\infty.
\end{align}
\end{proposition}

The proof requires a quantitative geometric estimate valid for arbitrary configurations of the soliton centers. We therefore begin with the following lemma.

\begin{lemma}\label{mathfraklem}
Let $\epsilon>0$ be sufficiently small, and suppose that $\mathfrak{x},\mathfrak{y},\mathfrak{z}\in\mathbb{R}^d$ satisfy
\begin{align*}
|\mathfrak{x}|=1,\qquad
\left||\mathfrak{y}|-1\right|\leq\frac{\epsilon}{100},\qquad
|\mathfrak{z}|\geq1-\frac{\epsilon}{100},
\\
|\mathfrak{x}-\mathfrak{y}|\geq1-\frac{\epsilon}{100},\qquad
|\mathfrak{y}-\mathfrak{z}|\geq1-\frac{\epsilon}{100},\qquad
|\mathfrak{z}-\mathfrak{x}|\geq1-\frac{\epsilon}{100}.
\end{align*}
Set
\begin{align*}
\hat{\mathfrak{x}}=\mathfrak{x},\qquad
\hat{\mathfrak{y}}=\frac{\mathfrak{y}}{|\mathfrak{y}|},\qquad
\hat{\mathfrak{z}}=\frac{\mathfrak{z}}{|\mathfrak{z}|}.
\end{align*}
Then we have
\begin{align}\label{mathfrakLem0}
\hat{\mathfrak{x}}\cdot\hat{\mathfrak{y}}\leq\frac{1}{2}+\frac{\epsilon}{10}.
\end{align}
Moreover, define $\mathfrak{A}$, $\mathfrak{B}$, and $\mathfrak{C}$ by
\begin{align*}
\mathfrak{A}
&=
\begin{cases}
\displaystyle
\max\left\{
\left(\hat{\mathfrak{y}}-\hat{\mathfrak{x}}\right)
\cdot\hat{\mathfrak{z}},0
\right\},
&\text{if }|\mathfrak{z}|\leq1+\dfrac{\epsilon}{100},
\\
\\
0,
&\text{if }|\mathfrak{z}|>1+\dfrac{\epsilon}{100},
\end{cases}
\\
\mathfrak{B}
&=
\begin{cases}
\displaystyle
\max\left\{
-\hat{\mathfrak{y}}\cdot
\frac{\mathfrak{z}-\mathfrak{y}}
{|\mathfrak{z}-\mathfrak{y}|},0
\right\},
&\text{if }|\mathfrak{y}-\mathfrak{z}|
\leq1+\dfrac{\epsilon}{100},
\\
0,
&\text{if }|\mathfrak{y}-\mathfrak{z}|
>1+\dfrac{\epsilon}{100},
\end{cases}
\\
\mathfrak{C}
&=
\begin{cases}
\displaystyle
\max\left\{
\hat{\mathfrak{x}}\cdot
\frac{\mathfrak{x}-\mathfrak{z}}
{|\mathfrak{x}-\mathfrak{z}|},0
\right\},
&\text{if }|\mathfrak{z}-\mathfrak{x}|
\leq1+\dfrac{\epsilon}{100},
\\
0,
&\text{if }|\mathfrak{z}-\mathfrak{x}|
>1+\dfrac{\epsilon}{100}.
\end{cases}
\end{align*}
Then we have
\begin{align}
\mathfrak{A}+\mathfrak{B}+\mathfrak{C}
&\leq
2-\hat{\mathfrak{x}}\cdot\hat{\mathfrak{y}}+\epsilon,
\label{mathfrakLem1}
\\
\mathfrak{B}
&\leq\frac{1}{2}+\epsilon,
\qquad
\mathfrak{C}\leq\frac{1}{2}+\epsilon.
\label{mathfrakLem2}
\end{align}
\end{lemma}
\begin{proof}
We may assume that $0<\epsilon\ll 1$. We first estimate
$\hat{\mathfrak{x}}\cdot\hat{\mathfrak{y}}$. Since we have
\begin{align*}
|\mathfrak{x}-\mathfrak{y}|^2
=1+|\mathfrak{y}|^2
-2|\mathfrak{y}|
\left(\hat{\mathfrak{x}}\cdot\hat{\mathfrak{y}}\right)
\geq
\left(1-\frac{\epsilon}{100}\right)^2,
\end{align*}
we have
\begin{align*}
\hat{\mathfrak{x}}\cdot\hat{\mathfrak{y}}
&\leq
\frac{|\mathfrak{y}|^2+\frac{\epsilon}{50}}
{2|\mathfrak{y}|}
\\
&\leq
\frac{\left(1+\frac{\epsilon}{100}\right)^2+\frac{\epsilon}{50}}
{2\left(1-\frac{\epsilon}{100}\right)}
\\
&\leq
\frac{1}{2}+\frac{\epsilon}{10},
\end{align*}
and we obtain \eqref{mathfrakLem0}. We next estimate $\mathfrak{B}$. By its definition, it suffices to consider the case
\begin{align*}
1-\frac{\epsilon}{100}\leq |\mathfrak{y}-\mathfrak{z}|\leq 1+\frac{\epsilon}{100}.
\end{align*}
In this case, the law of cosines yields
\begin{align*}
-\frac{\mathfrak{y}\cdot(\mathfrak{z}-\mathfrak{y})}
{|\mathfrak{y}||\mathfrak{y}-\mathfrak{z}|}
&=
\frac{\mathfrak{y}\cdot(\mathfrak{y}-\mathfrak{z})}
{|\mathfrak{y}||\mathfrak{y}-\mathfrak{z}|}
\\
&=
\frac{|\mathfrak{y}|^2+|\mathfrak{y}-\mathfrak{z}|^2
-|\mathfrak{z}|^2}
{2|\mathfrak{y}||\mathfrak{y}-\mathfrak{z}|}.
\end{align*}
Furthermore, we have
\begin{align*}
\frac{|\mathfrak{y}|^2+|\mathfrak{y}-\mathfrak{z}|^2
-|\mathfrak{z}|^2}
{2|\mathfrak{y}||\mathfrak{y}-\mathfrak{z}|}\leq \frac{2(1+\frac{\epsilon}{100})^2-(1-\frac{\epsilon}{100})^2}{2(1-\frac{\epsilon}{100})^2}\leq \frac{1}{2}+\frac{\epsilon}{20},
\end{align*}
and we obtain
\begin{align}\label{mathfrakBes}
\mathfrak{B}\leq \frac{2(1+\frac{\epsilon}{100})^2-(1-\frac{\epsilon}{100})^2}{2(1-\frac{\epsilon}{100})^2}\leq \frac{1}{2}+\frac{\epsilon}{20}.
\end{align}
We next estimate $\mathfrak{C}$. As in the estimate of $\mathfrak{B}$, it suffices to consider the case
\begin{align*}
|\mathfrak{z}-\mathfrak{x}|
\leq 1+\frac{\epsilon}{100}.
\end{align*}
Since $|\mathfrak{x}|=1$, the law of cosines gives
\begin{align*}
\frac{\mathfrak{x}\cdot(\mathfrak{x}-\mathfrak{z})}
{|\mathfrak{x}-\mathfrak{z}|}
=
\frac{1+|\mathfrak{x}-\mathfrak{z}|^2-|\mathfrak{z}|^2}
{2|\mathfrak{x}-\mathfrak{z}|}.
\end{align*}
Moreover,
\begin{align*}
\frac{1+|\mathfrak{x}-\mathfrak{z}|^2-|\mathfrak{z}|^2}
{2|\mathfrak{x}-\mathfrak{z}|}
&\leq
\frac{1+\left(1+\frac{\epsilon}{100}\right)^2
-\left(1-\frac{\epsilon}{100}\right)^2}
{2\left(1-\frac{\epsilon}{100}\right)}
\\
&\leq
\frac{1}{2}+\frac{\epsilon}{25}.
\end{align*}
Therefore,
\begin{align}\label{mathfrakCes}
\mathfrak{C}\leq \frac{1}{2}+\frac{\epsilon}{25}.
\end{align}
By \eqref{mathfrakBes} and \eqref{mathfrakCes}, we obtain \eqref{mathfrakLem2}. We now divide the proof of \eqref{mathfrakLem1} into four cases.

Case 1: $\mathfrak{A}=0$.

By \eqref{mathfrakLem0}, \eqref{mathfrakBes}, \eqref{mathfrakCes}, and the smallness of $\epsilon$, we obtain
\begin{align}\label{mathfrakCase1}
\begin{aligned}
\mathfrak{A}+\mathfrak{B}+\mathfrak{C}&\leq 1+\frac{9\epsilon}{100}
\\
&< 2-\left(\frac{1}{2}+\frac{\epsilon}{10}\right)+\epsilon
\\
&\leq 2-\hat{\mathfrak{x}}\cdot \hat{\mathfrak{y}}+\epsilon.
\end{aligned}
\end{align}

Case 2: $\mathfrak{A}\neq 0$ and $\mathfrak{C}\neq 0$.

By assumption, we have
\begin{align}\label{Case2assump}
1-\frac{\epsilon}{100}\leq |\mathfrak{z}|\leq 1+\frac{\epsilon}{100},\ 1-\frac{\epsilon}{100}\leq |\mathfrak{z}-\mathfrak{x}|\leq 1+\frac{\epsilon}{100}.
\end{align}
First, we estimate $\hat{\mathfrak{x}}\cdot \hat{\mathfrak{z}}$. By \eqref{Case2assump}, we have
\begin{align}\label{Case2es1}
\begin{aligned}
\hat{\mathfrak{x}}\cdot \hat{\mathfrak{z}}&=\frac{ \mathfrak{x}\cdot \mathfrak{z}}{|\mathfrak{z}|}
\\
&=\frac{|\mathfrak{x}|^2+|\mathfrak{z}|^2-|\mathfrak{z}-\mathfrak{x}|^2}{2|\mathfrak{z}|}
\\
&\geq \frac{1-\frac{\epsilon}{25}}{2(1+\frac{\epsilon}{100})}\geq \frac{1}{2}-\frac{3\epsilon}{100}.
\end{aligned}
\end{align}
Next, we estimate $\hat{\mathfrak{y}}\cdot \hat{\mathfrak{z}}$. Since we have
\begin{align*}
|\hat{\mathfrak{y}}-\hat{\mathfrak{z}}|&\geq |\mathfrak{y}-\mathfrak{z}|-|\mathfrak{y}-\hat{\mathfrak{y}}|-|\mathfrak{z}-\hat{\mathfrak{z}}|
\\
&=|\mathfrak{y}-\mathfrak{z}|-||\mathfrak{y}|-1|-||\mathfrak{z}|-1|
\\
&\geq 1-\frac{3\epsilon}{100},
\end{align*}
we obtain 
\begin{align}\label{Case2es2}
\hat{\mathfrak{y}}\cdot \hat{\mathfrak{z}}=1-\frac{|\hat{\mathfrak{y}}-\hat{\mathfrak{z}}|^2}{2}\leq \frac{1}{2}+\frac{3\epsilon}{100}.
\end{align}
By \eqref{Case2es1} and \eqref{Case2es2}, we obtain
\begin{align}\label{Case2es3}
\mathfrak{A}=\hat{\mathfrak{y}}\cdot \hat{\mathfrak{z}}-\hat{\mathfrak{x}}\cdot \hat{\mathfrak{z}}\leq \frac{3\epsilon}{50}.
\end{align}
By \eqref{mathfrakBes}, \eqref{mathfrakCes}, and \eqref{Case2es3}, we obtain
\begin{align}\label{mathfrakCase2}
\begin{aligned}
\mathfrak{A}+\mathfrak{B}+\mathfrak{C}&\leq 1+\frac{3\epsilon}{20}
\\
&< 2-\left(\frac{1}{2}+\frac{\epsilon}{10}\right)+\epsilon
\\
&\leq 2-\hat{\mathfrak{x}}\cdot \hat{\mathfrak{y}}+\epsilon.
\end{aligned}
\end{align}

Case 3: $\mathfrak{A}\neq 0$,\ $\mathfrak{B}\neq 0$,\ and $\mathfrak{C}=0$.

By assumption, we have
\begin{align}\label{Case3assu}
1-\frac{\epsilon}{100}\leq |\mathfrak{z}|\leq 1+\frac{\epsilon}{100},\ 1-\frac{\epsilon}{100}\leq |\mathfrak{z}-\mathfrak{y}|\leq 1+\frac{\epsilon}{100}.
\end{align}
Furthermore, since we have
\begin{align*}
\mathfrak{B}
&=
-\frac{\mathfrak{y}\cdot(\mathfrak{z}-\mathfrak{y})}
{|\mathfrak{y}||\mathfrak{z}-\mathfrak{y}|}
\\
&=
\frac{|\mathfrak{y}|^2-\mathfrak{y}\cdot\mathfrak{z}}
{|\mathfrak{y}||\mathfrak{z}-\mathfrak{y}|}
\\
&=
\frac{|\mathfrak{y}|
-|\mathfrak{z}|
\left(\hat{\mathfrak{y}}\cdot\hat{\mathfrak{z}}\right)}
{|\mathfrak{y}-\mathfrak{z}|},
\end{align*}
we obtain
\begin{align}\label{Case3es1}
\begin{aligned}
\left|
\mathfrak{B}
-\left(1-\hat{\mathfrak{y}}\cdot\hat{\mathfrak{z}}\right)
\right|
&=
\frac{
\left|
|\mathfrak{y}|
-|\mathfrak{z}|
\left(\hat{\mathfrak{y}}\cdot\hat{\mathfrak{z}}\right)
-\left(1-\hat{\mathfrak{y}}\cdot\hat{\mathfrak{z}}\right)
|\mathfrak{y}-\mathfrak{z}|
\right|
}
{|\mathfrak{y}-\mathfrak{z}|}
\\
&=
\frac{
\left|
\left(|\mathfrak{y}|-|\mathfrak{y}-\mathfrak{z}|\right)
+
\left(\hat{\mathfrak{y}}\cdot\hat{\mathfrak{z}}\right)
\left(|\mathfrak{y}-\mathfrak{z}|-|\mathfrak{z}|\right)
\right|
}
{|\mathfrak{y}-\mathfrak{z}|}
\\
&\leq
\frac{
\left||\mathfrak{y}|-|\mathfrak{y}-\mathfrak{z}|\right|
+
\left||\mathfrak{z}|-|\mathfrak{y}-\mathfrak{z}|\right|
}
{|\mathfrak{y}-\mathfrak{z}|}
\\
&\leq
\frac{\frac{\epsilon}{25}}
{1-\frac{\epsilon}{100}}
\leq
\frac{\epsilon}{20}.
\end{aligned}
\end{align}
On the other hand, the triangle inequality gives
\begin{align*}
\left|\hat{\mathfrak{y}}-\hat{\mathfrak{z}}\right|
&\leq
\left|\hat{\mathfrak{y}}-\mathfrak{y}\right|
+
\left|\mathfrak{y}-\mathfrak{z}\right|
+
\left|\hat{\mathfrak{z}}-\mathfrak{z}\right|
\\
&=
|\mathfrak{y}-\mathfrak{z}|
+
\left||\mathfrak{y}|-1\right|
+
\left||\mathfrak{z}|-1\right|
\\
&\leq
1+\frac{3\epsilon}{100}.
\end{align*}
Hence, we have
\begin{align}\label{Case3es2}
\begin{aligned}
\hat{\mathfrak{x}}\cdot\hat{\mathfrak{z}}
-\hat{\mathfrak{x}}\cdot\hat{\mathfrak{y}}
&=
\mathfrak{x}\cdot
\left(\hat{\mathfrak{z}}-\hat{\mathfrak{y}}\right)
\\
&\geq
-\left|\hat{\mathfrak{y}}-\hat{\mathfrak{z}}\right|
\\
&\geq
-1-\frac{3\epsilon}{100}.
\end{aligned}
\end{align}
Combining \eqref{Case3es1} and \eqref{Case3es2}, we obtain
\begin{align}\label{mathfrakCase3}
\begin{aligned}
\mathfrak{A}+\mathfrak{B}+\mathfrak{C}
&=
\mathfrak{A}+\mathfrak{B}
\\
&\leq
\left(\hat{\mathfrak{y}}-\hat{\mathfrak{x}}\right)
\cdot\hat{\mathfrak{z}}
+
\left(1-\hat{\mathfrak{y}}\cdot\hat{\mathfrak{z}}\right)
+\frac{\epsilon}{20}
\\
&=
1-\hat{\mathfrak{x}}\cdot\hat{\mathfrak{z}}
+\frac{\epsilon}{20}
\\
&<
2-\hat{\mathfrak{x}}\cdot\hat{\mathfrak{y}}
+\epsilon.
\end{aligned}
\end{align}

Case 4: $\mathfrak{A}\neq 0$,\ $\mathfrak{B}=\mathfrak{C}=0$.

Since we have $-1\leq \hat{\mathfrak{x}}\cdot \hat{\mathfrak{y}}\leq 1$, we obtain
\begin{align}\label{mathfrakCase4}
\mathfrak{A}+\mathfrak{B}+\mathfrak{C}=\mathfrak{A}=\left(\hat{\mathfrak{y}}-\hat{\mathfrak{x}}\right)\cdot \hat{\mathfrak{z}}\leq \left|\hat{\mathfrak{y}}-\hat{\mathfrak{x}}\right|=\sqrt{2-2\left(\hat{\mathfrak{x}}\cdot \hat{\mathfrak{y}}\right)}\leq 2-\hat{\mathfrak{x}}\cdot \hat{\mathfrak{y}}\leq 2-\hat{\mathfrak{x}}\cdot \hat{\mathfrak{y}}+\epsilon,
\end{align}
since we have for all $-1\leq x\leq 1$, $\sqrt{2-2x}\leq 2-x$.

Combining Cases 1--4, we obtain \eqref{mathfrakLem1}.
\end{proof}

We now use Lemma \ref{mathfraklem} to prove Proposition \ref{samesolitonhanare}. Before turning to its proof, we observe that Lemma \ref{limeplem} yields, for all sufficiently large $t$,
\begin{align}\label{rep2,2}
\begin{aligned}
V(t)&=
\mathcal{F}(\rho_1)+\mathcal{F}(\rho_2)+\mathcal{F}(\rho_3)+\mathcal{F}(\rho_4)-\mathcal{F}(R_1)-\mathcal{F}(R_2)
\\
&>-\frac{1}{100}\mathcal{F}(D).
\end{aligned}
\end{align}
We note that $D=\min{(\rho,R)}$. By \eqref{rep2,2}, we have for $t\gg 1$
\begin{align}\label{subsec3.2detukau1}
4\mathcal{F}(\rho)\geq \mathcal{F}(\rho_1)+\mathcal{F}(\rho_2)+\mathcal{F}(\rho_3)+\mathcal{F}(\rho_4)>\mathcal{F}(R)-\frac{1}{100}\mathcal{F}(D).
\end{align}
Therefore, setting $C_{\heartsuit}=\log{5}$, by \eqref{tukaeruF1}, we obtain for $t\gg 1$
\begin{align}\label{taikakues1}
R-\rho>-C_{\heartsuit}.
\end{align}
We next derive an estimate for $R-\rho$ that will be used in a bootstrap argument. To this end, we define the active index sets
\begin{align*}
\mathcal{A}_{\rho}(t)
&=
\left\{
i\in\mathbb{Z}/4\mathbb{Z}
\;\middle|\;
\rho_i(t)=\rho(t)
\right\},
\\
\mathcal{A}_{R}(t)
&=
\left\{
i\in\mathbb{Z}/4\mathbb{Z}
\;\middle|\;
R_i(t)=R(t)
\right\}.
\end{align*}

\begin{lemma}\label{|R|-rholem}
For every $C>0$, there exist constants $\kappa_C>0$ and $t_C>0$ such that the following holds. If $t>t_C$ and
\begin{align}\label{sameesass}
-C_{\heartsuit}
\leq R(t)-\rho(t)
\leq C,
\end{align}
then, for every $i,j\in\mathbb{Z}/4\mathbb{Z}$ satisfying
\begin{align*}
\rho_i(t)=\rho(t),
\qquad
R_j(t)=R(t),
\end{align*}
we have
\begin{align*}
\frac{d}{dt}\left(R_j(t)-\rho_i(t)\right)
\leq
-\kappa_C\mathcal{F}(D(t)).
\end{align*}
\end{lemma}
\begin{proof}
We argue by contradiction. Suppose that there exist sequences
\begin{align*}
\{t_n\}_{n\in\mathbb{N}}\subset\mathbb{R},
\qquad
\{i_n\}_{n\in\mathbb{N}},
\{j_n\}_{n\in\mathbb{N}}
\subset
\left(\mathbb{Z}/4\mathbb{Z}\right)^{\mathbb{N}}
\end{align*}
such that
\begin{align*}
&\lim_{n\to\infty}t_n=\infty,
\\
&j_n\in\mathcal{A}_R(t_n),\ i_n\in\mathcal{A}_{\rho}(t_n),
\\
&\frac{d}{dt}R_{j_n}(t_n)
=
\max_{j\in\mathcal{A}_R(t_n)}
\frac{d}{dt}R_j(t_n),\ \ \frac{d}{dt}\rho_{i_n}(t_n)
=
\min_{i\in\mathcal{A}_{\rho}(t_n)}
\frac{d}{dt}\rho_i(t_n),
\\
&-C_{\heartsuit}
\leq
R(t_n)-\rho(t_n)
\leq C,
\\
&\frac{d}{dt}
\left(R_{j_n}(t_n)-\rho_{i_n}(t_n)\right)
>
-\frac{1}{n}\mathcal{F}(D(t_n)).
\end{align*}
Since $i_n$ and $j_n$ take only finitely many values, after passing to a subsequence, we may fix $i,j\in\mathbb{Z}/4\mathbb{Z}$ such that
\begin{align}\label{rhoRlemsaishonokatei}
\begin{aligned}
&R_j(t_n)=R(t_n),\ \ \rho_i(t_n)=\rho(t_n),
\\
&\frac{d}{dt}R_j(t_n)
=
\max_{j'\in\mathcal{A}_R(t_n)}
\frac{d}{dt}R_{j'}(t_n),\ \ \frac{d}{dt}\rho_i(t_n)
=
\min_{i'\in\mathcal{A}_{\rho}(t_n)}
\frac{d}{dt}\rho_{i'}(t_n),
\\
&\left(\dot{R}_j(t_n)-\dot{\rho}_i(t_n)\right)
>
-\frac{1}{n}\mathcal{F}(D(t_n)).
\end{aligned}
\end{align}
Moreover, since we have
\begin{align*}
R_i=R_{i+2},\ R_{i+1}=R_{i+3},
\end{align*}
we may assume without loss of generality that either $j=i$ or $j=i+1$.

Case 1: $j=i$.

If $j=i$, by Lemma \ref{2,2centerlem} we have at $t=t_n$
\begin{align*}
\left(\dot{R}_i-\dot{\rho}_i\right)
&=
\mathcal{F}(\rho_i)(w_i\cdot v_i-2)
+\mathcal{F}(R_i)(w_i\cdot v_i-2)
\\
&\quad
+\mathcal{F}(\rho_{i+1})(v_i+w_i)\cdot v_{i+1}
+\mathcal{F}(\rho_{i-1})(v_i-w_i)\cdot v_{i-1}
\\
&\quad
-\mathcal{F}(\rho_{i+2})w_i\cdot v_{i+2}
-\mathcal{F}(R_{i+1})v_i\cdot w_{i+1}
+o\left(\mathcal{F}(D)\right).
\end{align*}
We first consider the case in which the sequence
\begin{align*}
\left\{
\frac{z_{i-1}(t_n)-z_i(t_n)}{\rho(t_n)}
\right\}_{n\in\mathbb{N}}
\end{align*}
is unbounded. After passing to a further subsequence, we may assume that
\begin{align*}
\lim_{n\to\infty}\left(R_{i+1}(t_n)-\rho(t_n)\right)&=\infty,
\\
\lim_{n\to\infty}\left(\rho_{i-1}(t_n)-\rho(t_n)\right)&=\infty,
\\
\lim_{n\to\infty}\left(\rho_{i+2}(t_n)-\rho(t_n)\right)&=\infty.
\end{align*}
Since we have
\begin{align*}
&\quad \lim_{n\to\infty}\left(v_i(t_n)+w_i(t_n)\right)\cdot v_{i+1}(t_n)
\\
&=\lim_{n\to\infty}\left(\frac{\left(R_i(t_n)-\rho_i(t_n)\right)\left(1+v_i(t_n)\cdot w_i(t_n)\right)}{\rho_{i+1}(t_n)}\right)=0,
\end{align*}
for all sufficiently large $n$, evaluation at $t=t_n$ yields
\begin{align*}
\frac{d}{dt}\left(R_i-\rho_i\right)&=\mathcal{F}(\rho_i)(w_i\cdot v_i-2)+\mathcal{F}(R_i)(w_i\cdot v_i-2)
\\
&\quad -\mathcal{F}(\rho_{i+2})w_i\cdot v_{i+2}+o\left(\mathcal{F}(D)\right)
\\
&\leq -\mathcal{F}(R)+o\left(\mathcal{F}(D)\right).
\end{align*}
By \eqref{tukaeruF1} and \eqref{sameesass}, we have
\begin{align*}
\mathcal{F}(R)\geq \frac{1}{2}e^{-C}\mathcal{F}(D).
\end{align*}
Therefore, we obtain $\left(\dot{R}_i-\dot{\rho}_i\right)\leq -\frac{1}{4}e^{-C}\mathcal{F}(D)$, which contradicts \eqref{rhoRlemsaishonokatei}.

Next, we assume that $\frac{z_{i-1}(t_n)-z_i(t_n)}{\rho(t_n)}$ is bounded. After passing to a subsequence, we may assume that the bounded sequences defined by
\begin{align*}
\mathfrak{x}_n&=v_i(t_n),
\\
\mathfrak{y}_n&=
\frac{z_{i+2}(t_n)-z_i(t_n)}{\rho(t_n)},
\\
\mathfrak{z}_n&=
\frac{z_{i-1}(t_n)-z_i(t_n)}{\rho(t_n)}
\end{align*}
converge to some $\mathfrak{x},\mathfrak{y},\mathfrak{z}\in\mathbb{R}^d$, respectively. For every $\epsilon>0$ and all sufficiently large $n$, we have
\begin{align*}
&|\mathfrak{x}_n|=1,\ \ \left||\mathfrak{y}_n|-1\right|
\leq\frac{\epsilon}{100},\ \ |\mathfrak{z}_n|\geq1,
\\
&|\mathfrak{x}_n-\mathfrak{y}_n|\geq1-\frac{\epsilon}{100},\ \ |\mathfrak{y}_n-\mathfrak{z}_n|\geq1-\frac{\epsilon}{100},\ \ |\mathfrak{z}_n-\mathfrak{x}_n|\geq1-\frac{\epsilon}{100}.
\end{align*}
Consequently, the limiting vectors
$\mathfrak{x},\mathfrak{y},\mathfrak{z}$
satisfy the assumptions of Lemma \ref{mathfraklem}. Keeping track of the orientations of the corresponding vectors, we also obtain, as $n\to\infty$,
\begin{align}\label{kyokugen6ko}
\begin{aligned}
v_i(t_n)
&\to\mathfrak{x},
&
w_i(t_n)
&\to\frac{\mathfrak{y}}{|\mathfrak{y}|},
&
v_{i-1}(t_n)
&\to-\frac{\mathfrak{z}}{|\mathfrak{z}|},
\\
v_{i+1}(t_n)
&\to
\frac{\mathfrak{y}-\mathfrak{x}}
{|\mathfrak{x}-\mathfrak{y}|},
&
v_{i+2}(t_n)
&\to
\frac{\mathfrak{z}-\mathfrak{y}}
{|\mathfrak{y}-\mathfrak{z}|},
&
w_{i+1}(t_n)
&\to
\frac{\mathfrak{z}-\mathfrak{x}}
{|\mathfrak{z}-\mathfrak{x}|}.
\end{aligned}
\end{align}
First, we estimate $\mathcal{F}(\rho_i)(w_i\cdot v_i-2)+\mathcal{F}(R_i)(w_i\cdot v_i-2)$. By \eqref{kyokugen6ko}, we have
\begin{align*}
\lim_{n\to\infty}(w_i(t_n)\cdot v_i(t_n)-2)=\frac{\mathfrak{x}\cdot \mathfrak{y}}{|\mathfrak{y}|}-2<0.
\end{align*}
By \eqref{tukaeruF1} and \eqref{sameesass}, we have
\begin{align*}
\mathcal{F}(\max{(\rho,R)})>m_C\mathcal{F}(D),
\end{align*}
where $m_C>0$ is a constant depending only on $C$. Therefore we have for large $n$
\begin{align*}
\mathcal{F}(\rho_i)(w_i\cdot v_i-2)+\mathcal{F}(R_i)(w_i\cdot v_i-2)<\left(\frac{\mathfrak{x}\cdot \mathfrak{y}}{|\mathfrak{y}|}-2+\epsilon\right)(1+m_C)\mathcal{F}(D).
\end{align*}
Since $\epsilon>0$ can be chosen arbitrarily small, we obtain, for all sufficiently large $n$,
\begin{align}\label{rhoiRies}
\mathcal{F}(\rho_i)(w_i\cdot v_i-2)+\mathcal{F}(R_i)(w_i\cdot v_i-2)<\left(\frac{\mathfrak{x}\cdot\mathfrak{y}}
{|\mathfrak{y}|}-2\right)\left(1+\frac{m_C}{2}\right)
\mathcal{F}(D).
\end{align}
Next, we estimate $\mathcal{F}(\rho_{i+1})(v_i+w_i)\cdot v_{i+1}$. Since we have
\begin{align*}
\lim_{n\to\infty} (v_i(t_n)+w_i(t_n))\cdot v_{i+1}(t_n)=\frac{\left(|\mathfrak{y}|\mathfrak{x}+\mathfrak{y}\right)\cdot \left(\mathfrak{y}-\mathfrak{x}\right)}{|\mathfrak{y}||\mathfrak{x}-\mathfrak{y}|}=\frac{\left(|\mathfrak{y}|-1\right)\left(|\mathfrak{y}|+\mathfrak{x}\cdot \mathfrak{y}\right)}{|\mathfrak{y}||\mathfrak{x}-\mathfrak{y}|},
\end{align*}
we have for large $n$
\begin{align*}
\left|(v_i(t_n)+w_i(t_n))\cdot v_{i+1}(t_n)\right|\leq 2\epsilon.
\end{align*}
Therefore we have for large $n$
\begin{align}\label{rhoi+1es}
\mathcal{F}(\rho_{i+1})(v_i+w_i)\cdot v_{i+1}<2\epsilon\mathcal{F}(D).
\end{align}
Next, we estimate $\mathcal{F}(\rho_{i-1})(v_i-w_i)\cdot v_{i-1}$. By \eqref{kyokugen6ko}, we have
\begin{align*}
\lim_{n\to\infty}(v_i(t_n)-w_i(t_n))\cdot v_{i-1}(t_n)=\left(\frac{\mathfrak{y}}{|\mathfrak{y}|}-\mathfrak{x}\right)\cdot \frac{\mathfrak{z}}{|\mathfrak{z}|}.
\end{align*}
We note that when we have $\frac{\rho_{i-1}}{\rho}>1+\frac{\epsilon}{200}$, we obtain $\mathcal{F}(\rho_{i-1})(v_i-w_i)\cdot v_{i-1}=o\left(\mathcal{F}(D)\right)$. Thus, we have for large $n$
\begin{align}\label{rhoi-1es}
\mathcal{F}(\rho_{i-1})(v_i-w_i)\cdot v_{i-1}< \left(\mathfrak{A}+\epsilon\right)\mathcal{F}(D).
\end{align}
Next, we estimate $ -\mathcal{F}(\rho_{i+2})w_i\cdot v_{i+2}$. By \eqref{kyokugen6ko}, we have
\begin{align*}
\lim_{n\to\infty}-w_i(t_n)\cdot v_{i+2}(t_n)=-\frac{\mathfrak{y}}{|\mathfrak{y}|}\cdot \frac{\mathfrak{z}-\mathfrak{y}}{|\mathfrak{z}-\mathfrak{y}|}.
\end{align*}
When we have $\frac{\rho_{i+2}}{\rho}>1+\frac{\epsilon}{200}$, we obtain $\mathcal{F}(\rho_{i+2})=o(\mathcal{F}(D))$. 
Therefore, we obtain for large $n$
\begin{align}\label{rhoi+2es}
-\mathcal{F}(\rho_{i+2})w_i\cdot v_{i+2}<\left(\mathfrak{B}+\epsilon\right)\mathcal{F}(D).
\end{align}
Next, we estimate $-\mathcal{F}(R_{i+1})v_i\cdot w_{i+1}$. By \eqref{kyokugen6ko}, we have
\begin{align*}
\lim_{n\to\infty}-v_i(t_n)\cdot w_{i+1}(t_n)=\mathfrak{x}\cdot \frac{\mathfrak{x}-\mathfrak{z}}{|\mathfrak{x}-\mathfrak{z}|}.
\end{align*}
When we have $\frac{R_{i+1}}{R}>1+\frac{\epsilon}{200}$, we obtain $\mathcal{F}(R_{i+1})=o(\mathcal{F}(D))$. Therefore, we obtain for large $n$
\begin{align}\label{Ri+1es}
-\mathcal{F}(R_{i+1})v_i\cdot w_{i+1}<\left(\mathfrak{C}+\epsilon\right)\mathcal{F}(D).
\end{align}
By \eqref{rhoiRies}--\eqref{Ri+1es} and Lemma \ref{mathfraklem}, for all sufficiently large $n$, we obtain at $t=t_n$,
\begin{align}\label{i=jcasees}
\begin{aligned}
\frac{d}{dt}\left(R_i-\rho_i\right)&\leq \left\{\left(\frac{\mathfrak{x}\cdot \mathfrak{y}}{|\mathfrak{y}|}-2\right)\left(1+\frac{m_C}{2}\right)+5\epsilon+\mathfrak{A}+\mathfrak{B}+\mathfrak{C}+o(1)\right\} \mathcal{F}(D).
\\
&<\left\{ \frac{m_C}{2}\left(\frac{\mathfrak{x}\cdot\mathfrak{y}}{|\mathfrak{y}|}-2\right)+6\epsilon+o(1)\right\}\mathcal{F}(D).
\end{aligned}
\end{align}
Since $\epsilon>0$ can be chosen arbitrarily small, it follows that, for all sufficiently large $n$,
\begin{align*}
\frac{d}{dt}\left(R_i-\rho_i\right)
<
-\frac{m_C}{4}\mathcal{F}(D).
\end{align*}
This contradicts \eqref{rhoRlemsaishonokatei}.

Case 2: $j=i+1$.

First, we define the transformed variables as follows:
\begin{align}\label{kaitenhanten}
\begin{aligned}
\overline{z}_i=z_{i+1},\ \overline{z}_{i+1}=z_i,\ \overline{z}_{i+2}=z_{i-1},\ \overline{z}_{i+3}=z_{i+2},
\\
\overline{v}_i=-v_i,\ \overline{v}_{i+1}=-v_{i-1},\ \overline{v}_{i+2}=-v_{i+2},\ \overline{v}_{i-1}=-v_{i+1},
\\
\overline{w}_i=w_{i+1},\ \overline{w}_{i+1}=w_i,
\\
\overline{\rho}_i=\rho_i,\ \overline{\rho}_{i+1}=\rho_{i-1},\ \overline{\rho}_{i+2}=\rho_{i+2},\ \overline{\rho}_{i-1}=\rho_{i+1},\\ 
\overline{R}_i=R_{i+1},\ \overline{R}_{i+1}=R_i.
\end{aligned}
\end{align}
Then, we have $\overline{\rho}_i=\rho_i$ and $\overline{R}_i=R_{i+1}$. Furthermore, by Lemma \ref{2,2centerlem}, we have
\begin{align*}
\dot{\overline{\rho}}_i=\dot{\rho}_i&=2\mathcal{F}(\rho_i)-\mathcal{F}(\rho_{i+1})v_i\cdot v_{i+1}-\mathcal{F}(\rho_{i-1})v_i\cdot v_{i-1}
\\
&\quad-\mathcal{F}(R_i)v_i\cdot w_i+\mathcal{F}(R_{i+1})v_i\cdot w_{i+1}+o\left(\mathcal{F}(D)\right)
\\
&=2\mathcal{F}(\overline{\rho}_i)-\mathcal{F}(\overline{\rho}_{i+1})\overline{v}_i\cdot \overline{v}_{i+1}-\mathcal{F}(\overline{\rho}_{i-1})\overline{v}_i\cdot \overline{v}_{i-1}
\\
&\quad-\mathcal{F}(\overline{R}_i)\overline{v}_i\cdot \overline{w}_i+\mathcal{F}(\overline{R}_{i+1})\overline{v}_i\cdot \overline{w}_{i+1}+o\left(\mathcal{F}(D)\right),
\\
\dot{\overline{R}}_i=\dot{R}_{i+1}&=-2\mathcal{F}(R_{i+1})+\mathcal{F}(\rho_{i+1})w_{i+1}\cdot v_{i+1}+\mathcal{F}(\rho_{i+2})w_{i+1}\cdot v_{i+2}
\\
&\quad -\mathcal{F}(\rho_{i-1})w_{i+1}\cdot v_{i-1}-\mathcal{F}(\rho_{i})w_{i+1}\cdot v_i+o\left(\mathcal{F}(D)\right)
\\
&=-2\mathcal{F}(\overline{R}_i)+\mathcal{F}(\overline{\rho}_i)\overline{w}_i\cdot \overline{v}_i+\mathcal{F}(\overline{\rho}_{i+1})\overline{w}_i\cdot \overline{v}_{i+1}
\\
&\quad-\mathcal{F}(\overline{\rho}_{i+2})\overline{w}_i\cdot \overline{v}_{i+2}-\mathcal{F}(\overline{\rho}_{i-1})\overline{w}_i\cdot \overline{v}_{i-1}+o\left(\mathcal{F}(D)\right).
\end{align*}
Therefore we have at $t=t_n$
\begin{align*}
\left(\dot{\overline{R}}_i-\dot{\overline{\rho}}_i\right)&=\mathcal{F}(\overline{\rho}_i)(\overline{w}_i\cdot \overline{v}_i-2)+\mathcal{F}(\overline{R}_i)(\overline{w}_i\cdot \overline{v}_i-2)
\\
&\quad +\mathcal{F}(\overline{\rho}_{i+1})(\overline{v}_i+\overline{w}_i)\cdot \overline{v}_{i+1}+\mathcal{F}(\overline{\rho}_{i-1})(\overline{v}_i-\overline{w}_i)\cdot \overline{v}_{i-1}
\\
&\quad -\mathcal{F}(\overline{\rho}_{i+2})\overline{w}_i\cdot \overline{v}_{i+2}-\mathcal{F}(\overline{R}_{i+1})\overline{v}_i\cdot \overline{w}_{i+1}+o\left(\mathcal{F}(D)\right).
\end{align*}
Moreover, the barred configuration has the same set of side lengths and the same set of diagonal lengths as the original configuration. In particular, at $t=t_n$, we have $\overline{\rho}_i=\rho,\ \overline{R}_i=R$. Hence all the assumptions used in Case 1 remain valid for the barred variables. Since the above expression has the same form as the corresponding expression in Case 1, the same argument yields a contradiction to \eqref{rhoRlemsaishonokatei}.

This completes the proof of Lemma \ref{|R|-rholem}.

\end{proof}

We now use this lemma to prove Proposition \ref{samesolitonhanare}.

\begin{proof}[Proof of Proposition \ref{samesolitonhanare}]
Fix an arbitrary constant $M>0$. First, we assume that there exists $t^{\prime}\gg 1$ such that 
\begin{align*}
-C_{\heartsuit}<R(t^{\prime})-\rho(t^{\prime})\leq M.
\end{align*}
Then, we introduce the following bootstrap estimate
\begin{align}\label{R-rhobs}
-C_{\heartsuit}<R(t)-\rho(t)\leq M.
\end{align}
We define $T_1\in [t^{\prime},\infty]$ by
\begin{align}\label{rho-RnoT1}
T_1=\sup{ \{t\in[t^{\prime},\infty)\ \mbox{such that}\ \eqref{R-rhobs}\ \mbox{holds on}\ [t^{\prime},t] \} }.
\end{align}
Furthermore, we assume $T_1<\infty$. Then, we have $R(T_1)-\rho(T_1)=M$. By Lemma \ref{|R|-rholem}, the estimate \eqref{R-rhobs} also holds for $0<t-T_1\ll 1$, which contradicts the maximality of $T_1$. Hence $T_1=\infty$. Then, by Lemma \ref{|R|-rholem}, we have for $t^{\prime}<T$
\begin{align*}
-\left(R(T)-\rho(T)\right)+\left(R(t^{\prime})-\rho(t^{\prime})\right)\gtrsim\int_{t^{\prime}}^T\mathcal{F}(D(s))ds\gtrsim \int_{t^{\prime}}^T\frac{ds}{s}.
\end{align*}
Since we have
\begin{align*}
\lim_{T\to\infty}\int_{t^{\prime}}^T\frac{ds}{s}=\infty,
\end{align*}
the preceding estimate implies that, for all sufficiently large $T$,
\begin{align*}
R(T)-\rho(T)<-C_{\heartsuit},
\end{align*}
which contradicts \eqref{R-rhobs}. Thus, we complete the proof.
\end{proof}
Proposition \ref{samesolitonhanare} yields \eqref{Vrep} and
$\mathcal{F}(R)=o\left(\mathcal{F}(D)\right)$.
We can therefore sharpen the center dynamics as follows.

\begin{lemma}\label{2,2centerdynamics2}
Let $\vec{u}$ be a $(2,2)$-sign soliton solution. Then, we have for $t\gg 1$
\begin{align}\label{rho=D}
\rho(t)=D(t),
\end{align}
and we have for every $i\in\mathbb{Z}/4\mathbb{Z}$ and every $1<\theta<\min{(p-1,2)}$,
\begin{align}
\dot{z}_i&=-\mathcal{F}(\rho_i)v_i+\mathcal{F}(\rho_{i-1})v_{i-1}+\mathcal{F}(R_i)w_i+O\left(e^{-\theta D}\right) \label{2,2centerrep}
\\
&=-\mathcal{F}(\rho_i)v_i+\mathcal{F}(\rho_{i-1})v_{i-1}+o\left(\mathcal{F}(D)\right). \label{2,2centertaikakunasi}
\end{align}
Furthermore, we have
\begin{align}\label{2,2henchousono2}
\begin{aligned}
\dot{\rho}_i&=2\mathcal{F}(\rho_i)-\mathcal{F}(\rho_{i+1})v_i\cdot v_{i+1}-\mathcal{F}(\rho_{i-1})v_i\cdot v_{i-1}
\\
&\quad-\mathcal{F}(R_i)v_i\cdot w_i+\mathcal{F}(R_{i+1})v_i\cdot w_{i+1}+O\left(e^{-\theta D}\right)
\\
&=2\mathcal{F}(\rho_i)-\mathcal{F}(\rho_{i+1})v_i\cdot v_{i+1}-\mathcal{F}(\rho_{i-1})v_i\cdot v_{i-1}+o\left(\mathcal{F}(D)\right).
\end{aligned}
\end{align}
\end{lemma}
\begin{proof}
By Proposition \ref{samesolitonhanare}, we have \eqref{Vrep} and \eqref{rho=D}. Therefore, by Lemma \ref{epzes} and Lemma \ref{Ksolikakikae}, we obtain \eqref{2,2centerrep}. \eqref{2,2centertaikakunasi} holds obviously by 
\begin{align*}
\mathcal{F}(R_i)+\mathcal{F}(R_{i+1})\lesssim \mathcal{F}(R)\ll \mathcal{F}(D).
\end{align*}
Similarly, \eqref{2,2henchousono2} follows from the same argument.
\end{proof}

\section{Classification of the asymptotic geometry of $(2,2)$-sign soliton solutions}

In this section, we classify the asymptotic geometry of $(2,2)$-sign soliton solutions. For $t\gg 1$, we define
\begin{align}\label{Ldef}
\mathcal{L}(t)=-\log{V(t)}.
\end{align}
We note that $\mathcal{L}$ is well-defined since \eqref{Vrep} holds for all sufficiently large $t$. Then, by Proposition \ref{samesolitonhanare}, we have for $t\gg 1$
\begin{align}\label{VsimmathcalF}
V(t)\sim \mathcal{F}(D(t)),
\end{align}
and we obtain
\begin{align}\label{mathcalLsim}
\left|\mathcal{L}(t)-D(t)\right|\lesssim \log{D(t)}\sim \log{\mathcal{L}(t)}.
\end{align}
Next, we define 
\begin{align}
z_g&=\frac{z_1+z_2+z_3+z_4}{4}, \label{jushindef}
\\
\mathcal{I}&=\sum_{i=1}^4\left|z_i-z_g\right|^2, \label{mathcalI}
\\
\mathcal{S}&=\sum_{i=1}^4\left|\dot{z}_i\right|^2\label{mathcalS},
\\
\mathcal{N}&=\sum_{i=1}^4\rho_i\mathcal{F}(\rho_i)-\sum_{i=1}^2R_i\mathcal{F}(R_i).\label{mathcalN}
\end{align}

\subsection{Equations for subsequential limits of the normalized center configuration}

We first derive differential identities for $V$, $\mathcal{L}$, and $\mathcal{I}$.

\begin{lemma}\label{mathcalbibunlem}
Let $\vec{u}$ be a $(2,2)$-sign soliton solution. Then, we have
\begin{align}
\dot{V}&=-\mathcal{S}+O\left(\frac{V^2}{D}\right),\label{Vbibun}
\\
\dot{\mathcal{L}}&=\frac{\mathcal{S}}{V}+O\left(\frac{V}{D}\right),\label{mathcalLbibun}
\\
\frac{1}{2}\dot{\mathcal{I}}&=\mathcal{N}+O\left(V+\frac{\mathcal{I}^{\frac{1}{2}}V}{D}\right).\label{mathcalIbibun}
\end{align}
\end{lemma}
\begin{proof}
First, we regard the collection of the centers as a configuration
\begin{align*}
\boldsymbol{z}
=(z_1,z_2,z_3,z_4)\in(\mathbb{R}^d)^4.
\end{align*}
For a differentiable function $V:(\mathbb{R}^d)^4\to\mathbb{R}$ and $i\in\mathbb{Z}/4\mathbb{Z}$, we denote by $D_iV(\boldsymbol{z})\in\mathbb{R}^d$ the derivative of $V$ with respect to its $i$-th block variable. More precisely, $D_iV(\boldsymbol{z})$ is characterized by
\begin{align}\label{DiVdef}
D_iV(\boldsymbol{z})\cdot h
=
\left.
\frac{d}{d\varepsilon}
V(z_1,\ldots,z_i+\varepsilon h,\ldots,z_4)
\right|_{\varepsilon=0}
\end{align}
for every $h\in\mathbb{R}^d$. Using \eqref{mathcalFhyouka1}, Proposition \ref{samesolitonhanare},
and Lemma \ref{2,2centerdynamics2}, we obtain
\begin{align}
\dot{z}_i=-D_iV+O\left(\frac{V}{D}\right).
\end{align}
Then, we have
\begin{align*}
\dot{V}=\sum_{i=1}^4D_iV\cdot \dot{z}_i=\sum_{i=1}^4 \left(-\dot{z}_i+O\left(\frac{V}{D}\right)\right)\cdot \dot{z}_i=-\mathcal{S}+O\left(\frac{V^2}{D}\right),
\end{align*}
and we obtain \eqref{Vbibun}. Next, we prove \eqref{mathcalLbibun}. We note that we have $V(t)\to 0$ and that $V(t)>0$ for all sufficiently large $t$. Then, by \eqref{Vbibun}, we have
\begin{align*}
\dot{\mathcal{L}}=-\frac{\dot{V}}{V}=\frac{\mathcal{S}}{V}+O\left(\frac{V}{D}\right),
\end{align*}
which proves \eqref{mathcalLbibun}. We next consider $\mathcal{I}$. First, by Lemma \ref{2,2centerdynamics2}, there exists $\theta>1$ such that 
\begin{align*}
\dot{z}_g
=\frac{\dot{z}_1+\dot{z}_2+\dot{z}_3+\dot{z}_4}{4}
=O\left(e^{-\theta D}\right).
\end{align*}
Then, we have
\begin{align}\label{mathcalIkeisan1}
\begin{aligned}
\frac{1}{2}\dot{\mathcal{I}}&=\sum_{i=1}^4\left(z_i-z_g\right)\cdot \left(\dot{z}_i+O\left(e^{-\theta D}\right)\right)
\\
&=\sum_{i=1}^4\left(z_i-z_g\right)\cdot\left(-D_iV+O\left(\frac{V}{D}\right)\right)
\\
&=\mathcal{N}_V+O\left(\frac{\mathcal{I}^{\frac{1}{2}}V}{D}\right),
\end{aligned}
\end{align}
where
\begin{align}\label{mathcalNVdef}
\mathcal{N}_V=-\sum_{i=1}^4\rho_i\mathcal{F}^{\prime}(\rho_i)+\sum_{i=1}^2R_i\mathcal{F}^{\prime}(R_i).
\end{align}
By \eqref{mathcalFhyouka1}, we have
\begin{align}\label{mathcalNnikansuru}
\begin{aligned}
&\quad \left|\mathcal{N}_V-\mathcal{N}\right|
\\
&\lesssim \sum_{i=1}^4\rho_i\left|\mathcal{F}^{\prime}(\rho_i)+\mathcal{F}(\rho_i)\right|+\sum_{i=1}^2R_i\left|\mathcal{F}^{\prime}(R_i)+\mathcal{F}(R_i)\right|
\\
&\lesssim \sum_{i=1}^4\mathcal{F}(\rho_i)+\sum_{i=1}^2\mathcal{F}(R_i)
\\
&\lesssim V.
\end{aligned}
\end{align}
By \eqref{mathcalIkeisan1} and \eqref{mathcalNnikansuru}, we have
\begin{align*}
\frac{1}{2}\dot{\mathcal{I}}=\mathcal{N}+O\left(V+\frac{\mathcal{I}^{\frac{1}{2}}V}{D}\right),
\end{align*}
and we obtain \eqref{mathcalIbibun}.
\end{proof}

Next, we establish two-sided estimates for $\mathcal{I}$, $\mathcal{S}$, and $\mathcal{N}$.

\begin{lemma}\label{mathcalichiyou}
Let $\vec{u}$ be a $(2,2)$-sign soliton solution. Then, we have
\begin{align}
\mathcal{N}=DV+O(V),\label{Nichiyou}
\\
\mathcal{S}\sim V^2,\label{Sichiyou}
\\
\mathcal{I}\sim \mathcal{L}^2.\label{Lichiyou}
\end{align}
\end{lemma}
\begin{proof}
First, we estimate $\mathcal{N}$. By \eqref{mathcalFhyouka1}, we have for $1\ll r^{\prime}\leq r$,
\begin{align}\label{N-DVnojunbi}
\left(r-r^{\prime}\right)\mathcal{F}(r)\lesssim \left(r-r^{\prime}\right)e^{-(r-r^{\prime})}\mathcal{F}(r^{\prime})\lesssim \mathcal{F}(r^{\prime}).
\end{align}
By \eqref{N-DVnojunbi}, we have for $t\gg 1$
\begin{align*}
\left|\mathcal{N}-DV\right|\leq\sum_{i=1}^4 (\rho_i-D)\mathcal{F}(\rho_i)+\sum_{i=1}^2\left(R_i-D\right)\mathcal{F}(R_i)\lesssim \mathcal{F}(D)\sim V,
\end{align*}
and we obtain \eqref{Nichiyou}.

Next, we estimate $\mathcal{S}$. First, we prove 
\begin{align}\label{Snotamehyouka}
\max_{i\in\mathbb{Z}/4\mathbb{Z}}\left|\mathcal{F}(\rho_{i-1})v_{i-1}-\mathcal{F}(\rho_i)v_i\right|\geq \frac{V}{10}.
\end{align}
We assume that for every $i\in\mathbb{Z}/4\mathbb{Z}$,
\begin{align}\label{V/10miman}
\left|\mathcal{F}(\rho_{i-1})v_{i-1}-\mathcal{F}(\rho_i)v_i\right|<\frac{V}{10}.
\end{align}
Since any two indices in $\mathbb{Z}/4\mathbb{Z}$ can be connected by at most two consecutive steps, \eqref{V/10miman} and the triangle inequality imply that, for every $i,j\in\mathbb{Z}/4\mathbb{Z}$,
\begin{align}\label{V/5miman}
\left|\mathcal{F}(\rho_{j})v_{j}-\mathcal{F}(\rho_i)v_i\right|<\frac{V}{5}.
\end{align}
Here, we choose $k\in\mathbb{Z}/4\mathbb{Z}$ such that $\rho_k=\rho$. Then, by Proposition \ref{samesolitonhanare}, we have
\begin{align}\label{V41/10miman}
V(t)=\sum_{i=1}^4\mathcal{F}(\rho_i)+o(V)\leq \frac{41}{10}\mathcal{F}(D).
\end{align}
Moreover, considering the four vectors $\{\mathcal{F}(\rho_i)v_i\}_{i\in\mathbb{Z}/4\mathbb{Z}}$, by \eqref{V/5miman} and \eqref{V41/10miman}, we have for every $i\in\mathbb{Z}/4\mathbb{Z}$,
\begin{align*}
\mathcal{F}(\rho_i)v_i\cdot v_k&=\left\{\mathcal{F}(\rho_i)v_i-\mathcal{F}(\rho_k)v_k+\mathcal{F}(\rho_k)v_k\right\}\cdot v_k
\\
&=\mathcal{F}(D)+v_k\cdot \left\{\mathcal{F}(\rho_i)v_i-\mathcal{F}(\rho_k)v_k\right\}
\\
&\geq \mathcal{F}(D)-\left|\mathcal{F}(\rho_i)v_i-\mathcal{F}(\rho_k)v_k\right|>\frac{10V}{41}-\frac{V}{5}=\frac{9V}{205}>0,
\end{align*}
and we obtain
\begin{align}\label{0gasei}
\sum_{i=1}^4\rho_iv_i\cdot v_k>0.
\end{align}
On the other hand, by \eqref{iroirodef}, we have
\begin{align*}
\sum_{i=1}^4\rho_iv_i\cdot v_k=\sum_{i=1}^4v_k\cdot \left(z_{i+1}-z_i\right)=0,
\end{align*}
which contradicts \eqref{0gasei}. Therefore we obtain \eqref{Snotamehyouka}. We also have
\begin{align*}
\max_{i\in\mathbb{Z}/4\mathbb{Z}}\left|\mathcal{F}(\rho_{i-1})v_{i-1}-\mathcal{F}(\rho_i)v_i\right|\lesssim \max_{i\in\mathbb{Z}/4\mathbb{Z}}\mathcal{F}(\rho_i)\lesssim V,
\end{align*}
and we obtain
\begin{align}\label{08011tukau1}
\sum_{i=1}^4\left|\mathcal{F}(\rho_{i-1})v_{i-1}-\mathcal{F}(\rho_i)v_i\right|^2\sim V^2.
\end{align}
By Lemma \ref{2,2centerdynamics2} and \eqref{08011tukau1}, we obtain \eqref{Sichiyou}.

Next, we estimate $\mathcal{I}$. First, we prove $\mathcal{I}\gtrsim D^2\sim \mathcal{L}^2$. We define
\begin{align*}
\mathcal{S}_V=\sum_{i=1}^4\left|D_iV\right|^2.
\end{align*}
Then, by the Cauchy--Schwarz inequality, we have
\begin{align}\label{cauchy1}
\mathcal{N}_V^2=\left|\sum_{i=1}^4\left(z_i-z_g\right)\cdot D_iV\right|^2\leq \mathcal{I}\mathcal{S}_V.
\end{align}
Furthermore, by Lemma \ref{2,2centerdynamics2} and \eqref{mathcalNnikansuru}, we have
\begin{align}\label{NVtoSVes}
\mathcal{N}_V=\mathcal{N}+O\left(V\right),\ \mathcal{S}_V=\mathcal{S}+O\left(\frac{V^2}{D}\right).
\end{align}
Thus, we obtain 
\begin{align}\label{cauchy2}
\left(\mathcal{N}+O(V)\right)^2\leq \mathcal{I}\left(\mathcal{S}+O\left(\frac{V^2}{D}\right)\right).
\end{align}
By \eqref{Nichiyou}, \eqref{Sichiyou}, and \eqref{cauchy2}, we obtain $\mathcal{I}\gtrsim D^2\sim \mathcal{L}^2$. Then, we rewrite \eqref{mathcalIbibun} as 
\begin{align}\label{Ibibun2}
\frac{1}{2}\dot{\mathcal{I}}&=\mathcal{N}+O\left(\frac{\mathcal{I}^{\frac{1}{2}}V}{D}\right).
\end{align}
Then, by \eqref{Sichiyou}, \eqref{cauchy2} and \eqref{Ibibun2}, we have
\begin{align*}
\frac{d}{dt}\left(\mathcal{I}^{\frac{1}{2}}\right)=\frac{\mathcal{N}}{\mathcal{I}^{\frac{1}{2}}}+O\left(\frac{V}{D}\right)\lesssim \sqrt{\mathcal{S}}\lesssim V\lesssim \dot{\mathcal{L}},
\end{align*}
which implies $\mathcal{I}\lesssim \mathcal{L}^2$. Therefore, we obtain \eqref{Lichiyou}.
\end{proof}

Next, we prove the following lemma.
\begin{lemma}\label{tenretulemma1}
Let $\vec{u}$ be a $(2,2)$-sign soliton solution. Then, for every sequence of times $t_n\to\infty$, there exists a sequence of times $t_n^{\prime}\to\infty$ such that
\begin{align}
\begin{aligned}\label{hodai0802}
\lim_{n\to\infty}
\log{\frac{\mathcal{L}(t_n^{\prime})}{\mathcal{L}(t_n)}}
&=0,
\\
\lim_{n\to\infty}
\frac{
\mathcal{I}(t_n^{\prime})\mathcal{S}_V(t_n^{\prime})
-\mathcal{N}_V^2(t_n^{\prime})
}{
\mathcal{L}^2(t_n^{\prime})V(t_n^{\prime})^2
}
&=0.
\end{aligned}
\end{align}
\end{lemma}

\begin{proof}
We define
\begin{align*}
\mathcal{P}=\frac{\mathcal{I}}{\mathcal{L}^2},\ \mathcal{Q}=\frac{\mathcal{S}_V}{V^2},\ \mathcal{R}=\frac{\mathcal{N}_V}{\mathcal{L}V}.
\end{align*}
First, Lemma \ref{mathcalbibunlem} and Lemma \ref{mathcalichiyou} imply that $\dot{\mathcal L}\sim V>0$. Moreover, we have
\begin{align}\label{Pode0811}
\dot{\mathcal{P}}&=\frac{\dot{\mathcal{I}}}{\mathcal{L}^2}-2\frac{\mathcal{I}\dot{\mathcal{L}}}{\mathcal{L}^3}=\frac{2\mathcal{N}_V}{\mathcal{L}^2}-2\frac{\mathcal{I}\dot{\mathcal{L}}}{\mathcal{L}^3}+O\left(\frac{V}{\mathcal{L}^2}\right).
\end{align}
For all sufficiently large $t$, multiplying both sides of \eqref{Pode0811} by $\frac{\mathcal{L}}{\dot{\mathcal{L}}}$, we obtain
\begin{align*}
\frac{\mathcal{L}}{\dot{\mathcal{L}}}\dot{\mathcal{P}}&=\frac{2\mathcal{N}_V}{\mathcal{L}\dot{\mathcal{L}}}-\frac{2\mathcal{I}}{\mathcal{L}^2}+O\left(\frac{1}{\mathcal{L}}\right)
\\
&=\frac{2\mathcal{N}_V}{\mathcal{L}\left(\frac{\mathcal{S}_V}{V}+O\left(\frac{V}{D}\right)\right)}-2\mathcal{P}+O\left(\frac{1}{\mathcal{L}}\right)
\\
&=\frac{2\mathcal{N}_VV}{\mathcal{L}\mathcal{S}_V}-2\mathcal{P}+O\left(\frac{1}{\mathcal{L}}\right)
\\
&=\frac{2\mathcal{R}}{\mathcal{Q}}-2\mathcal{P}+O\left(\frac{1}{\mathcal{L}}\right).
\end{align*}
Here, we define
\begin{align*}
\mathcal{J}=\mathcal{P}\mathcal{Q}-\mathcal{R}^2.
\end{align*}
Then, by \eqref{cauchy1}, we have $\mathcal{J}\geq 0$. Moreover, by direct computation, we have
\begin{align*}
\frac{\mathcal{L}}{\dot{\mathcal{L}}}\dot{\mathcal{P}}&=\frac{2\mathcal{R}}{\mathcal{Q}}-\frac{2\mathcal{R}^2+2\mathcal{J}}{\mathcal{Q}}+O\left(\frac{1}{\mathcal{L}}\right)
\\
&=\frac{2\mathcal{R}(1-\mathcal{R})}{\mathcal{Q}}-\frac{2\mathcal{J}}{\mathcal{Q}}+O\left(\frac{1}{\mathcal{L}}\right).
\end{align*}
We estimate $\mathcal{Q}$ and $\mathcal{R}$. By \eqref{mathcalLsim} and Lemma \ref{mathcalichiyou}, we have
\begin{align*}
&\mathcal{P}\sim 1,\ \mathcal{Q}\sim 1,
\\
&\mathcal{R}=\frac{\mathcal{N}_V}{\mathcal{L}V}=\frac{DV+O(V)}{\mathcal{L}V}=\frac{D}{\mathcal{L}}+O\left(\frac{1}{\mathcal{L}}\right)=1+O\left(\frac{\log{\mathcal{L}}}{\mathcal{L}}\right).
\end{align*}
Therefore, we obtain
\begin{align*}
\frac{\mathcal{L}}{\dot{\mathcal{L}}}\dot{\mathcal{P}}=-\frac{2\mathcal{J}}{\mathcal{Q}}+O\left(\frac{\log{\mathcal{L}}}{\mathcal{L}}\right).
\end{align*}
Thus, there exists $c>0$ such that, for all sufficiently large $t$,
\begin{align}\label{bibunhutoushiki}
\dot{\mathcal{P}}\leq -c\mathcal{J}\frac{\dot{\mathcal{L}}}{\mathcal{L}}+\frac{1}{c}\frac{\log{\mathcal{L}}}{\mathcal{L}^2}\dot{\mathcal{L}}.
\end{align}
Integrating \eqref{bibunhutoushiki} on $[t,T]$, where $1\ll t<T$, we have
\begin{align*}
\mathcal{P}(T)-\mathcal{P}(t)\leq -c\int_t^T\mathcal{J}(s)\frac{\dot{\mathcal{L}}(s)ds}{\mathcal{L}(s)}+\frac{1}{c}\left\{\left(\frac{\log{\mathcal{L}(t)}}{\mathcal{L}(t)}+\frac{1}{\mathcal{L}(t)}\right)-\left(\frac{\log{\mathcal{L}(T)}}{\mathcal{L}(T)}+\frac{1}{\mathcal{L}(T)}\right)\right\},
\end{align*}
and we obtain
\begin{align*}
\int_t^{\infty}\mathcal{J}(s)\frac{\dot{\mathcal{L}}(s)ds}{\mathcal{L}(s)}<\infty.
\end{align*}
Here, for any sequence of times $t_n\to\infty$, we define $\delta_{t_n}$ as 
\begin{align*}
\delta_{t_n}=\int_{t_n}^{\infty}\mathcal{J}(s)\frac{\dot{\mathcal{L}}(s)ds}{\mathcal{L}(s)}.
\end{align*}
Then $\delta_{t_n}\to 0$ as $n\to\infty$. For all sufficiently large $n$, we define $t_n^{\prime}$ as follows.
\begin{itemize}
\item When $\delta_{t_n}=0$, the continuity of $\mathcal{J}$ and the positivity of $\dot{\mathcal{L}}/\mathcal{L}$ imply that $\mathcal{J}(s)=0$ for every $s\geq t_n$. We therefore set $t_n^{\prime}=t_n$.

\item We consider $\delta_{t_n}>0$. Since we have
\begin{align*}
\int_{t_n}^{\infty}\frac{\dot{\mathcal{L}}(s)ds}{\mathcal{L}(s)}=\lim_{t\to\infty}\left(\log{\mathcal{L}}(t)-\log{\mathcal{L}(t_n)}\right)=\infty,
\end{align*}
there exists $T_n>t_n$ such that 
\begin{align*}
\log{\frac{\mathcal{L}(T_n)}{\mathcal{L}(t_n)}}=\int_{t_n}^{T_n}\frac{\dot{\mathcal{L}}(s)ds}{\mathcal{L}(s)}=\sqrt{\delta_{t_n}}.
\end{align*}
Then, we note that by the definition of $\delta_{t_n}$, we have
\begin{align*}
\int_{t_n}^{T_n}\mathcal{J}(s)\frac{\dot{\mathcal{L}}(s)ds}{\mathcal{L}(s)}\leq \delta_{t_n}.
\end{align*}
Then, by the weighted mean value theorem for integrals, there exists $t^{\prime}\in [t_n,T_n]$ such that 
\begin{align*}
\mathcal{J}(t^{\prime})\leq \sqrt{\delta_{t_n}}.
\end{align*}
Then, we define $t_n^{\prime}=t^{\prime}$.
\end{itemize}
By the choice of $t_n^{\prime}$, we obtain \eqref{hodai0802}. Therefore, we complete the proof.
\end{proof}

The estimates for $V$ and $\mathcal{S}$ also yield an estimate for $\mathcal{F}(D)$.

\begin{lemma}\label{Fsimt-1lem}
Let $\vec{u}$ be a $(2,2)$-sign soliton solution. Then, we have
\begin{align}\label{mathcalFsimt-1}
\begin{aligned}
\mathcal{F}(D(t))\sim \frac{1}{t},
\\
\left|D-\log{t}+\frac{d-1}{2}\log{(\log{t})}\right|\lesssim 1.
\end{aligned}
\end{align}
\end{lemma}
\begin{proof}
By \eqref{Vbibun} and \eqref{Sichiyou}, we have
\begin{align*}
\dot{V}\sim -V^2.
\end{align*}
Therefore, we have
\begin{align*}
\frac{d}{dt}\left(\frac{1}{V(t)}\right)\sim 1,
\end{align*}
and we obtain 
\begin{align}\label{Vsimt-1}
V(t)\sim \frac{1}{t}.
\end{align}
By \eqref{VsimmathcalF} and \eqref{Vsimt-1}, we obtain $\mathcal{F}(D)\sim t^{-1}$. Furthermore, by the form of $\mathcal{F}$, we have $|D-\log{t}+\frac{d-1}{2}\log{(\log{t})}|\lesssim 1$, and hence we obtain \eqref{mathcalFsimt-1}.
\end{proof}
Using Lemma \ref{Fsimt-1lem}, we now prove that $z_g$ converges.
\begin{lemma}\label{zgshusoku}
Let $\vec{u}$ be a $(2,2)$-sign soliton solution. Then, there exists $z_{\infty}\in \mathbb{R}^d$ such that for any $0<\theta<\min{(p-2,1)}$,
\begin{align}\label{jushines}
|z_g(t)-z_{\infty}|\lesssim t^{-\theta}.
\end{align}
\end{lemma}
\begin{proof}
Fix $0<\theta<\min{(p-2,1)}$ and choose $\theta^{\prime}$ such that $1+\theta<\theta^{\prime}<\min{(2,p-1)}$. Then, by Lemma \ref{2,2centerdynamics2}, we have
\begin{align*}
|\dot{z}_g(t)|\lesssim e^{-\theta^{\prime}D}\lesssim \mathcal{F}(D)^{1+\theta}\sim t^{-\theta-1}.
\end{align*}
Therefore, there exists $z_{\infty}\in \mathbb{R}^d$ such that 
\begin{align*}
|z_g(t)-z_{\infty}|\lesssim \int_t^{\infty} s^{-\theta-1}ds\sim t^{-\theta},
\end{align*}
and we complete the proof.
\end{proof}

Here, we prove a subsequential rigidity result for $(2,2)$-sign soliton solutions.

\begin{lemma}\label{tenretujouken}
Let $\vec{u}$ be a $(2,2)$-sign soliton solution. Then, for every sequence of times $t_n\to\infty$, there exists a subsequence $(t_n^{\prime})$ of $(t_n)$ such that the following properties hold.
\begin{enumerate}
\item There exist $Y_1,Y_2,Y_3,Y_4\in\mathbb{R}^d$ such that, for every $i\in\mathbb{Z}/4\mathbb{Z}$,
\begin{align}\label{abouttenretu1}
\lim_{n\to\infty}\left\{\frac{1}{\mathcal{L}(t_n^{\prime})}\left(z_i(t_n^{\prime})-z_g(t_n^{\prime})\right)\right\}=Y_i.
\end{align}

\item There exist $\lambda>0$ and $\alpha_i\geq 0$ for $i\in\mathbb{Z}/4\mathbb{Z}$ such that
\begin{align}
\lambda Y_i=-\alpha_i\left(Y_{i+1}-Y_i\right)+\alpha_{i-1}\left(Y_i-Y_{i-1}\right),\label{odekarakuru0808}
\\
\sum_{i=1}^4Y_i=0,\ \sum_{i=1}^4\alpha_i=1,\label{wateisuu0808}
\\
\left|Y_i-Y_j\right|\geq 1\ (i\neq j).\label{nagasa1ijou0808}
\end{align}

\item For every $i\in\mathbb{Z}/4\mathbb{Z}$ such that $\alpha_i>0$, we have
\begin{align*}
|Y_i-Y_{i+1}|=1.
\end{align*}
\end{enumerate}
\end{lemma}

\begin{proof}
First, we note that by \eqref{Lichiyou}, we have for every $i\in\mathbb{Z}/4\mathbb{Z}$,
\begin{align}\label{zi-zglesssimL}
\left|z_i(t)-z_g(t)\right|\lesssim \mathcal{I}^{\frac{1}{2}}(t)\lesssim \mathcal{L}(t).
\end{align}
We define
\begin{align}
\mathtt{Y}_i(t)&=\frac{1}{\mathcal{L}(t)}\left(z_i(t)-z_g(t)\right),
\\
\mathtt{W}_i(t)&=\frac{-D_iV(t)}{V(t)}.
\end{align}
Then, by \eqref{Lichiyou}, \eqref{NVtoSVes}, and \eqref{zi-zglesssimL}, the functions $\mathtt{Y}_i(t)$ and $\mathtt{W}_i(t)$ are bounded. Therefore, we may assume that there exist $(t_n^{\prime})\subset (t_n)$, $Y_i$ and $W_i$ such that 
\begin{align*}
\lim_{n\to\infty} \mathtt{Y}_i(t_n^{\prime})&=Y_i,
\\
\lim_{n\to\infty}\mathtt{W}_i(t_n^{\prime})&=W_i.
\end{align*}
Therefore, by Lemma \ref{tenretulemma1}, there exists a sequence $s_n\to\infty$ such that 
\begin{align*}
\lim_{n\to\infty}\log{\frac{\mathcal{L}(s_n)}{\mathcal{L}(t_n^{\prime})}}=0,
\\
\lim_{n\to\infty}\frac{\mathcal{I}(s_n)\mathcal{S}_V(s_n)-\mathcal{N}_V^2(s_n)}{\mathcal{L}^2(s_n)V^2(s_n)}=0.
\end{align*}
We divide the proof into four steps.

Step 1: We prove $\lim_{n\to\infty} \left|\mathtt{Y}_i(t_n^{\prime})-\mathtt{Y}_i(s_n)\right|=0$. 

By Lemma \ref{2,2centerdynamics2} and direct computation, we have for all sufficiently large $t$
\begin{align*}
\dot{\mathtt{Y}}_i&=\frac{1}{\mathcal{L}}\left(\dot{z}_i-\dot{z}_g\right)-\frac{\dot{\mathcal{L}}}{\mathcal{L}}\mathtt{Y}_i
\\
&=\frac{\dot{\mathcal{L}}}{\mathcal{L}}\left(\frac{\dot{z}_i}{\dot{\mathcal{L}}}-\mathtt{Y}_i\right)+O\left(\frac{V}{\mathcal{L}}\right)
\\
&=\frac{\dot{\mathcal{L}}}{\mathcal{L}}\left(\frac{V^2}{\mathcal{S}}\mathtt{W}_i-\mathtt{Y}_i \right)+O\left(\frac{V}{\mathcal{L}}\right)=O\left(\frac{\dot{\mathcal{L}}}{\mathcal{L}}\right).
\end{align*}
%Since $\mathcal{L}$ is increasing for all sufficiently large $t$, for sufficiently large $n$, we define 
%\begin{align*}
%\mathfrak{s}_n=\log{\mathcal{L}(s_n)},\ \mathfrak{t}_n=\log{\mathcal{L}(t_n^{\prime})}.
%\end{align*}
Then, we have
\begin{align*}
\left|\mathtt{Y}_i(t_n^{\prime})-\mathtt{Y}_i(s_n)\right|&\lesssim  \left|\int_{s_n}^{t_n^{\prime}}\frac{\dot{\mathcal{L}}(t)dt}{\mathcal{L}(t)} \right|=\left|\log{\frac{\mathcal{L}(s_n)}{\mathcal{L}(t_n^{\prime})}}\right|,
\end{align*}
and we obtain $\lim_{n\to\infty}\left|\mathtt{Y}_i(t_n^{\prime})-\mathtt{Y}_i(s_n)\right|=0$. In particular, we obtain
\begin{align}\label{Ynokyokugen0808}
\lim_{n\to\infty}\mathtt{Y}_i(s_n)=Y_i.
\end{align}

Step 2: We prove that there exists $\lambda>0$ such that $\lim_{n\to\infty}\mathtt{W}_i(s_n)=\lambda Y_i$.

By assumption, we have
\begin{align*}
&\quad \lim_{n\to\infty}\left\{\left(\sum_{i=1}^4\left|\mathtt{Y}_i(s_n)\right|^2\right)\left(\sum_{i=1}^4\left|\mathtt{W}_i(s_n)\right|^2\right)-\left(\sum_{i=1}^4 \mathtt{Y}_i(s_n)\cdot\mathtt{W}_i(s_n)\right)^2\right\}
\\
&=\lim_{n\to\infty}\frac{\mathcal{I}(s_n)\mathcal{S}_V(s_n)-\mathcal{N}_V^2(s_n)}{\mathcal{L}^2(s_n)V^2(s_n)}=0.
\end{align*}
We define
\begin{align*}
\lambda_n=\frac{\sum_{i=1}^4\mathtt{Y}_i(s_n)\cdot \mathtt{W}_i(s_n)}{\sum_{i=1}^4\left|\mathtt{Y}_i(s_n)\right|^2}.
\end{align*}
Since we have
\begin{align*}
\lim_{n\to\infty}\left(\sum_{i=1}^4\mathtt{Y}_i(s_n)\cdot \mathtt{W}_i(s_n)\right)=\lim_{n\to\infty}\frac{\mathcal{N}_V(s_n)}{\mathcal{L}(s_n)V(s_n)}=1,
\end{align*}
there exists $\lambda>0$ such that 
\begin{align*}
\lim_{n\to\infty}\lambda_n=\lambda.
\end{align*}
By direct computation, we have
\begin{align*}
\sum_{i=1}^4\left|\mathtt{W}_i(s_n)-\lambda_n\mathtt{Y}_i(s_n)\right|^2&=\sum_{i=1}^4\left|\mathtt{W}_i(s_n)\right|^2-2\lambda_n\sum_{i=1}^4\mathtt{Y}_i(s_n)\cdot \mathtt{W}_i(s_n)+\lambda_n^2\sum_{i=1}^4\left|\mathtt{Y}_i(s_n)\right|^2
\\
&=\sum_{i=1}^4\left|\mathtt{W}_i(s_n)\right|^2-\frac{\left(\sum_{i=1}^4\mathtt{Y}_i(s_n)\cdot \mathtt{W}_i(s_n)\right)^2}{\sum_{i=1}^4\left|\mathtt{Y}_i(s_n)\right|^2}
\\
&=\frac{\left(\sum_{i=1}^4\left|\mathtt{Y}_i(s_n)\right|^2\right)\left(\sum_{i=1}^4\left|\mathtt{W}_i(s_n)\right|^2\right)-\left(\sum_{i=1}^4 \mathtt{Y}_i(s_n)\cdot\mathtt{W}_i(s_n)\right)^2}{\sum_{i=1}^4\left|\mathtt{Y}_i(s_n)\right|^2}
\\
&\to 0 \ (n\to\infty).
\end{align*}
Thus, we obtain
\begin{align}\label{Whyouka0808}
\lim_{n\to\infty}\mathtt{W}_i(s_n)=\lambda Y_i.
\end{align}

Step 3: We prove \eqref{odekarakuru0808}.

By \eqref{mathcalFhyouka1}, Proposition \ref{samesolitonhanare}, and \eqref{VsimmathcalF}, we have for every
$i\in\mathbb{Z}/4\mathbb{Z}$,
\begin{align}\label{Wibetaexpansion}
\mathtt{W}_i(t)&=-\frac{\mathcal{F}(\rho_i(t))}{V(t)}v_i(t)+\frac{\mathcal{F}(\rho_{i-1}(t))}{V(t)}v_{i-1}(t)+o(1).
\end{align}
We define
\begin{align*}
\beta_{i,n}=\frac{\mathcal{L}(s_n)\mathcal{F}(\rho_i(s_n))}{\rho_i(s_n)V(s_n)}.
\end{align*}
Since we have 
\begin{align*}
\mathtt{Y}_{i+1}(s_n)-\mathtt{Y}_i(s_n)=\frac{\rho_i(s_n)}{\mathcal{L}(s_n)}v_i(s_n),
\end{align*}
\eqref{Wibetaexpansion} gives
\begin{align}\label{Wibetatukau}
\begin{aligned}
\mathtt{W}_i(s_n)&=-\beta_{i,n}\left(\mathtt{Y}_{i+1}(s_n)-\mathtt{Y}_i(s_n)\right)
\\
&\quad+\beta_{i-1,n}\left(\mathtt{Y}_i(s_n)-\mathtt{Y}_{i-1}(s_n)\right)+o(1).
\end{aligned}
\end{align}
The sequences $(\beta_{i,n})_n$ are nonnegative and bounded. After passing to a common subsequence and relabeling it as $(s_n)$, we may assume that
\begin{align}\label{alphanoteigi}
\alpha_i=\lim_{n\to\infty}\beta_{i,n}=\lim_{n\to\infty}\frac{\mathcal{L}(s_n)\mathcal{F}(\rho_i(s_n))}{\rho_i(s_n)V(s_n)}\geq 0
\end{align}
for every $i\in\mathbb{Z}/4\mathbb{Z}$. By \eqref{Ynokyokugen0808}, \eqref{Whyouka0808}, \eqref{Wibetatukau}, and \eqref{alphanoteigi}, we obtain \eqref{odekarakuru0808}.

Step 4: We prove the remaining assertions.

First, summing \eqref{odekarakuru0808} over $i=1,2,3,4$ yields
\begin{align*}
\sum_{i=1}^4Y_i=0.
\end{align*}
As for \eqref{nagasa1ijou0808}, we have
\begin{align*}
\left|\mathtt{Y}_i-\mathtt{Y}_j\right|
=
\frac{1}{\mathcal{L}}\left|z_i-z_j\right|
\geq
\frac{D}{\mathcal{L}}.
\end{align*}
Passing to the limit along the sequence $s_n\to\infty$, we obtain \eqref{nagasa1ijou0808}. Moreover, if $\alpha_i>0$, then the definition of $\alpha_i$ in
\eqref{alphanoteigi} implies that $\mathcal{F}(\rho_i(s_n)) \gtrsim \mathcal{F}(D(s_n))$. Consequently,
\begin{align*}
\left|\rho_i(s_n)-\mathcal{L}(s_n)\right|
\lesssim \log{\mathcal{L}(s_n)},
\end{align*}
and hence
\begin{align*}
|Y_i-Y_{i+1}|
=
\lim_{n\to\infty}
\frac{\rho_i(s_n)}{\mathcal{L}(s_n)}
=1.
\end{align*}
It remains only to prove that $\sum_{i=1}^4\alpha_i=1$. We begin by proving that for every $i\in\mathbb{Z}/4\mathbb{Z}$, 
\begin{align}\label{0808hyouka1}
\frac{\mathcal{F}(\rho_i)}{\rho_i}=\frac{\mathcal{F}(\rho_i)}{D}+o\left(\frac{\mathcal{F}(D)}{D}\right).
\end{align}
When $\rho_i-D\leq \sqrt{D}$, we have
\begin{align*}
\left|\frac{\mathcal{F}(\rho_i)}{\rho_i}-\frac{\mathcal{F}(\rho_i)}{D}\right|\lesssim \frac{\mathcal{F}(D)}{D\sqrt{D}}.
\end{align*}
When $\rho_i-D>\sqrt{D}$, we have $\mathcal{F}(\rho_i)\lesssim \frac{\mathcal{F}(D)}{D}$, and therefore we obtain
\begin{align*}
\left|\frac{\mathcal{F}(\rho_i)}{\rho_i}-\frac{\mathcal{F}(\rho_i)}{D}\right|\lesssim \frac{\mathcal{F}(D)}{D^2}.
\end{align*}
Thus, in either case, we obtain \eqref{0808hyouka1}. We note that by Proposition \ref{samesolitonhanare}, we have
\begin{align}\label{0808hyouka2}
V(t)=\sum_{i=1}^4\mathcal{F}(\rho_i)+o\left(\mathcal{F}(D)\right).
\end{align}
By \eqref{0808hyouka1} and \eqref{0808hyouka2}, we obtain
\begin{align}\label{0808hyouka3}
\sum_{i=1}^4\frac{\mathcal{L}\mathcal{F}(\rho_i)}{\rho_iV}=\sum_{i=1}^4\frac{\mathcal{L}\mathcal{F}(\rho_i)}{DV}+o(1)=\frac{\mathcal{L}}{D}+o(1)=1+o(1).
\end{align}
By \eqref{alphanoteigi} and \eqref{0808hyouka3}, we complete the proof.

\end{proof}

\subsection{The center-configuration equation and geometric classification}

We next give a more explicit characterization of the conditions in Lemma \ref{tenretujouken}.

\begin{lemma}\label{Yhishi}
Let $Y_i\in\mathbb{R}^d$ and $\alpha_i\geq0$ for $i\in\mathbb{Z}/4\mathbb{Z}$, and let $\lambda>0$. We assume that we have for every $i\in\mathbb{Z}/4\mathbb{Z}$,
\begin{align}\label{Yichiass}
\begin{aligned}
\lambda Y_i=-\alpha_i\left(Y_{i+1}-Y_i\right)+\alpha_{i-1}\left(Y_i-Y_{i-1}\right),
\\
\sum_{i=1}^4Y_i=0,\ \sum_{i=1}^4\alpha_i=1,
\\
\left|Y_i-Y_j\right|\geq 1\ \ (i\neq j),
\\
\left|Y_{i+1}-Y_i\right|=1\ \ (\alpha_i>0).
\end{aligned}
\end{align}
Then one of the following two alternatives holds:
\begin{itemize}
\item There exist $u\in S^{d-1}$ and $k\in\mathbb{Z}/4\mathbb{Z}$ such that 
\begin{align}\label{tyokusengata}
\begin{aligned}
\left(Y_k,Y_{k+1},Y_{k+2},Y_{k+3}\right)&=\left(-\frac{3}{2}u,-\frac{1}{2}u,\frac{1}{2}u,\frac{3}{2}u\right),
\\
\left(\alpha_k,\alpha_{k+1},\alpha_{k+2},\alpha_{k+3},\lambda\right)&=\left(\frac{3}{10},\frac{2}{5},\frac{3}{10},0,\frac{1}{5}\right).
\end{aligned}
\end{align}

\item There exist $u,v\in S^{d-1}$ such that 
\begin{align}\label{hishigatagata}
\begin{aligned}
\left(Y_1,Y_2,Y_3,Y_4\right)&=\left(\frac{u+v}{2},\frac{-u+v}{2},\frac{-u-v}{2},\frac{u-v}{2}\right),
\\
\left(\alpha_1,\alpha_2,\alpha_3,\alpha_4,\lambda\right)&=\left(\frac{1}{4},\frac{1}{4},\frac{1}{4},\frac{1}{4},\frac{1}{2}\right).
\end{aligned}
\end{align}
Furthermore, $u$ and $v$ satisfy 
\begin{align}\label{hishigatanaisekijouken}
-\frac{1}{2}\leq u\cdot v\leq \frac{1}{2}.
\end{align}
\end{itemize}
\end{lemma}

\begin{proof}
First, we define $x_i=Y_{i+1}-Y_i$ and $\mathcal{V}=\{ i\in\mathbb{Z}/4\mathbb{Z}:\ \alpha_i>0\}$. We note that by \eqref{Yichiass}, when $i\in\mathcal{V}$, we have $|x_i|=1$. 

We assume that $|\mathcal{V}|=1$ and $i\in\mathcal{V}$. Then, we have $\alpha_{i+1}=\alpha_{i+2}=\alpha_{i+3}=0$. Therefore, we obtain $Y_{i+2}=Y_{i+3}=0$ and we have $|x_{i+2}|=0$, which contradicts $|Y_{i+3}-Y_{i+2}|\geq 1$. Therefore, we obtain $|\mathcal{V}|\neq 1$.

Next, we assume that $|\mathcal{V}|=2$. Then, since $|\mathcal{V}|=2$, its two elements are either adjacent or opposite with respect to the cyclic ordering of $\mathbb{Z}/4\mathbb{Z}$. Therefore, after a cyclic relabeling of the indices, one of the following two cases occurs:

Case 1: There exists $i\in\mathbb{Z}/4\mathbb{Z}$ such that $\{i,i+1\}=\mathcal{V}$.

Then, we have $Y_{i+3}=0$ and $Y_i+Y_{i+1}+Y_{i+2}=0$. Furthermore, by \eqref{Yichiass}, we have
\begin{align}\label{tyokusen0808}
\begin{aligned}
\lambda Y_i&=-\alpha_i\left(Y_{i+1}-Y_i\right),
\\
\lambda Y_{i+1}&=-\alpha_{i+1}\left(Y_{i+2}-Y_{i+1}\right)+\alpha_i\left(Y_{i+1}-Y_i\right),
\\
\lambda Y_{i+2}&=\alpha_{i+1}\left(Y_{i+2}-Y_{i+1}\right).
\end{aligned}
\end{align}
By \eqref{tyokusen0808}, there exist $u\in S^{d-1}$ and $\lambda_i,\lambda_{i+1},\lambda_{i+2}\in \mathbb{R}$ such that 
\begin{align*}
Y_i=\lambda_iu,\ Y_{i+1}=\lambda_{i+1}u,\ Y_{i+2}=\lambda_{i+2}u.
\end{align*}
Then, we have
\begin{align*}
Y_i+Y_{i+1}+Y_{i+2}=\left(\lambda_i+\lambda_{i+1}+\lambda_{i+2}\right)u=0,
\end{align*}
and we obtain 
\begin{align}\label{2kocase1Yass}
\lambda_i+\lambda_{i+1}+\lambda_{i+2}=0.
\end{align}
By $|Y_{i+1}-Y_i|=|Y_{i+2}-Y_{i+1}|=1$, we have 
\begin{align*}
\left|\lambda_{i+1}-\lambda_i\right|=\left|\lambda_{i+2}-\lambda_{i+1}\right|=1,
\end{align*}
and we obtain
\begin{align}\label{2kocase1Yass2}
\lambda_{i+1}-\lambda_i=\lambda_{i+2}-\lambda_{i+1}\ \ \mbox{or}\ \ \lambda_{i+1}-\lambda_i=-\left(\lambda_{i+2}-\lambda_{i+1}\right). 
\end{align}
By \eqref{2kocase1Yass} and \eqref{2kocase1Yass2}, we have $\lambda_{i+1}=0$ or $\lambda_{i+2}=\lambda_i$, and therefore we have $Y_{i+1}=Y_{i+3}=0$ or  $Y_i=Y_{i+2}$, which contradicts.

Case 2: There exists $i\in\mathbb{Z}/4\mathbb{Z}$ such that $\{i,i+2\}=\mathcal{V}$.

Then, we have $\alpha_{i+1}=\alpha_{i+3}=0$. In addition, by \eqref{Yichiass}, we have
\begin{align}\label{koyuuhoutei}
\begin{aligned}
\lambda Y_i&=-\alpha_i\left(Y_{i+1}-Y_i\right),
\\
\lambda Y_{i+1}&=\alpha_i\left(Y_{i+1}-Y_i\right),
\\
\lambda Y_{i+2}&=-\alpha_{i+2}\left(Y_{i+3}-Y_{i+2}\right),
\\
\lambda Y_{i+3}&=\alpha_{i+2}\left(Y_{i+3}-Y_{i+2}\right),
\\
&\alpha_i+\alpha_{i+2}=1.
\end{aligned}
\end{align}
By \eqref{koyuuhoutei}, we obtain
\begin{align*}
Y_i+Y_{i+1}=0,\ Y_{i+2}+Y_{i+3}=0,\ \alpha_i=\alpha_{i+2}=\frac{1}{2},\ \lambda=1.
\end{align*}
Thus, there exist $x_i,x_{i+2}\in S^{d-1}$ such that 
\begin{align*}
\left(Y_i,Y_{i+1},Y_{i+2},Y_{i+3}\right)=\left(-\frac{x_i}{2},\frac{x_i}{2},-\frac{x_{i+2}}{2},\frac{x_{i+2}}{2}\right).
\end{align*}
Then, we obtain
\begin{align*}
2&\leq \left|Y_{i+2}-Y_i\right|+\left|Y_{i+3}-Y_i\right|
\\
&=\left|\frac{x_{i+2}}{2}-\frac{x_i}{2}\right|+\left|\frac{x_{i+2}}{2}+\frac{x_i}{2}\right|
\\
&\leq \sqrt{2}\sqrt{\left|\frac{x_{i+2}}{2}-\frac{x_i}{2}\right|^2+\left|\frac{x_{i+2}}{2}+\frac{x_i}{2}\right|^2}
\\
&=\sqrt{2},
\end{align*}
which is a contradiction. 

By Case 1 and Case 2, we have $|\mathcal{V}|\neq 2$.

Next, we assume that $|\mathcal{V}|=3$. Then, there exists $k\in\mathbb{Z}/4\mathbb{Z}$ such that 
\begin{align*}
\{k,k+1,k+2\}=\mathcal{V}.
\end{align*}
Then, we have $\alpha_{k+3}=0$. Therefore, by \eqref{Yichiass}, we obtain
\begin{align}\label{koyuuhoutei2}
\begin{aligned}
\lambda Y_k&=-\alpha_k\left(Y_{k+1}-Y_k\right),
\\
\lambda Y_{k+1}&=-\alpha_{k+1}\left(Y_{k+2}-Y_{k+1}\right)+\alpha_k\left(Y_{k+1}-Y_k\right),
\\
\lambda Y_{k+2}&=-\alpha_{k+2}\left(Y_{k+3}-Y_{k+2}\right)+\alpha_{k+1}\left(Y_{k+2}-Y_{k+1}\right),
\\
\lambda Y_{k+3}&=\alpha_{k+2}\left(Y_{k+3}-Y_{k+2}\right).
\end{aligned}
\end{align}
By \eqref{koyuuhoutei2}, there exist $\tilde{u}\in S^{d-1}$ and $(\lambda_i)_{i\in\mathbb{Z}/4\mathbb{Z}}\subset \mathbb{R}$ such that for every $i\in\mathbb{Z}/4\mathbb{Z}$
\begin{align}\label{Ylambdatildeu}
Y_i=\lambda_i \tilde{u}. 
\end{align}
Then, we have
\begin{align*}
\left|\lambda_{k+1}-\lambda_k\right|=\left|\lambda_{k+2}-\lambda_{k+1}\right|=\left|\lambda_{k+3}-\lambda_{k+2}\right|=1.
\end{align*}
Observe that the three differences $\lambda_{k+1}-\lambda_k$, $\lambda_{k+2}-\lambda_{k+1}$, and $\lambda_{k+3}-\lambda_{k+2}$ must all have the same sign. Indeed, if two of them had opposite signs, then either
$Y_k=Y_{k+2}$ or $Y_{k+1}=Y_{k+3}$ would hold, contradicting \eqref{Yichiass}. Moreover, since $\lambda_k+\lambda_{k+1}+\lambda_{k+2}+\lambda_{k+3}=0$, there exists $\varsigma=\pm 1$ such that
\begin{align}\label{lambdahugou}
\lambda_k=-\frac{3\varsigma}{2},\ \lambda_{k+1}=-\frac{\varsigma}{2},\ \lambda_{k+2}=\frac{\varsigma}{2},\ \lambda_{k+3}=\frac{3\varsigma}{2}. 
\end{align}
By \eqref{koyuuhoutei2}, \eqref{Ylambdatildeu}, and \eqref{lambdahugou}, we obtain
\begin{align*}
\alpha_k=\frac{3}{2}\lambda,\ \alpha_{k+1}=2\lambda,\ \alpha_{k+2}=\frac{3}{2}\lambda.
\end{align*}
By $\alpha_k+\alpha_{k+1}+\alpha_{k+2}=1$, when we define $u=\varsigma \tilde{u}$, we obtain \eqref{tyokusengata}.

Finally, we assume that $|\mathcal{V}|=4$. Then, we have for every $i\in \mathbb{Z}/4\mathbb{Z}$ $|x_i|=1$. Furthermore, by \eqref{Yichiass}, we have
\begin{align}
&x_1+x_2+x_3+x_4=0, \label{wa0}
\\
&\lambda x_1=2\alpha_1x_1-\alpha_2x_2-\alpha_4x_4, \label{x10809}
\\
&\lambda x_2=-\alpha_1x_1+2\alpha_2x_2-\alpha_3x_3, \label{x20809}
\\
&\lambda x_3=-\alpha_2x_2+2\alpha_3x_3-\alpha_4x_4, \label{x30809}
\\
&\lambda x_4=-\alpha_1 x_1-\alpha_3x_3+2\alpha_4x_4. \label{x40809}
\end{align}
By \eqref{wa0}--\eqref{x40809}, we have
\begin{align}\label{linearindependentass}
\begin{aligned}
(\lambda-2\alpha_1)x_1&=(\lambda-2\alpha_3)x_3,
\\
(\lambda-2\alpha_2)x_2&=(\lambda-2\alpha_4)x_4.
\end{aligned}
\end{align}
Here, we assume that $x_1\neq -x_3$ and $x_2\neq -x_4$. Since $x_1,x_2,x_3,x_4$ are unit vectors, either $x_1=x_3$ or $x_1$ and $x_3$ are linearly independent. Similarly, either $x_2=x_4$ or $x_2$ and $x_4$ are linearly independent.

We assume that $x_1=x_3$. By \eqref{wa0}, we obtain $x_2+x_4=-2x_1$. Then, we have
\begin{align*}
\left|x_2+x_4\right|^2=\left|-2x_1\right|^2=4,
\end{align*}
and we obtain $x_2\cdot x_4=1$. Therefore, we have $x_2=x_4=-x_1=-x_3$. Then, we obtain $Y_1=Y_3$, which contradicts \eqref{Yichiass}. Therefore, we obtain $x_1\neq x_3$.

Similarly, assuming that $x_2=x_4$ leads to a contradiction. Hence, we have $x_2\neq x_4$.

If $x_1$ and $x_3$ are linearly independent and $x_2$ and $x_4$ are linearly independent, by \eqref{linearindependentass}, we have $\alpha_1=\alpha_2=\alpha_3=\alpha_4=\frac{\lambda}{2}$. Then, by \eqref{x10809}, we obtain
\begin{align*}
\lambda x_1=\lambda\left(x_1-\frac{1}{2}x_2-\frac{1}{2}x_4\right),
\end{align*}
and we have $x_2+x_4=0$, which contradicts.

The preceding argument shows that either $x_1+x_3=0$ or $x_2+x_4=0$. Furthermore, by \eqref{wa0}, we obtain
\begin{align}\label{x1+x3=0andx2+x4=0}
x_1+x_3=0,\ x_2+x_4=0.
\end{align}
Here, we define $u=-x_1$ and $v=x_4$. Then, $u,v\in S^{d-1}$. Note that by \eqref{x1+x3=0andx2+x4=0} and $\sum_{i=1}^4Y_i=0$, we have $Y_1=-Y_3$ and $Y_2=-Y_4$. Then, by the definition of $x_i$, we obtain
\begin{align}\label{Y0809}
Y_1=\frac{u+v}{2},\ Y_2=\frac{-u+v}{2},\ Y_3=\frac{-u-v}{2},\ Y_4=\frac{u-v}{2}.
\end{align}
Furthermore, by \eqref{Y0809}, if $u$ and $v$ were parallel, two of the points $Y_1,Y_2,Y_3,Y_4$ would coincide,
contradicting \eqref{Yichiass}. Hence, $u$ and $v$ are linearly independent. Then, by \eqref{x10809}-\eqref{x40809}, we have for every $i\in\mathbb{Z}/4\mathbb{Z}$
\begin{align*}
\left(\lambda-2\alpha_i\right)x_i=\left(\alpha_{i-1}-\alpha_{i+1}\right)x_{i+1}.
\end{align*}
Since $x_i$ and $x_{i+1}$ are linearly independent, we have $\alpha_i=\frac{\lambda}{2}$. Furthermore, by $\alpha_1+\alpha_2+\alpha_3+\alpha_4=1$, we obtain
\begin{align*}
\alpha_1=\alpha_2=\alpha_3=\alpha_4=\frac{1}{4},\ \lambda=\frac{1}{2}.
\end{align*}
Therefore, we have \eqref{hishigatagata}. Then, since we have $|Y_3-Y_1|\geq 1$ and $|Y_4-Y_2|\geq 1$, we have
\begin{align*}
|u+v|\geq 1,\ |u-v|\geq 1,
\end{align*}
which implies \eqref{hishigatanaisekijouken}.
\end{proof}

We now use this result to establish the following subsequential classification of the asymptotic geometry of $(2,2)$-sign soliton solutions. Before we introduce the lemma, we define $\Omega_{l,i}$, $\Omega_l$, and $\Omega_r$ as 
\begin{align*}
&\Omega_{l,i}
=
\left\{
(x_1,x_2,x_3,x_4)\in \left(\mathbb{R}^d\right)^{\mathbb{Z}/4\mathbb{Z}}
\,\middle|\,
\begin{aligned}
&\mbox{there exists}\ u\in S^{d-1}\ \mbox{such that} \\
&x_i=-\frac{3}{2}u,\ x_{i+1}=-\frac{1}{2}u,\ x_{i+2}=\frac{1}{2}u,\ x_{i+3}=\frac{3}{2}u.
\end{aligned}
\right\},
\\
\\
&\Omega_l=\Omega_{l,1}\cup \Omega_{l,2}\cup \Omega_{l,3}\cup \Omega_{l,4},
\\
&\Omega_r=\left\{\left(\frac{u+v}{2},\frac{-u+v}{2},\frac{-u-v}{2},\frac{u-v}{2}\right)\in \left(\mathbb{R}^d\right)^{\mathbb{Z}/4\mathbb{Z}}: u,v\in S^{d-1},\ -\frac{1}{2}\leq u\cdot v\leq \frac{1}{2}\right\}.
\end{align*}

\begin{lemma}\label{sequencerough}
Let $\vec{u}$ be a $(2,2)$-sign soliton solution. Then, for every sequence of times $t_n\to\infty$, there exist a subsequence $(t_n^{\prime})$ of $(t_n)$ and $\left(\mathtt{z}_1,\mathtt{z}_2,\mathtt{z}_3,\mathtt{z}_4\right)\in \Omega_l\cup \Omega_r$ such that
\begin{align*}
&\quad\lim_{n\to\infty}\left(\frac{z_1(t_n^{\prime})-z_g(t_n^{\prime})}{\mathcal{L}(t_n^{\prime})},\frac{z_2(t_n^{\prime})-z_g(t_n^{\prime})}{\mathcal{L}(t_n^{\prime})},\frac{z_3(t_n^{\prime})-z_g(t_n^{\prime})}{\mathcal{L}(t_n^{\prime})},\frac{z_4(t_n^{\prime})-z_g(t_n^{\prime})}{\mathcal{L}(t_n^{\prime})}\right)
\\
&=\left(\mathtt{z}_1,\mathtt{z}_2,\mathtt{z}_3,\mathtt{z}_4\right).
\end{align*}
\end{lemma}

\begin{proof}
By Lemma \ref{tenretujouken}, there exist $(t_n^{\prime})\subset (t_n)$, $Y_i\in\mathbb{R}^d$, $\alpha_i\geq 0$ ($i\in\mathbb{Z}/4\mathbb{Z}$), and $\lambda>0$ such that \eqref{abouttenretu1}--\eqref{nagasa1ijou0808} hold. Then, by applying Lemma \ref{Yhishi}, we obtain
\begin{align}\label{Omegain}
(Y_1,Y_2,Y_3,Y_4)\in \Omega_l\cup \Omega_r.
\end{align}
By \eqref{abouttenretu1} and \eqref{Omegain}, we complete the proof.
\end{proof}

\subsection{Full time classification of the asymptotic geometry}

In this subsection, we establish full time asymptotic convergence. We first prove a basic property of the relevant sets.
\begin{lemma}\label{setlemma}
For every $i\in\mathbb{Z}/4\mathbb{Z}$, we have
\begin{align*}
\operatorname{dist}\left(\Omega_{l,i},\Omega_r\right)>0.
\end{align*}
For every $i,j\in\mathbb{Z}/4\mathbb{Z}$ satisfying $i\neq j$, we have
\begin{align*}
\operatorname{dist}\left(\Omega_{l,i},\Omega_{l,j}\right)>0.
\end{align*}
\end{lemma}
\begin{proof}
By definition, $\Omega_r$ and $\Omega_{l,i}$ are compact for every $i\in\mathbb{Z}/4\mathbb{Z}$. If $(x_1,x_2,x_3,x_4)\in\Omega_{l,i}$, then
\begin{align*}
|x_i-x_{i+3}|=3,
\end{align*}
whereas the other three cyclically consecutive distances are equal to $1$. Thus, the index $i$ is uniquely determined by the position of the cyclic gap of length $3$. Consequently,
\begin{align*}
\Omega_{l,i}\cap\Omega_{l,j}=\varnothing
\qquad (i\neq j).
\end{align*}
On the other hand, every $(x_1,x_2,x_3,x_4)\in\Omega_r$ satisfies
\begin{align*}
|x_{i+1}-x_i|=1
\end{align*}
for every $i\in\mathbb{Z}/4\mathbb{Z}$. Hence,
\begin{align*}
\Omega_{l,i}\cap\Omega_r=\varnothing
\end{align*}
for every $i\in\mathbb{Z}/4\mathbb{Z}$. Since disjoint compact subsets of a metric space have positive distance, the desired conclusions follow.
\end{proof}

Here, we define $\Gamma(\vec{u})$ as 
\begin{align}\label{Gammaudef}
\Gamma(\vec{u})(t)=\left(\frac{z_1(t)-z_g(t)}{\mathcal{L}(t)},\frac{z_2(t)-z_g(t)}{\mathcal{L}(t)},\frac{z_3(t)-z_g(t)}{\mathcal{L}(t)},\frac{z_4(t)-z_g(t)}{\mathcal{L}(t)}\right).
\end{align}

\begin{proposition}\label{2,2fullconv1}
Let $\vec{u}$ be a $(2,2)$-sign soliton solution. Then exactly one of the following two alternatives holds:
\begin{itemize}
\item There exists $i\in \mathbb{Z}/4\mathbb{Z}$ such that 
\begin{align}\label{shuugoukyoku1}
\lim_{t\to\infty}\operatorname{dist}(\Gamma(\vec{u})(t),\Omega_{l,i})=0.
\end{align}

\item We have
\begin{align}\label{shuugoukyoku2}
\lim_{t\to\infty}\operatorname{dist}(\Gamma(\vec{u})(t),\Omega_r)=0.
\end{align}
\end{itemize}
\end{proposition}
\begin{proof}
By Lemma \ref{sequencerough}, there exists a sequence $t_n\to\infty$ such that 
\begin{align*}
\lim_{n\to\infty}\operatorname{dist}\left(\Gamma(\vec{u})(t_n),\Omega_l\right)=0\ \mbox{or}\ \lim_{n\to\infty}\operatorname{dist}\left(\Gamma(\vec{u})(t_n),\Omega_r\right)=0. 
\end{align*}

We consider $\lim_{n\to\infty}\operatorname{dist}\left(\Gamma(\vec{u})(t_n),\Omega_l\right)=0$. Since
\begin{align*}
\Omega_l=\Omega_{l,1}\cup\Omega_{l,2}\cup\Omega_{l,3}\cup\Omega_{l,4},
\end{align*}
after passing to a further subsequence if necessary, we may assume that there exists $i\in\mathbb{Z}/4\mathbb{Z}$ such that
\begin{align*}
\lim_{n\to\infty}\operatorname{dist}\left(\Gamma(\vec{u})(t_n),\Omega_{l,i}\right)=0.
\end{align*}
Then, we show $\lim_{t\to\infty}\operatorname{dist}\left(\Gamma(\vec{u})(t),\Omega_{l,i}\right)=0$. We assume that there exist sufficiently small $0<\delta\ll 1$ and a sequence $t_n^{\prime}\to \infty$ such that, for all $n\in\mathbb{N}$,
\begin{align*}
\operatorname{dist}\left(\Gamma(\vec{u})(t_n^{\prime}),\Omega_{l,i}\right)>\delta.
\end{align*}
Then, by Lemma \ref{sequencerough}, there exists $(t_n^{\prime \prime})\subset (t_n^{\prime})$ such that 
\begin{align*}
\lim_{n\to\infty}\operatorname{dist}\left(\Gamma(\vec{u})(t_n^{\prime \prime}),\Omega_r\cup \Omega_{l,i+1}\cup \Omega_{l,i+2}\cup \Omega_{l,i+3} \right)=0.
\end{align*}
By decreasing $\delta>0$ if necessary, we may assume that
\begin{align*}
4\delta<\min_{\substack{A,B\in\{\Omega_r,\Omega_{l,1},\Omega_{l,2},\Omega_{l,3},\Omega_{l,4}\}
\\
A\neq B}}\operatorname{dist}(A,B).
\end{align*}
Moreover, after passing to further subsequences and relabeling, we may
assume that
\begin{align*}
t_n<t_n^{\prime\prime}<t_{n+1}
\end{align*}
for every $n\in\mathbb{N}$. Then, since $\Gamma(\vec{u})$ is a continuous function and Lemma \ref{setlemma} holds, there exists a sequence $s_n\to \infty$ such that for every $n\in\mathbb{N}$,
\begin{align}\label{tneretuuki}
\operatorname{dist}\left(\Gamma(\vec{u})(s_n),\Omega_r\cup \Omega_l\right)\geq \delta.
\end{align}
On the other hand, by Lemma \ref{sequencerough}, there exists $(s_n^{\prime})\subset (s_n)$ such that 
\begin{align*}
\lim_{n\to\infty}\operatorname{dist}\left(\Gamma(\vec{u})(s_n^{\prime}),\Omega_r\cup \Omega_l\right)=0,
\end{align*}
which contradicts \eqref{tneretuuki}. Therefore, we obtain \eqref{shuugoukyoku1}. The case
\begin{align*}
\lim_{n\to\infty}\operatorname{dist}\left(\Gamma(\vec{u})(t_n),\Omega_r\right)=0
\end{align*}
is treated in the same way and yields \eqref{shuugoukyoku2}. Finally, Lemma \ref{setlemma} shows that the two alternatives are mutually exclusive. This completes the proof.
\end{proof}

\section{Long-time behavior of line-type solutions}

By Proposition \ref{2,2fullconv1}, the asymptotic geometry of every $(2,2)$-sign soliton solution is either of line type or
of rhombus type. In this section, we determine the precise long-time behavior in the line-type case. We begin with the
following definitions.

\begin{definition}\label{typedef}
\begin{enumerate}
\item A $(2,2)$-sign soliton solution $\vec{u}$ is called a \textit{line-type solution} if there exists $i\in \mathbb{Z}/4\mathbb{Z}$ such that 
\begin{align}\label{linetypedefes}
\lim_{t\to\infty}\operatorname{dist}(\Gamma(\vec{u})(t),\Omega_{l,i})=0.
\end{align}

\item A $(2,2)$-sign soliton solution $\vec{u}$ is called an \textit{$i$-line-type solution} for $i\in\mathbb{Z}/4\mathbb{Z}$ if we have
\begin{align}\label{i-typedefes}
\lim_{t\to\infty}\operatorname{dist}(\Gamma(\vec{u})(t),\Omega_{l,i})=0.
\end{align}

\item A $(2,2)$-sign soliton solution $\vec{u}$ is called a \textit{rhombus-type solution} if we have
\begin{align}\label{hishidefes}
\lim_{t\to\infty}\operatorname{dist}(\Gamma(\vec{u})(t),\Omega_r)=0.
\end{align}
\end{enumerate}
\end{definition}

\subsection{ODEs for line-type solutions}

In this subsection, we begin the analysis of line-type solutions. By cyclic symmetry, it suffices to consider the long-time behavior of 1-line-type solutions. We first derive explicit equations governing their soliton centers.

\begin{lemma}\label{1-type-ode}
Let $\vec{u}$ be a 1-line-type solution. Then, for every $1<\theta<\min{(p-1,\frac{3}{2})}$, we have for $t\gg 1$,
\begin{align}\label{1-line-ode}
\begin{aligned}
\dot{z}_1&=-\mathcal{F}(\rho_1)v_1+O\left(t^{-\theta}\right),
\\
\dot{z}_2&=\mathcal{F}(\rho_1)v_1-\mathcal{F}(\rho_2)v_2+O\left(t^{-\theta}\right),
\\
\dot{z}_3&=\mathcal{F}(\rho_2)v_2-\mathcal{F}(\rho_3)v_3+O\left(t^{-\theta}\right),
\\
\dot{z}_4&=\mathcal{F}(\rho_3)v_3+O\left(t^{-\theta}\right),
\end{aligned}
\end{align}
\begin{align}\label{1-line-nagasaode}
\begin{aligned}
\dot{\rho}_1&=2\mathcal{F}(\rho_1)-\mathcal{F}(\rho_2)v_1\cdot v_2+O\left(t^{-\theta}\right),
\\
\dot{\rho}_2&=-\mathcal{F}(\rho_1)v_1\cdot v_2+2\mathcal{F}(\rho_2)-\mathcal{F}(\rho_3)v_2\cdot v_3+O\left(t^{-\theta}\right),
\\
\dot{\rho}_3&=-\mathcal{F}(\rho_2)v_2\cdot v_3+2\mathcal{F}(\rho_3)+O\left(t^{-\theta}\right).
\end{aligned}
\end{align}
\end{lemma}
\begin{proof}
By Definition \ref{typedef}, we have
\begin{align*}
\mathcal{F}(R_1)+\mathcal{F}(R_2)+\mathcal{F}(\rho_4)\lesssim t^{-\theta}.
\end{align*}
Combining this estimate with Lemma \ref{2,2centerdynamics2} and Lemma \ref{Fsimt-1lem}, we obtain \eqref{1-line-ode} and \eqref{1-line-nagasaode}.
\end{proof}

Here, we define 
\begin{align*}
\mathfrak{a}_{12}&=1-v_1\cdot v_2,
\\
\mathfrak{a}_{23}&=1-v_2\cdot v_3,
\\
\mathfrak{a}_{13}&=1-v_1\cdot v_3,
\\
|\mathfrak{a}|&=\sqrt{\mathfrak{a}_{12}^2+\mathfrak{a}_{23}^2+\mathfrak{a}_{13}^2}.
\end{align*}

Then, by \eqref{i-typedefes}, we have
\begin{align}\label{linenaisekiga1}
\begin{aligned}
0\leq \min{(\mathfrak{a}_{12},\mathfrak{a}_{23},\mathfrak{a}_{13})},
\\
\lim_{t\to\infty}|\mathfrak{a}(t)|=0.
\end{aligned}
\end{align}

\begin{lemma}\label{line-type-naisekinagasaode}
Let $\vec{u}$ be a 1-line-type solution. Then, for every $1<\theta<\min{(p-1,\frac{3}{2})}$, we have for $t\gg 1$,
\begin{align}\label{1-line-typenagasaode}
\begin{aligned}
\dot{\rho}_1&=2\mathcal{F}(\rho_1)-\mathcal{F}(\rho_2)(1-\mathfrak{a}_{12})+O\left(t^{-\theta}\right),
\\
\dot{\rho}_2&=-\mathcal{F}(\rho_1)(1-\mathfrak{a}_{12})+2\mathcal{F}(\rho_2)-\mathcal{F}(\rho_3)(1-\mathfrak{a}_{23})+O\left(t^{-\theta}\right),
\\
\dot{\rho}_3&=-\mathcal{F}(\rho_2)(1-\mathfrak{a}_{23})+2\mathcal{F}(\rho_3)+O\left(t^{-\theta}\right).
\end{aligned}
\end{align}
Furthermore, we have for $t\gg 1$,
\begin{align}\label{kinenaisekiode}
\begin{aligned}
\dot{\mathfrak{a}}_{12}&=\left(\frac{\mathcal{F}(\rho_2)}{\rho_1}+\frac{\mathcal{F}(\rho_1)}{\rho_2}\right)\mathfrak{a}_{12}\left(2-\mathfrak{a}_{12}\right)
\\
&\quad+\frac{\mathcal{F}(\rho_3)}{\rho_2}\left(\mathfrak{a}_{12}+\mathfrak{a}_{23}-\mathfrak{a}_{13}-\mathfrak{a}_{12}\mathfrak{a}_{23}\right)+O\left(\frac{t^{-\theta}}{\log{t}}\right),
\\
\dot{\mathfrak{a}}_{23}&=\left(\frac{\mathcal{F}(\rho_2)}{\rho_3}+\frac{\mathcal{F}(\rho_3)}{\rho_2}\right)\mathfrak{a}_{23}\left(2-\mathfrak{a}_{23}\right)
\\
&\quad+\frac{\mathcal{F}(\rho_1)}{\rho_2}\left(\mathfrak{a}_{12}+\mathfrak{a}_{23}-\mathfrak{a}_{13}-\mathfrak{a}_{12}\mathfrak{a}_{23}\right)+O\left(\frac{t^{-\theta}}{\log{t}}\right),
\\
\dot{\mathfrak{a}}_{13}&=\frac{\mathcal{F}(\rho_2)}{\rho_1}\left(\mathfrak{a}_{12}+\mathfrak{a}_{13}-\mathfrak{a}_{23}-\mathfrak{a}_{12}\mathfrak{a}_{13}\right)
\\
&\quad+\frac{\mathcal{F}(\rho_2)}{\rho_3}\left(\mathfrak{a}_{23}+\mathfrak{a}_{13}-\mathfrak{a}_{12}-\mathfrak{a}_{23}\mathfrak{a}_{13}\right)+O\left(\frac{t^{-\theta}}{\log{t}}\right).
\end{aligned}
\end{align}
\end{lemma}
\begin{proof}
Equation \eqref{1-line-typenagasaode} follows directly from Lemma \ref{1-type-ode}. We prove \eqref{kinenaisekiode}. By direct computation, we have
\begin{align*}
\dot{\mathfrak{a}}_{12}&=-\dot{v}_1\cdot v_2-v_1\cdot \dot{v}_2,
\\
\dot{\mathfrak{a}}_{23}&=-\dot{v}_2\cdot v_3-v_2\cdot \dot{v}_3,
\\
\dot{\mathfrak{a}}_{13}&=-\dot{v}_1\cdot v_3-v_1\cdot \dot{v}_3,
\\
\dot{v}_1&=-\frac{\mathcal{F}(\rho_2)}{\rho_1}\left(v_2-(v_1\cdot v_2)v_1\right)+O\left(\frac{t^{-\theta}}{\log{t}}\right),
\\
\dot{v}_2&=-\frac{\mathcal{F}(\rho_1)}{\rho_2}\left(v_1-(v_1\cdot v_2)v_2\right)-\frac{\mathcal{F}(\rho_3)}{\rho_2}\left(v_3-(v_2\cdot v_3)v_2\right)+O\left(\frac{t^{-\theta}}{\log{t}}\right),
\\
\dot{v}_3&=-\frac{\mathcal{F}(\rho_2)}{\rho_3}\left(v_2-(v_2\cdot v_3)v_3\right)+O\left(\frac{t^{-\theta}}{\log{t}}\right),
\\
1-\left(v_1\cdot v_2\right)^2&=\mathfrak{a}_{12}(2-\mathfrak{a}_{12}),
\\
1-\left(v_2\cdot v_3\right)^2&=\mathfrak{a}_{23}(2-\mathfrak{a}_{23}),
\\
1-\left(v_1\cdot v_3\right)^2&=\mathfrak{a}_{13}(2-\mathfrak{a}_{13}).
\end{align*}
Combining the identities above, we obtain \eqref{kinenaisekiode}.

\end{proof}

\subsection{Rough comparison estimates for $\rho_1,\rho_2,\rho_3$}

In this subsection, we derive estimates for $\rho_1,\rho_2,\rho_3$. To obtain more precise comparison functions for $\rho_i$, we further define $L_1,L_2,L_3:[0,\infty)\to\mathbb{R}$ by
\begin{align}\label{L_iode}
\begin{aligned}
L_1(0)&=L_2(0)=L_3(0)\gg 1,
\\
\dot{L}_1&=2\mathcal{F}(L_1)-\mathcal{F}(L_2),
\\
\dot{L}_2&=-\mathcal{F}(L_1)+2\mathcal{F}(L_2)-\mathcal{F}(L_3),
\\
\dot{L}_3&=-\mathcal{F}(L_2)+2\mathcal{F}(L_3).
\end{aligned}
\end{align}
Then, $L_1,L_2,L_3$ satisfy the following estimate.
\begin{lemma}\label{Liasymptotics}
We have
\begin{align}\label{Linagasa}
\begin{aligned}
L_1(t)&=L(t)-\log{\frac{3}{2}}+O\left(\frac{1}{\log{t}}\right),
\\
L_2(t)&=L(t)-\log{2}+O\left(\frac{1}{\log{t}}\right),
\\
L_3(t)&=L(t)-\log{\frac{3}{2}}+O\left(\frac{1}{\log{t}}\right).
\end{aligned}
\end{align}
\end{lemma}

\begin{proof}
By the symmetry of \eqref{L_iode} and uniqueness, we have
\begin{align*}
L_1(t)=L_3(t).
\end{align*}
Then, we have
\begin{align}\label{L1L2ode}
\begin{aligned}
\dot{L}_1&=2\mathcal{F}(L_1)-\mathcal{F}(L_2),
\\
\dot{L}_2&=-2\mathcal{F}(L_1)+2\mathcal{F}(L_2).
\end{aligned}
\end{align}
Then, we have
\begin{align*}
\dot{L}_1-\dot{L}_2=4\mathcal{F}(L_1)-3\mathcal{F}(L_2).
\end{align*}
Here, we introduce the bootstrap estimate
\begin{align}\label{Linotamenobs}
L_1(t)-L_2(t)\geq 0.
\end{align}
We define $T_1\in [0,\infty]$ as 
\begin{align}\label{LibsnoT1}
T_1=\sup{\{t\in[0,\infty)\ \mbox{such that}\ \eqref{Linotamenobs}\ \mbox{holds on}\ [0,t] \}}.
\end{align}
We assume $T_1<\infty$. Then, we have at $t=T_1$ $L_1(T_1)=L_2(T_1)$, and we obtain $\dot{L}_1-\dot{L}_2>0$, which contradicts the maximality of $T_1$. Thus, we have for all $t\geq 0$ $L_1\geq L_2$. In addition, we have $\dot{L}_2\geq 0$, and we obtain $L_2\gg 1$. Next, fix a sufficiently large constant $C>0$ and impose the following bootstrap assumption:
\begin{align}\label{Linotamenobs2}
L_1(t)-L_2(t)\leq C.
\end{align}
We define $T_2\in [0,\infty]$ as 
\begin{align}\label{LibsnoT2}
T_2=\sup{\{t\in[0,\infty)\ \mbox{such that}\ \eqref{Linotamenobs2}\ \mbox{holds on}\ [0,t] \}}.
\end{align}
Furthermore, we assume $T_2<\infty$. Then, we have at $t=T_2$
\begin{align*}
\dot{L}_1-\dot{L}_2=4\mathcal{F}(L_2+C)-3\mathcal{F}(L_2)<0,
\end{align*}
since $C$ is sufficiently large so that for all $r\gg 1$
\begin{align*}
\mathcal{F}(r+C)<\frac{1}{2}\mathcal{F}(r).
\end{align*}
This contradicts the maximality of $T_2$ unless $T_2=\infty$. Hence, $T_2=\infty$. Therefore, we have
\begin{align}\label{L1L2sasim}
0\leq L_1(t)-L_2(t)\lesssim 1.
\end{align}
Then, by \eqref{L1L2ode}, we have
\begin{align*}
\frac{\dot{L}_1+\dot{L}_2}{2}=\frac{1}{2}\mathcal{F}(L_2).
\end{align*}
Furthermore, by \eqref{mathcalFhyouka1} and \eqref{L1L2sasim}, we have
\begin{align*}
\frac{\dot{L}_1+\dot{L}_2}{2}=\frac{1}{2}\mathcal{F}(L_2)\sim \mathcal{F}\left(\frac{L_1+L_2}{2}\right),
\end{align*}
and we obtain
\begin{align}\label{L1L2heikinsim}
\left|\frac{L_1+L_2}{2}-L\right|\lesssim 1.
\end{align}
Then, by \eqref{L1L2sasim} and \eqref{L1L2heikinsim}, we obtain
\begin{align}\label{L1L2Lsim}
\begin{aligned}
|L_1(t)-L(t)|&\leq \left|L_1(t)-\frac{L_1(t)+L_2(t)}{2}\right|+\left|\frac{L_1(t)+L_2(t)}{2}-L(t)\right|\lesssim 1,
\\
|L_2(t)-L(t)|&\leq \left|L_2(t)-\frac{L_1(t)+L_2(t)}{2}\right|+\left|\frac{L_1(t)+L_2(t)}{2}-L(t)\right|\lesssim 1.
\end{aligned}
\end{align}
Here, we define $l_1,l_2$ as 
\begin{align}\label{l1l2def}
\begin{aligned}
l_1(t)=L_1(t)-L(t)+\log{\frac{3}{2}},
\\
l_2(t)=L_2(t)-L(t)+\log{2}.
\end{aligned}
\end{align}
Then, by \eqref{L1L2Lsim}, $l_1,l_2$ are bounded. Furthermore, by \eqref{tukaeruF1}, we have
\begin{align}\label{lhyouka1}
\begin{aligned}
\left|\mathcal{F}(L_1)-\frac{3}{2}e^{-l_1}\mathcal{F}(L)\right|\lesssim \frac{\mathcal{F}(L)}{L}\sim \frac{1}{t\log{t}},
\\
\left|\mathcal{F}(L_2)-2e^{-l_2}\mathcal{F}(L)\right|\lesssim \frac{\mathcal{F}(L)}{L}\sim \frac{1}{t\log{t}}.
\end{aligned}
\end{align}
By \eqref{Lnoode}, \eqref{L1L2ode}, \eqref{l1l2def}, and \eqref{lhyouka1}, we obtain
\begin{align}\label{LinoLyapunov1}
\begin{aligned}
\dot{l}_1&=\left\{3\left(e^{-l_1}-1\right)-2\left(e^{-l_2}-1\right)\right\}\mathcal{F}(L)+O\left(\frac{1}{t\log{t}}\right),
\\
\dot{l}_2&=\left\{-3\left(e^{-l_1}-1\right)+4\left(e^{-l_2}-1\right)\right\}\mathcal{F}(L)+O\left(\frac{1}{t\log{t}}\right).
\end{aligned}
\end{align}
We define $\mathfrak{L}_l$ as 
\begin{align}\label{mathfrakLdef}
\mathfrak{L}_l=3\left(e^{-l_1}-1+l_1\right)+2\left(e^{-l_2}-1+l_2\right).
\end{align}
Since $l_1$ and $l_2$ are bounded, we have
\begin{align}\label{mathfrakLsim}
\mathfrak{L}_l\sim |l_1|^2+|l_2|^2.
\end{align}
By direct computation, we have
\begin{align*}
\dot{\mathfrak{L}}_l&=3\dot{l}_1\left(1-e^{-l_1}\right)+2\dot{l}_2\left(1-e^{-l_2}\right)
\\
&=\left\{-9\left(1-e^{-l_1}\right)^2+12\left(1-e^{-l_1}\right)\left(1-e^{-l_2}\right)-8\left(1-e^{-l_2}\right)^2\right\}\mathcal{F}(L)
\\
&\quad+O\left(\frac{|l_1|+|l_2|}{t\log{t}}\right)
\\
&\leq \left\{ -3\left(1-e^{-l_1}\right)^2-2\left(1-e^{-l_2}\right)^2\right\}\mathcal{F}(L)+O\left(\frac{|l_1|+|l_2|}{t\log{t}}\right).
\end{align*}
Therefore, there exists $c>0$ such that for sufficiently large $t$,
\begin{align}\label{mathfrakLlhutoushiki}
\dot{\mathfrak{L}}_l\leq-\frac{c}{t}\mathfrak{L}_l+\frac{\sqrt{\mathfrak{L}_l}}{ct\log{t}}\leq -\frac{c}{2t}\mathfrak{L}_l+\frac{1}{c^3t(\log{t})^2}.
\end{align}
By \eqref{mathfrakLlhutoushiki} and Gronwall's inequality, we obtain
\begin{align*}
\mathfrak{L}_l\lesssim \frac{1}{(\log{t})^2},
\end{align*}
and we obtain 
\begin{align}\label{l1l2es}
|l_1|+|l_2|\lesssim \frac{1}{\log{t}}.
\end{align}
By \eqref{l1l2def} and \eqref{l1l2es}, we complete the proof.
\end{proof}

We next use $L_1,L_2,L_3$ to estimate the differences $\rho_i-L_i$. We begin by showing that these differences remain bounded.

\begin{lemma}\label{rhoi-Liyuukailem}
Let $\vec{u}$ be a 1-line-type solution. Then, we have
\begin{align}\label{rhii-Liyuukai}
\left|\rho_1-L_1\right|+\left|\rho_2-L_2\right|+\left|\rho_3-L_3\right|\lesssim 1.
\end{align}
\end{lemma}
\begin{proof}
First, we note that by \eqref{rho=D} and \eqref{i-typedefes}, $D=\min_{i=1,2,3}\rho_i$ and by \eqref{linenaisekiga1}, we may assume that we have for $t\gg 1$
\begin{align*}
0\leq \mathfrak{a}_{12}+\mathfrak{a}_{23}+\mathfrak{a}_{13}<\frac{1}{10}.
\end{align*}
We assume that there exists $t^{\prime}\gg 1$ such that 
\begin{align*}
\max_{i=1,2,3}\rho_i(t^{\prime})\leq D(t^{\prime})+\log{5}.
\end{align*}
Then, we introduce the following bootstrap estimate
\begin{align}\label{rho-Lnotame1}
\max_{i=1,2,3}\rho_i(t)\leq D(t)+\log{5}.
\end{align}
We define $T_1\in [t^{\prime},\infty]$ by
\begin{align}\label{rho-Lyuukaibs}
T_1=\sup{\{t\in[t^{\prime},\infty)\ \mbox{such that}\ \eqref{rho-Lnotame1}\ \mbox{holds on}\ [t^{\prime},t] \}}.
\end{align}
Furthermore, we assume $T_1<\infty$. Then, we have at $t=T_1$ $\max_{i=1,2,3}\rho_i(t)=D(t)+\log{5}$. We distinguish the following six possible orderings.

Case 1: $\rho_1\leq \rho_2\leq \rho_3$. Then, we have at $t=T_1$
\begin{align*}
\dot{\rho}_3-\dot{\rho}_1&=-2\mathcal{F}(\rho_1)+2\mathcal{F}(\rho_3)+\left(\mathfrak{a}_{23}-\mathfrak{a}_{12}\right)\mathcal{F}(\rho_2)+o\left(\mathcal{F}(D)\right)
\\
&\leq -2\mathcal{F}(D)+\frac{1}{10}\mathcal{F}(\rho_2)+2\mathcal{F}(\rho_3)+o\left(\mathcal{F}(D)\right)
\\
&\leq  -2\mathcal{F}(D)+\frac{1}{10}\mathcal{F}(D)+\frac{1}{2}\mathcal{F}(D)+o\left(\mathcal{F}(D)\right)
\\
&<-\mathcal{F}(D).
\end{align*}

Case 2: $\rho_1\leq \rho_3\leq \rho_2$. Then, we have at $t=T_1$
\begin{align*}
\dot{\rho}_2-\dot{\rho}_1&=-\left(3-\mathfrak{a}_{12}\right)\mathcal{F}(\rho_1)+\left(3-\mathfrak{a}_{12}\right)\mathcal{F}(\rho_2)
\\
&\quad -\left(1-\mathfrak{a}_{23}\right)\mathcal{F}(\rho_3)+o\left(\mathcal{F}(D)\right)
\\
&\leq -\frac{29}{10}\mathcal{F}(D)+3\mathcal{F}(\rho_2)+o\left(\mathcal{F}(D)\right)
\\
&\leq\left(-\frac{29}{10}+\frac{3}{4}\right)\mathcal{F}(D)+o\left(\mathcal{F}(D)\right)
\\
&<-\mathcal{F}(D).
\end{align*}

Case 3: $\rho_2\leq \rho_1\leq \rho_3$. Then, we have at $t=T_1$
\begin{align*}
\dot{\rho}_3-\dot{\rho}_2&=-\left(3-\mathfrak{a}_{23}\right)\mathcal{F}(\rho_2)+\left(3-\mathfrak{a}_{23}\right)\mathcal{F}(\rho_3)
\\
&\quad+\left(1-\mathfrak{a}_{12}\right)\mathcal{F}(\rho_1)+o\left(\mathcal{F}(D)\right)
\\
&\leq -\frac{29}{10}\mathcal{F}(D)+3\mathcal{F}(\rho_3)+\mathcal{F}(\rho_1)+o\left(\mathcal{F}(D)\right)
\\
&\leq -\frac{29}{10}\mathcal{F}(D)+\frac{3}{4}\mathcal{F}(D)+\mathcal{F}(D)+o\left(\mathcal{F}(D)\right)
\\
&<-\mathcal{F}(D).
\end{align*}

Case 4: $\rho_2\leq \rho_3\leq \rho_1$. Then, we have at $t=T_1$
\begin{align*}
\dot{\rho}_1-\dot{\rho}_2&=-\left(3-\mathfrak{a}_{12}\right)\mathcal{F}(\rho_2)+\left(3-\mathfrak{a}_{12}\right)\mathcal{F}(\rho_1)
\\
&\quad+\left(1-\mathfrak{a}_{23}\right)\mathcal{F}(\rho_3)+o\left(\mathcal{F}(D)\right)
\\
&\leq -\frac{29}{10}\mathcal{F}(D)+3\mathcal{F}(\rho_1)+\mathcal{F}(\rho_3)+o\left(\mathcal{F}(D)\right)
\\
&\leq -\frac{29}{10}\mathcal{F}(D)+\frac{3}{4}\mathcal{F}(D)+\mathcal{F}(D)+o\left(\mathcal{F}(D)\right)
\\
&<-\mathcal{F}(D).
\end{align*}

Case 5: $\rho_3\leq \rho_1\leq \rho_2$. Then, we have at $t=T_1$
\begin{align*}
\dot{\rho}_2-\dot{\rho}_3&=-\left(3-\mathfrak{a}_{23}\right)\mathcal{F}(\rho_3)+\left(3-\mathfrak{a}_{23}\right)\mathcal{F}(\rho_2)
\\
&\quad -\left(1-\mathfrak{a}_{12}\right)\mathcal{F}(\rho_1)+o\left(\mathcal{F}(D)\right)
\\
&\leq -\frac{29}{10}\mathcal{F}(D)+3\mathcal{F}(\rho_2)+o\left(\mathcal{F}(D)\right)
\\
&\leq \left(-\frac{29}{10}+\frac{3}{4}\right)\mathcal{F}(D)+o\left(\mathcal{F}(D)\right)
\\
&<-\mathcal{F}(D).
\end{align*}
Case 6: $\rho_3\leq \rho_2\leq \rho_1$. Then, we have at $t=T_1$
\begin{align*}
\dot{\rho}_1-\dot{\rho}_3&=-2\mathcal{F}(\rho_3)+2\mathcal{F}(\rho_1)+\left(\mathfrak{a}_{12}-\mathfrak{a}_{23}\right)\mathcal{F}(\rho_2)+o\left(\mathcal{F}(D)\right)
\\
&\leq -2\mathcal{F}(D)+\frac{1}{10}\mathcal{F}(\rho_2)+2\mathcal{F}(\rho_1)+o\left(\mathcal{F}(D)\right)
\\
&\leq  -2\mathcal{F}(D)+\frac{1}{10}\mathcal{F}(D)+\frac{1}{2}\mathcal{F}(D)+o\left(\mathcal{F}(D)\right)
\\
&<-\mathcal{F}(D).
\end{align*}

The preceding estimates contradict the maximality of $T_1$ unless $T_1=\infty$. Hence, $T_1=\infty$. Therefore, if there exists $t^{\prime}\gg 1$ such that $\max_{i=1,2,3}\rho_i(t^{\prime})\leq D(t^{\prime})+\log{5}$, we have for $t\geq t^{\prime}$, $\max_{i=1,2,3}\rho_i(t)\leq D(t)+\log{5}$. Here, we assume that there exists $t_n\to\infty$ such that 
\begin{align*}
\max_{i=1,2,3}\rho_i(t_n)>D(t_n)+\log{5}.
\end{align*}
By the preceding argument, if $\max_{i=1,2,3}\rho_i(t)\leq D(t)+\log{5}$ for some $t>t_n$, then there would exist a sufficiently large $n^{\prime}$ such that $\max_{i=1,2,3}\rho_i(t_{n^{\prime}})\leq D(t_{n^{\prime}})+\log{5}$, which is a contradiction. Therefore, $\max_{i=1,2,3}\rho_i(t)>D(t)+\log{5}$ for every $t>t_n$. By the same computation as in the six cases above, for $T\gg t_n$, we obtain
\begin{align*}
-\left(\max_{i=1,2,3}\rho_i(T)-D(T)\right)+\left(\max_{i=1,2,3}\rho_i(t_n)-D(t_n)\right)\gtrsim \int_{t_n}^T\mathcal{F}(D(s))ds\sim \int_{t_n}^T\frac{ds}{s},
\end{align*}
which implies $\max_{i=1,2,3}\rho_i(T)<D(T)$. This is a contradiction. Therefore, we have for all $t\gg 1$ and $i=1,2,3$,
\begin{align}\label{bshakairho-L}
0\leq \rho_i-D\leq \log{5}.
\end{align}
By \eqref{Lnagasa}, \eqref{mathcalFsimt-1}, and \eqref{bshakairho-L}, we have 
\begin{align}\label{rhoi-Lbound0822}
|\rho_i-L|\lesssim 1.
\end{align}
By \eqref{Linagasa} and \eqref{rhoi-Lbound0822}, we obtain \eqref{rhii-Liyuukai}.
\end{proof}
We now use Lemma \ref{rhoi-Liyuukailem} to refine the estimates for $\rho_i-L_i$. In analyzing the long-time behavior of $i$-line-type solutions, it is essential to obtain an error term of order $t^{-\theta}$ for some $\theta>0$. Accordingly, from this point onward, we focus on establishing the existence of such a positive exponent $\theta$, rather than on determining its precise value.
\begin{lemma}\label{rhoinagasalem}
Let $\vec{u}$ be a 1-line-type solution. Then, there exists $0<\theta_1<\min{(p-2,\frac{1}{2})}$ such that, for $t\gg 1$
\begin{align}\label{rhoinagasalemes}
\begin{aligned}
\rho_1&=L_1+O\left(|\mathfrak{a}|+t^{-\theta_1}\right),
\\
\rho_2&=L_2+O\left(|\mathfrak{a}|+t^{-\theta_1}\right),
\\
\rho_3&=L_3+O\left(|\mathfrak{a}|+t^{-\theta_1}\right).
\end{aligned}
\end{align}
\end{lemma}

\begin{remark}
All constants denoted by $\theta_i$ with numerical subscripts, starting with $\theta_1$ introduced in this lemma, are fixed once and for all and retain the same meaning throughout the remainder of the paper.
\end{remark}

\begin{proof}
First, we note that by Lemma \ref{rhoi-Liyuukailem}, we have
\begin{align}\label{mathcalFrhoknosim}
\mathcal{F}(\rho_1)\sim \mathcal{F}(\rho_2)\sim \mathcal{F}(\rho_3)\sim \frac{1}{t}.
\end{align}
We define $\zeta_1,\zeta_2,\zeta_3$ as 
\begin{align}\label{zetaidef}
\begin{aligned}
\zeta_1&=\rho_1-L_1,
\\
\zeta_2&=\rho_2-L_2,
\\
\zeta_3&=\rho_3-L_3,
\\
|\zeta|&=\sqrt{\zeta_1^2+\zeta_2^2+\zeta_3^2}.
\end{aligned}
\end{align}
First, we calculate $\dot{\zeta}_1$. By Lemma \ref{line-type-naisekinagasaode} and \eqref{mathcalFrhoknosim} and direct computation, we have for $t\gg 1$
\begin{align}\label{dotzeta1es1}
\begin{aligned}
\dot{\zeta}_1&=\dot{\rho}_1-\dot{L}_1
\\
&=2\mathcal{F}(\rho_1)-\mathcal{F}(\rho_2)-2\mathcal{F}(L_1)+\mathcal{F}(L_2)+O\left(\frac{|\mathfrak{a}|}{t}+t^{-\theta^{\prime}}\right),
\end{aligned}
\end{align}
where $1<\theta^{\prime}<\min{(p-1,\frac{3}{2})}$ is a constant. Furthermore, by \eqref{tukaeruF1}, we have
\begin{align}\label{rho1rho2es1}
\begin{aligned}
\left|\mathcal{F}(\rho_1)-e^{-\zeta_1}\mathcal{F}(L_1)\right|&\lesssim \frac{|\zeta|}{t\log{t}},
\\
\left|\mathcal{F}(\rho_2)-e^{-\zeta_2}\mathcal{F}(L_2)\right|&\lesssim \frac{|\zeta|}{t\log{t}},
\\
\left|\mathcal{F}(\rho_3)-e^{-\zeta_3}\mathcal{F}(L_3)\right|&\lesssim \frac{|\zeta|}{t\log{t}}.
\end{aligned}
\end{align}
By \eqref{tukaeruF1}, \eqref{Linagasa}, \eqref{dotzeta1es1}, and \eqref{rho1rho2es1}, we obtain
\begin{align*}
\dot{\zeta}_1&=2\left(e^{-\zeta_1}-1\right)\mathcal{F}(L_1)-\left(e^{-\zeta_2}-1\right)\mathcal{F}(L_2)+O\left(\frac{|\mathfrak{a}|}{t}+\frac{|\zeta|}{t\log{t}}+t^{-\theta^{\prime}}\right)
\\
&=\left\{ \frac{3}{2}\left(e^{-\zeta_1}-1\right)-\left(e^{-\zeta_2}-1\right)\right\}\mathcal{F}(L_2)+O\left(\frac{|\mathfrak{a}|}{t}+\frac{|\zeta|}{t\log{t}}+t^{-\theta^{\prime}}\right).
\end{align*}
Using the same estimates together with the equations for $\dot{\rho}_2,\dot{\rho}_3$ in
\eqref{1-line-typenagasaode} and those for $\dot{L}_2,\dot{L}_3$ in \eqref{L_iode}, we obtain
\begin{align*}
\dot{\zeta}_2&=\left\{-\frac{3}{4}\left(e^{-\zeta_1}-1\right)+2\left(e^{-\zeta_2}-1\right)-\frac{3}{4}\left(e^{-\zeta_3}-1\right)\right\}\mathcal{F}(L_2)
\\
&\quad+O\left(\frac{|\mathfrak{a}|}{t}+\frac{|\zeta|}{t\log{t}}+t^{-\theta^{\prime}}\right),
\\
\dot{\zeta}_3&=\left\{-\left(e^{-\zeta_2}-1\right)+\frac{3}{2}\left(e^{-\zeta_3}-1\right) \right\}\mathcal{F}(L_2)+O\left(\frac{|\mathfrak{a}|}{t}+\frac{|\zeta|}{t\log{t}}+t^{-\theta^{\prime}}\right).
\end{align*}
Next, we define
\begin{align*}
\mathfrak{L}=\frac{3}{4}\left(e^{-\zeta_1}+\zeta_1-1\right)+\left(e^{-\zeta_2}+\zeta_2-1\right)+\frac{3}{4}\left(e^{-\zeta_3}+\zeta_3-1\right).
\end{align*}
Since $|\zeta|$ is bounded, there exist $C_1>0, C_2>0$ such that for $t\gg 1$
\begin{align}\label{mathfraksim1}
C_1|\zeta(t)|^2\leq \mathfrak{L}(t)\leq C_2|\zeta(t)|^2.
\end{align}
Furthermore, we have
\begin{align*}
\dot{\mathfrak{L}}&=\frac{3}{4}\dot{\zeta}_1\left(1-e^{-\zeta_1}\right)+\dot{\zeta}_2\left(1-e^{-\zeta_2}\right)+\frac{3}{4}\dot{\zeta}_3\left(1-e^{-\zeta_3}\right)
\\
&=-\left\{ \frac{9}{8}\left(1-e^{-\zeta_1}\right)^2+2\left(1-e^{-\zeta_2}\right)^2+\frac{9}{8}\left(1-e^{-\zeta_3}\right)^2\right\}\mathcal{F}(L_2)
\\
&\quad+\left\{\frac{3}{2}\left(1-e^{-\zeta_1}\right)\left(1-e^{-\zeta_2}\right)+\frac{3}{2}\left(1-e^{-\zeta_2}\right)\left(1-e^{-\zeta_3}\right) \right\}\mathcal{F}(L_2)
\\
&\quad+O\left(\frac{|\mathfrak{a}||\zeta|}{t}+\frac{|\zeta|^2}{t\log{t}}+|\zeta|t^{-\theta^{\prime}}\right).
\end{align*}
Thus, there exists $c>0$ such that, for all sufficiently large $t$,
\begin{align}\label{mathfrakLhyouka1}
\dot{\mathfrak{L}}\leq -\frac{c}{t}\mathfrak{L}+\frac{1}{c}\left(\frac{|\mathfrak{a}||\zeta|}{t}+\frac{|\zeta|^2}{t\log{t}}+|\zeta|t^{-\theta^{\prime}}\right).
\end{align}
Furthermore, by direct computation, we have
\begin{align}\label{keisan0818-1}
\begin{aligned}
|\mathfrak{a}||\zeta|&\leq \frac{C_1c^2}{4}|\zeta|^2+\frac{1}{C_1c^2}|\mathfrak{a}|^2\leq \frac{c^2}{4}\mathfrak{L}(t)+\frac{1}{C_1c^2}|\mathfrak{a}(t)|^2,
\\
|\zeta|t^{-\theta^{\prime}}&\leq \frac{C_1c^2}{4t}|\zeta |^2+\frac{1}{C_1c^2}t^{1-2\theta^{\prime}}\leq \frac{c^2}{4t}\mathfrak{L}(t)+\frac{1}{C_1c^2}t^{1-2\theta^{\prime}}.
\end{aligned}
\end{align}
Since $t$ is sufficiently large, we have for $t\gg 1$
\begin{align}\label{keisan0818-2}
\frac{|\zeta(t)|^2}{t\log{t}}\leq \frac{c^2}{4t}\mathfrak{L}(t).
\end{align}
By \eqref{mathfrakLhyouka1}, \eqref{keisan0818-1}, and \eqref{keisan0818-2}, we have for $t\gg 1$
\begin{align}\label{mathfrakLtukaubibunhutousiki}
\dot{\mathfrak{L}}\leq -\frac{c}{4t}\mathfrak{L}+\frac{|\mathfrak{a}(t)|^2}{C_1c^3t}+\frac{1}{C_1c^3}t^{1-2\theta^{\prime}}.
\end{align}
Here, we define
\begin{align}\label{mathfrakHdef}
\mathfrak{H}(t)=t^{\frac{c}{4}}\mathfrak{L}(t).
\end{align}
Then, by \eqref{mathfrakLtukaubibunhutousiki}, we have
\begin{align*}
\dot{\mathfrak{H}}\leq \frac{1}{C_1c^3}t^{\frac{c}{4}-1}|\mathfrak{a}|^2+\frac{1}{C_1c^3}t^{1+\frac{c}{4}-2\theta^{\prime}},
\end{align*}
and therefore we have for $1\ll s<t$,
\begin{align}\label{mathfrakLnojunbi}
t^{\frac{c}{4}}\mathfrak{L}(t)-s^{\frac{c}{4}}\mathfrak{L}(s)\leq \frac{1}{C_1c^3}\int_s^t r^{\frac{c}{4}-1}|\mathfrak{a}(r)|^2dr+\frac{1}{C_1c^3}\int_s^t r^{1+\frac{c}{4}-2\theta^{\prime}}dr.
\end{align}
We note that by \eqref{kinenaisekiode}, there exist $C_3$ and $T^{\prime}\gg 1$ and $1<\theta^{\prime \prime}<\frac{3}{2}$ such that for $t\geq T^{\prime}$
\begin{align}\label{anoteiginadoed}
\begin{aligned}
\left|\frac{d}{dt}|\mathfrak{a}|^2\right|&\leq \frac{C_3|\mathfrak{a}|^2}{t\log{t}}+C_3t^{-\theta^{\prime \prime}},
\\
\log{T^{\prime}}&>\frac{8C_3}{c}.
\end{aligned}
\end{align}
Then, by direct computation, we have for $T^{\prime}\leq s<t$,
\begin{align}\label{Lhyoukanotame0818}
\begin{aligned}
\frac{c}{4}\int_s^tr^{\frac{c}{4}-1}|\mathfrak{a}(r)|^2dr&=\left(t^{\frac{c}{4}}|\mathfrak{a}(t)|^2-s^{\frac{c}{4}}|\mathfrak{a}(s)|^2\right)-\int_s^tr^{\frac{c}{4}}\frac{d}{dr}|\mathfrak{a}(r)|^2dr
\\
&\leq \left(t^{\frac{c}{4}}|\mathfrak{a}(t)|^2-s^{\frac{c}{4}}|\mathfrak{a}(s)|^2\right)+\int_s^t \frac{C_3r^{\frac{c}{4}-1}|\mathfrak{a}(r)|^2}{\log{r}}dr
\\
&\quad+C_3\int_s^t r^{\frac{c}{4}-\theta^{\prime \prime}}dr.
\end{aligned}
\end{align}
We now apply \eqref{mathfrakLnojunbi}, \eqref{anoteiginadoed}, and \eqref{Lhyoukanotame0818} with $s=T^{\prime}$. Absorbing the terms depending only on $T^{\prime}$ into the implicit constant, we obtain
\begin{align*}
t^{\frac{c}{4}}\mathfrak{L}(t)\lesssim 1+t^{\frac{c}{4}}|\mathfrak{a}(t)|^2+\int_{T^{\prime}}^t r^{\frac{c}{4}-\theta^{\prime\prime}}dr+\int_{T^{\prime}}^t r^{1+\frac{c}{4}-2\theta^{\prime}}dr.
\end{align*}
Therefore, choosing 
\begin{align*}
0<\theta^{\prime\prime\prime}<\min{(\frac{c}{4},\theta^{\prime}-1,\theta^{\prime\prime}-1)},
\end{align*}
we obtain
\begin{align}\label{mathfrakLeskanryou}
\mathfrak{L}(t)\lesssim |\mathfrak{a}(t)|^2+t^{-\theta^{\prime \prime \prime}}.
\end{align}
By \eqref{mathfraksim1} and \eqref{mathfrakLeskanryou}, when we define $\theta_1=\frac{\theta^{\prime \prime \prime}}{2}$, we obtain \eqref{rhoinagasalemes}.
\end{proof}

\subsection{Inner-product estimates for 1-line-type solutions}

In this subsection, we use the estimates for $\rho_i$ to derive estimates for $\mathfrak{a}_{12}$, $\mathfrak{a}_{23}$, and $\mathfrak{a}_{13}$. We begin by rewriting the differential equations governing these quantities.

\begin{lemma}\label{naisekiodekantan}
Let $\vec{u}$ be a 1-line-type solution. Then, there exists $\theta_2>1$ such that 
\begin{align}\label{kinenaisekiode2}
\begin{aligned}
\dot{\mathfrak{a}}_{12}&=\frac{\mathcal{F}(L_2)}{L}\left(\frac{17}{4}\mathfrak{a}_{12}+\frac{3}{4}\mathfrak{a}_{23}-\frac{3}{4}\mathfrak{a}_{13}\right)+O\left(\frac{|\mathfrak{a}|^2}{t\log{t}}+\frac{|\mathfrak{a}|}{t(\log{t})^2}+t^{-\theta_2}\right),
\\
\dot{\mathfrak{a}}_{23}&=\frac{\mathcal{F}(L_2)}{L}\left(\frac{3}{4}\mathfrak{a}_{12}+\frac{17}{4}\mathfrak{a}_{23}-\frac{3}{4}\mathfrak{a}_{13}\right)+O\left(\frac{|\mathfrak{a}|^2}{t\log{t}}+\frac{|\mathfrak{a}|}{t(\log{t})^2}+t^{-\theta_2}\right),
\\
\dot{\mathfrak{a}}_{13}&=\frac{2\mathcal{F}(L_2)}{L}\mathfrak{a}_{13}+O\left(\frac{|\mathfrak{a}|^2}{t\log{t}}+\frac{|\mathfrak{a}|}{t(\log{t})^2}+t^{-\theta_2}\right).
\end{aligned}
\end{align}
\end{lemma}
\begin{proof}
By Lemma \ref{rhoinagasalem}, we have for $i,j\in\{1,2,3\}$,
\begin{align*}
\left|\frac{\mathcal{F}(\rho_i)}{\rho_j}-\frac{\mathcal{F}(L_i)}{L}\right|&\lesssim \left|\frac{\mathcal{F}(\rho_i)-\mathcal{F}(L_i)}{\rho_j}\right|+\left|\frac{\mathcal{F}(L_i)(\rho_j-L)}{\rho_jL}\right|
\\
&\lesssim \frac{|\mathfrak{a}|}{t\log{t}}+t^{-\theta_1-1}.
\end{align*}
Therefore, by \eqref{kinenaisekiode} and \eqref{Linagasa}, we have
\begin{align*}
\dot{\mathfrak{a}}_{12}&=\left(\frac{\mathcal{F}(\rho_2)}{\rho_1}+\frac{\mathcal{F}(\rho_1)}{\rho_2}\right)\mathfrak{a}_{12}\left(2-\mathfrak{a}_{12}\right)
\\
&\quad+\frac{\mathcal{F}(\rho_3)}{\rho_2}\left(\mathfrak{a}_{12}+\mathfrak{a}_{23}-\mathfrak{a}_{13}-\mathfrak{a}_{12}\mathfrak{a}_{23}\right)+O\left(\frac{t^{-\theta_1-1}}{\log{t}}\right)
\\
&=\frac{\mathcal{F}(L_1)+\mathcal{F}(L_2)}{L}2\mathfrak{a}_{12}+\frac{\mathcal{F}(L_3)}{L}\left(\mathfrak{a}_{12}+\mathfrak{a}_{23}-\mathfrak{a}_{13}\right)+O\left(\frac{|\mathfrak{a}|^2}{t\log{t}}+\frac{|\mathfrak{a}|}{t(\log{t})^2}+t^{-\theta_1-1}\right)
\\
&=\frac{\mathcal{F}(L_2)}{L}\left(\frac{17}{4}\mathfrak{a}_{12}+\frac{3}{4}\mathfrak{a}_{23}-\frac{3}{4}\mathfrak{a}_{13}\right)+O\left(\frac{|\mathfrak{a}|^2}{t\log{t}}+\frac{|\mathfrak{a}|}{t(\log{t})^2}+t^{-\theta_1-1}\right).
\end{align*}
The corresponding equations for $\dot{\mathfrak{a}}_{23}$ and $\dot{\mathfrak{a}}_{13}$ follow from the same computation. Setting $\theta_2=\theta_1+1$ completes the proof.
\end{proof}

We now use this estimate to control $\mathfrak{a}_{12}$, $\mathfrak{a}_{23}$, and $\mathfrak{a}_{13}$.
\begin{lemma}\label{naisekiosaerusharp}
Let $\vec{u}$ be a 1-line-type solution. Then, there exists $\theta_3>0$ such that 
\begin{align}\label{,athfrakasharpes}
|\mathfrak{a}|\lesssim t^{-\theta_3}.
\end{align} 
\end{lemma}

\begin{proof}
By Lemma \ref{naisekiodekantan}, there exist constants $0<c<1$ and $\theta_2>1$ such that for $t\gg 1$,
\begin{align}\label{mathfrakaodees2}
\dot{\mathfrak{a}}_{12}+\dot{\mathfrak{a}}_{23}+\dot{\mathfrak{a}}_{13}&\geq  \frac{c}{t\log{t}}\left(5\mathfrak{a}_{12}+5\mathfrak{a}_{23}+\frac{1}{2}\mathfrak{a}_{13}\right)-\frac{1}{c}\left(\frac{|\mathfrak{a}|^2}{t\log{t}}+\frac{|\mathfrak{a}|}{t(\log{t})^2}+t^{-\theta_2}\right).
\end{align}

Furthermore, by \eqref{linenaisekiga1}, there exists $T\gg 1$ such that, for all $t>T$,
\begin{align}\label{mathfrakalimosae}
\frac{|\mathfrak{a}|^2}{t\log{t}}+\frac{|\mathfrak{a}|}{t(\log{t})^2}\leq \frac{c^2}{2t\log{t}}\left(\mathfrak{a}_{12}+\mathfrak{a}_{23}+\mathfrak{a}_{13}\right).
\end{align}
By \eqref{mathfrakaodees2} and \eqref{mathfrakalimosae}, we obtain
\begin{align}\label{mathfrakawaes}
\begin{aligned}
\dot{\mathfrak{a}}_{12}+\dot{\mathfrak{a}}_{23}+\dot{\mathfrak{a}}_{13}&\geq  \frac{c}{t\log{t}}\left(5\mathfrak{a}_{12}+5\mathfrak{a}_{23}+\frac{1}{2}\mathfrak{a}_{13}\right)
\\
&\quad-\frac{c}{2t\log{t}}\left(\mathfrak{a}_{12}+\mathfrak{a}_{23}+\mathfrak{a}_{13}\right)-\frac{3}{c}t^{-\theta_2}
\\
&\geq -\frac{3}{c}t^{-\theta_2}.
\end{aligned}
\end{align}
Integrating \eqref{mathfrakawaes} on $[t,T^{\prime}]$, we obtain
\begin{align*}
\mathfrak{a}_{12}(t)+\mathfrak{a}_{23}(t)+\mathfrak{a}_{13}(t)&\leq \mathfrak{a}_{12}(T^{\prime})+\mathfrak{a}_{23}(T^{\prime})+\mathfrak{a}_{13}(T^{\prime})+\frac{3}{c}\int_t^{T^{\prime}}s^{-\theta_2}ds
\\
&\leq \mathfrak{a}_{12}(T^{\prime})+\mathfrak{a}_{23}(T^{\prime})+\mathfrak{a}_{13}(T^{\prime})+\frac{3}{c\left(\theta_2-1\right)}t^{1-\theta_2}.
\end{align*}
Letting $T^{\prime}\to\infty$ and setting $\theta_3=\theta_2-1$, we obtain \eqref{,athfrakasharpes}.

\end{proof}

By Lemma \ref{rhoinagasalem} and Lemma \ref{naisekiosaerusharp}, there exists $\theta_4>0$ such that 
\begin{align}\label{rhokaizen}
\begin{aligned}
\rho_1&=L_1+O\left(t^{-\theta_4}\right),
\\
\rho_2&=L_2+O\left(t^{-\theta_4}\right),
\\
\rho_3&=L_3+O\left(t^{-\theta_4}\right).
\end{aligned}
\end{align}

In particular, this estimate yields the following result.
\begin{lemma}
Let $\vec{u}$ be a 1-line-type solution. Then, there exists $\theta_5>0$ such that for $t\gg 1$,
\begin{align}\label{taishouesline}
\begin{aligned}
z_2(t)+z_3(t)&=2z_{\infty}+O\left(t^{-\theta_5}\right),
\\
z_1(t)+z_4(t)&=2z_{\infty}+O\left(t^{-\theta_5}\right).
\end{aligned}
\end{align}
\end{lemma}

\begin{proof}
First, we prove that there exists $\theta^{\prime}>0$ such that 
\begin{align}\label{taishouesline2}
\left|z_1(t)+z_4(t)-z_2(t)-z_3(t)\right|\lesssim t^{-\theta^{\prime}}.
\end{align}
We note that $L_1=L_3$. By direct computation, we have
\begin{align}\label{taishoues1}
\begin{aligned}
\left|z_1+z_4-z_2-z_3\right|&=\left|\left(z_4-z_3\right)-\left(z_2-z_1\right)\right|
\\
&=\left|\rho_3v_3-\rho_1v_1-\rho_1v_3+\rho_1v_3 \right|
\\
&\leq \left|\rho_3-\rho_1\right|+\rho_1\left|v_1-v_3\right|.
\end{aligned}
\end{align}
By \eqref{rhokaizen}, we have
\begin{align}\label{rho1-rho3}
\left|\rho_1-\rho_3\right|\lesssim t^{-\theta_4}.
\end{align}
Furthermore, by \eqref{,athfrakasharpes}, we have for $t\gg 1$
\begin{align}\label{mathfraka13es}
\left|v_1-v_3\right|^2=2\left(1-v_1\cdot v_3\right)=2\mathfrak{a}_{13}\lesssim t^{-\theta_3}.
\end{align}
Thus, when we define $0<\theta^{\prime}<\min{(\frac{\theta_3}{2},\theta_4)}$, by \eqref{taishoues1}, \eqref{rho1-rho3}, and \eqref{mathfraka13es}, we have
\begin{align*}
\left|z_1+z_4-z_2-z_3\right|&\lesssim t^{-\theta_4}+\left(\log{t}\right)t^{-\frac{\theta_3}{2}}\lesssim t^{-\theta^{\prime}},
\end{align*}
and we obtain \eqref{taishouesline2}. By \eqref{jushines} and \eqref{taishouesline2}, we complete the proof by choosing $\theta_5=\theta^{\prime}$.
\end{proof}

To describe the positions of $z_i$ relative to $z_{\infty}$, we define, for $i=1,2,3,4$,
\begin{align}\label{hatzdef}
\hat{z}_i(t)=z_i(t)-z_{\infty}.
\end{align}
Lemma \ref{1-type-ode} can then be reformulated as follows.
\begin{lemma}\label{1-type-odekaizen5}
Let $\vec{u}$ be a 1-line-type solution. Then, there exists $\theta_6>1$ such that for $t\gg 1$,
\begin{align}\label{1-line-hatode}
\begin{aligned}
\dot{\hat{z}}_1&=\frac{\mathcal{F}(L_1)}{L_1}\left(\hat{z}_1-\hat{z}_2\right)+O\left(t^{-\theta_6}\right),
\\
\dot{\hat{z}}_2&=-\frac{\mathcal{F}(L_1)}{L_1}\hat{z}_1+\left(\frac{\mathcal{F}(L_1)}{L_1}+\frac{2\mathcal{F}(L_2)}{L_2}\right)\hat{z}_2+O\left(t^{-\theta_6}\right).
\end{aligned}
\end{align}
\end{lemma}

\begin{proof}
By \eqref{rhokaizen}, we have for $i=1,2,3$,
\begin{align}\label{fracFrhorhohyouka}
\left|\frac{\mathcal{F}(\rho_i)}{\rho_i}-\frac{\mathcal{F}(L_i)}{L_i}\right|\lesssim \left|\frac{\mathcal{F}(\rho_i)-\mathcal{F}(L_i)}{\rho_i}\right|+\left|\frac{\mathcal{F}(L_i)(\rho_i-L_i)}{\rho_iL_i}\right|\lesssim \frac{t^{-1-\theta_4}}{\log{t}}.
\end{align}

By \eqref{1-line-ode} and \eqref{fracFrhorhohyouka}, we have
\begin{align}\label{1-line-odever2}
\begin{aligned}
\dot{z}_1&=-\frac{\mathcal{F}(L_1)}{L_1}\left(z_2-z_1\right)+O\left(t^{-1-\theta_4}\right),
\\
\dot{z}_2&=\frac{\mathcal{F}(L_1)}{L_1}\left(z_2-z_1\right)-\frac{\mathcal{F}(L_2)}{L_2}\left(z_3-z_2\right)+O\left(t^{-1-\theta_4}\right).
\end{aligned}
\end{align}
By \eqref{taishouesline}, \eqref{1-line-odever2}, and $\dot{\hat{z}}_1=\dot{z}_1,\dot{\hat{z}}_2=\dot{z}_2$, we obtain \eqref{1-line-hatode} by choosing $\theta_6=\min{(1+\theta_4,1+\theta_5)}$.

\end{proof}

Using Lemma \ref{1-type-odekaizen5}, we now determine the long-time behavior of $z_1,z_2,z_3,z_4$.

\begin{lemma}\label{1-line-type-asymp}
Let $\vec{u}$ be a 1-line-type solution. Then, there exist $\omega_{\infty}\in S^{d-1}$ and $\theta_7>0$ such that for $t\gg 1$,
\begin{align}\label{1-typeasumpes}
\begin{aligned}
z_1(t)&=z_{\infty}-\left(L_1+\frac{1}{2}L_2\right)\omega_{\infty}+O\left(t^{-\theta_7}\right),
\\
z_2(t)&=z_{\infty}-\frac{L_2}{2}\omega_{\infty}+O\left(t^{-\theta_7}\right),
\\
z_3(t)&=z_{\infty}+\frac{L_2}{2}\omega_{\infty}+O\left(t^{-\theta_7}\right),
\\
z_4(t)&=z_{\infty}+\left(L_1+\frac{1}{2}L_2\right)\omega_{\infty}+O\left(t^{-\theta_7}\right).
\end{aligned}
\end{align}
\end{lemma}
\begin{proof}
We define 
\begin{align*}
\mathfrak{D}_1=L_1+\frac{1}{2}L_2, \mathfrak{D}_2=\frac{1}{2}L_2.
\end{align*}
Then, by \eqref{L_iode}, we have
\begin{align*}
\dot{\mathfrak{D}}_1=\mathcal{F}(L_1),\ \dot{\mathfrak{D}}_2=-\mathcal{F}(L_1)+\mathcal{F}(L_2).
\end{align*}
Here, we define
\begin{align*}
\hat{u}_1=\frac{\hat{z}_1}{\mathfrak{D}_1}, \hat{u}_2=\frac{\hat{z}_2}{\mathfrak{D}_2}.
\end{align*}
By \eqref{Lnagasa}, \eqref{mathcalLsim}, \eqref{mathcalFsimt-1}, and \eqref{jushines}, we have
\begin{align*}
\lim_{t\to\infty}\frac{\mathcal{L}(t)}{L(t)}=1,
\\
\lim_{t\to\infty}\frac{z_g(t)-z_{\infty}}{L(t)}=0.
\end{align*}
Moreover, we have
\begin{align*}
\lim_{t\to\infty}\frac{\mathfrak{D}_1(t)}{L(t)}=\frac{3}{2},\ \lim_{t\to\infty}\frac{\mathfrak{D}_2(t)}{L(t)}=\frac{1}{2}.
\end{align*}
Therefore, \eqref{i-typedefes} yields
\begin{align}\label{hatusalim}
\begin{aligned}
\lim_{t\to\infty}\left|\hat{u}_1(t)-\hat{u}_2(t)\right|=0,
\\
\lim_{t\to\infty}\left|\hat{u}_1(t)\right|=1,
\\
\lim_{t\to\infty}\left|\hat{u}_2(t)\right|=1.
\end{aligned}
\end{align}
Furthermore, we have for $t\gg 1$
\begin{align*}
\dot{\hat{u}}_1&=\frac{\dot{\hat{z}}_1}{\mathfrak{D}_1}-\frac{\dot{\mathfrak{D}}_1\hat{z}_1}{\mathfrak{D}_1^2}
\\
&=\frac{\mathcal{F}(L_1)}{L_1\mathfrak{D}_1}\left(\hat{z}_1-\hat{z}_2\right)-\frac{\mathcal{F}(L_1)}{\mathfrak{D}_1^2}\hat{z}_1+O\left(t^{-\theta_6}\right)
\\
&=\frac{\mathcal{F}(L_1)}{L_1}\hat{u}_1-\frac{\mathfrak{D}_2}{\mathfrak{D}_1}\frac{\mathcal{F}(L_1)}{L_1}\hat{u}_2-\frac{\mathcal{F}(L_1)}{\mathfrak{D}_1}\hat{u}_1+O\left(t^{-\theta_6}\right)
\\
&=\frac{\mathcal{F}(L_1)}{L_1}\left\{\left(1-\frac{L_1}{\mathfrak{D}_1}\right)\hat{u}_1-\frac{\mathfrak{D}_2}{\mathfrak{D}_1}\hat{u}_2\right\}+O\left(t^{-\theta_6}\right)
\\
&=\frac{\mathcal{F}(L_1)}{L_1}\frac{\mathfrak{D}_2}{\mathfrak{D}_1}\left(\hat{u}_1-\hat{u}_2\right)+O\left(t^{-\theta_6}\right),
\\
\dot{\hat{u}}_2&=\frac{\dot{\hat{z}}_2}{\mathfrak{D}_2}-\frac{\dot{\mathfrak{D}}_2\hat{z}_2}{\mathfrak{D}_2^2}
\\
&=-\frac{\mathcal{F}(L_1)}{L_1\mathfrak{D}_2}\hat{z}_1+\frac{1}{\mathfrak{D}_2}\left(\frac{\mathcal{F}(L_1)}{L_1}+\frac{2\mathcal{F}(L_2)}{L_2}\right)\hat{z}_2+\frac{\mathcal{F}(L_1)-\mathcal{F}(L_2)}{\mathfrak{D}_2^2}\hat{z}_2+O\left(t^{-\theta_6}\right)
\\
&=-\frac{\mathcal{F}(L_1)}{L_1}\frac{\mathfrak{D}_1}{\mathfrak{D}_2}\hat{u}_1+\frac{\mathcal{F}(L_1)}{L_1}\frac{\mathfrak{D}_1}{\mathfrak{D}_2}\hat{u}_2+O\left(t^{-\theta_6}\right)
\\
&=\frac{\mathcal{F}(L_1)}{L_1}\frac{\mathfrak{D}_1}{\mathfrak{D}_2}\left(\hat{u}_2-\hat{u}_1\right)+O\left(t^{-\theta_6}\right).
\end{align*}
Therefore, when we define $\mathcal{W}=\left|\hat{u}_1-\hat{u}_2\right|$, we have
\begin{align*}
\dot{\mathcal{W}}=\frac{\mathcal{F}(L_1)}{L_1}\left(\frac{\mathfrak{D}_1}{\mathfrak{D}_2}+\frac{\mathfrak{D}_2}{\mathfrak{D}_1}\right)\mathcal{W}+O\left(t^{-\theta_6}\right).
\end{align*}
Therefore, there exists $c>0$ such that for $t\gg 1$
\begin{align}\label{wlinees}
\dot{\mathcal{W}}\geq -ct^{-\theta_6}.
\end{align}
By \eqref{wlinees}, we have for $1\ll t<T$
\begin{align}\label{Wesline2}
\mathcal{W}(t)\leq \mathcal{W}(T)+c\int_t^Ts^{-\theta_6}ds<\mathcal{W}(T)+\frac{c}{\theta_6-1}t^{-\theta_6+1}.
\end{align}
By \eqref{hatusalim} and \eqref{Wesline2}, letting $T\to\infty$, we obtain
\begin{align}\label{Wlinees3}
\left|\hat{u}_1(t)-\hat{u}_2(t)\right|\lesssim t^{-\theta_6+1}.
\end{align}
By \eqref{Wlinees3}, we obtain
\begin{align*}
\dot{\hat{u}}_1=O\left(t^{-\theta_6}\right).
\end{align*}
Therefore, there exists $-\omega_{\infty}\in\mathbb{R}^d$ such that 
\begin{align}\label{hatues}
\left|\hat{u}_1(t)+\omega_{\infty}\right|\lesssim t^{-\theta_6+1}.
\end{align}
We note that by $\lim_{t\to\infty}|\hat{u}_1(t)|=1$, we obtain $\omega_{\infty}\in S^{d-1}$. Therefore, when we define $0<\theta_7<\theta_6-1$, by \eqref{taishouesline}, \eqref{Wlinees3}, and \eqref{hatues}, we obtain \eqref{1-typeasumpes}.
\end{proof}

We now use Lemma \ref{1-line-type-asymp} to describe the long-time behavior of line-type solutions.
\begin{lemma}\label{1-linekansei}
Let $\vec{u}$ be a line-type solution. Then, there exist $\sigma=\pm 1$, $z_1,z_2,z_3,z_4:[0,\infty)\to\mathbb{R}^d$ and $z_{\infty}\in\mathbb{R}^d$ and $\omega_{\infty}\in S^{d-1}$ such that 
\begin{align}\label{4linesoli2}
\left\| u(t)-\sigma \sum_{k=1}^4 (-1)^kQ(\cdot-z_k(t))\right\|_{H^1}+\|{\partial}_tu(t)\|_{L^2}\lesssim t^{-1},
\end{align}
\begin{align}\label{zestheorem}
\begin{aligned}
z_1(t)&=z_{\infty}-\left\{\frac{3}{2}\left(\log{t}-\frac{d-1}{2}\log{(\log{t})}+c_{\star}\right)-\log{3}+\frac{1}{2}\log{2} \right\}\omega_{\infty}+O\left(\frac{\log{(\log{t})}}{\log{t}}\right),
\\
z_2(t)&=z_{\infty}-\left\{\frac{1}{2}\left(\log{t}-\frac{d-1}{2}\log{(\log{t})}+c_{\star}\right)-\frac{1}{2}\log{2} \right\}\omega_{\infty}+O\left(\frac{\log{(\log{t})}}{\log{t}}\right),
\\
z_3(t)&=z_{\infty}+\left\{\frac{1}{2}\left(\log{t}-\frac{d-1}{2}\log{(\log{t})}+c_{\star}\right)-\frac{1}{2}\log{2} \right\}\omega_{\infty}+O\left(\frac{\log{(\log{t})}}{\log{t}}\right),
\\
z_4(t)&=z_{\infty}+\left\{\frac{3}{2}\left(\log{t}-\frac{d-1}{2}\log{(\log{t})}+c_{\star}\right)-\log{3}+\frac{1}{2}\log{2} \right\}\omega_{\infty}+O\left(\frac{\log{(\log{t})}}{\log{t}}\right).
\end{aligned}
\end{align}
\end{lemma}

\begin{proof}
By definition, $\vec{u}$ is an $i$-line-type solution for some $i\in\mathbb{Z}/4\mathbb{Z}$. After cyclically relabeling the centers and absorbing the resulting parity change into the global sign $\sigma$, we may assume that $\vec{u}$ is a 1-line-type solution. By Lemma \ref{epzes} and $\mathcal{F}(D(t))\sim t^{-1}$, we obtain \eqref{4linesoli2}. Moreover, Lemma \ref{1-line-type-asymp} gives, for some $\omega_\infty\in S^{d-1}$,

\begin{align*}
z_1&=z_\infty-\left(L_1+\frac{1}{2}L_2\right)\omega_\infty+O(t^{-\theta_7}),
\\
z_2&=z_\infty-\frac{1}{2}L_2\omega_\infty+O(t^{-\theta_7}),
\\
z_3&=z_\infty+\frac{1}{2}L_2\omega_\infty+O(t^{-\theta_7}),
\\
z_4&=z_\infty+\left(L_1+\frac{1}{2}L_2\right)\omega_\infty+O(t^{-\theta_7}).
\end{align*}
In addition, by Lemma \ref{Liasymptotics} and \eqref{Lnagasa}, we obtain \eqref{zestheorem}.

\end{proof}

\section{Long-time behavior of rhombus-type solutions}

In this section,  we determine the long-time behavior of rhombus-type solutions.

\subsection{Estimates arising from the rhombus geometry}

In this subsection, we collect several geometric estimates that follow directly from the assumption that $\vec{u}$ is a rhombus-type solution. First, we define
\begin{align}\label{sadef}
\Delta_{\rho}(t)=\max_{i=1,2,3,4}\rho_i(t)-\min_{i=1,2,3,4}\rho_i(t).
\end{align}

\begin{lemma}\label{hishinaisekilem1}
Let $\vec{u}$ be a rhombus-type solution. Then, we have for $t\gg 1$
\begin{align}\label{naisekirhombes}
\left|v_1\cdot v_2-v_3\cdot v_4\right|+\left|v_1\cdot v_4-v_2\cdot v_3\right|\lesssim \frac{\Delta_{\rho}}{\log{t}}.
\end{align}
Furthermore, we have for $t\gg 1$ and $i\in\mathbb{Z}/4\mathbb{Z}$,
\begin{align}\label{naisekirhombes2}
\begin{aligned}
\left|w_i\cdot v_i-w_i\cdot v_{i+1}\right|+\left|w_i\cdot v_i+w_i\cdot v_{i+2}\right|+\left|w_i\cdot v_i+w_i\cdot v_{i+3}\right|&\lesssim \frac{\Delta_{\rho}}{\log{t}},
\\
\left|w_i\cdot v_{i+1}+w_i\cdot v_{i+2}\right|+\left|w_i\cdot v_{i+1}+w_i\cdot v_{i+3}\right|+\left|w_i\cdot v_{i+2}-w_i\cdot v_{i+3}\right|&\lesssim \frac{\Delta_{\rho}}{\log{t}}.
\end{aligned}
\end{align}
\end{lemma}

\begin{proof}
By direct computation, we have
\begin{align*}
\left|z_1-z_3\right|^2&=\left|\rho_1v_1+\rho_2v_2\right|^2=\rho_1^2+\rho_2^2+2\rho_1\rho_2v_1\cdot v_2
\\
&=\left|\rho_3v_3+\rho_4v_4\right|^2=\rho_3^2+\rho_4^2+2\rho_3\rho_4v_3\cdot v_4.
\end{align*}
Therefore, we obtain
\begin{align*}
\left|v_1\cdot v_2-v_3\cdot v_4\right|&=\left|\frac{|z_1-z_3|^2-\rho_1^2-\rho_2^2}{2\rho_1\rho_2}-\frac{|z_1-z_3|^2-\rho_3^2-\rho_4^2}{2\rho_3\rho_4}\right|
\\
&\lesssim \frac{\left|\rho_3\rho_4-\rho_1\rho_2\right|\left|z_1-z_3\right|^2}{\rho_1\rho_2\rho_3\rho_4}+\left|\frac{\rho_3}{\rho_4}+\frac{\rho_4}{\rho_3}-\frac{\rho_1}{\rho_2}-\frac{\rho_2}{\rho_1}\right|
\\
&\lesssim \frac{\Delta_\rho}{\log{t}}.
\end{align*}
By the same argument, we obtain $\left|v_1\cdot v_4-v_2\cdot v_3\right|\lesssim \frac{\Delta_{\rho}}{\log{t}}$, and we obtain \eqref{naisekirhombes}. To prove \eqref{naisekirhombes2}, a direct computation gives, for every
$i\in\mathbb{Z}/4\mathbb{Z}$,
\begin{align*}
w_i\cdot v_i&=\frac{R_i^2+\rho_i^2-\rho_{i+1}^2}{2R_i\rho_i},
\\
w_i\cdot v_{i+1}&=\frac{R_i^2+\rho_{i+1}^2-\rho_i^2}{2R_i\rho_{i+1}},
\\
w_i\cdot v_{i+2}&=-\frac{R_i^2+\rho_{i+2}^2-\rho_{i+3}^2}{2R_i\rho_{i+2}},
\\
w_i\cdot v_{i+3}&=-\frac{R_i^2+\rho_{i+3}^2-\rho_{i+2}^2}{2R_i\rho_{i+3}}.
\end{align*}
Since $\rho_j\sim R_k\sim\log t$ and $|\rho_j-\rho_k|\leq\Delta_{\rho}$ for all $j,k\in\mathbb{Z}/4\mathbb{Z}$, estimate \eqref{naisekirhombes2} follows from these identities.
\end{proof}

\subsection{Asymptotic coplanarity of the four centers}

In this subsection, we establish the asymptotic coplanarity of the four centers $z_1,z_2,z_3,z_4$. We begin by estimating $\Delta_{\rho}$.

\begin{lemma}\label{henchourhomblem}
Let $\vec{u}$ be a rhombus-type solution. Then, there exists $\theta_8>0$ such that 
\begin{align}\label{deltarhoes}
\Delta_{\rho}(t)\lesssim t^{-\theta_8}.
\end{align}
\end{lemma}
\begin{proof}
We note that by Proposition \ref{samesolitonhanare}, we have for $t\gg 1$
\begin{align}\label{naisekidouhugoujouken}
\max{(|v_1\cdot v_2|,|v_2\cdot v_3|, |v_3\cdot v_4|, |v_4\cdot v_1|)}<\frac{2}{3}.
\end{align}
Fix $1<\theta<\min\{p-1,2\}$. In view of  Lemma \ref{2,2centerdynamics2}, Lemma \ref{Fsimt-1lem}, and
\eqref{naisekidouhugoujouken}, a direct computation yields
\begin{align}\label{rho2-rho1es1}
\begin{aligned}
\dot{\rho}_2-\dot{\rho}_1&=\left(2+v_1\cdot v_2\right)\left(\mathcal{F}(\rho_2)-\mathcal{F}(\rho_1)\right)-\mathcal{F}(\rho_3)v_2\cdot v_3+\mathcal{F}(\rho_4)v_1\cdot v_4
\\
&\quad+\mathcal{F}(R_1)\left(v_2\cdot w_3+v_1\cdot w_1\right)-\mathcal{F}(R_2)\left(w_2\cdot v_2+v_1\cdot w_2\right)+O\left(t^{-\theta}\right)
\\
&=\left(2+v_1\cdot v_2\right)\left(\mathcal{F}(\rho_2)-\mathcal{F}(\rho_1)\right)+v_1\cdot v_4\left(\mathcal{F}(\rho_4)-\mathcal{F}(\rho_3)\right)
\\
&\quad+\mathcal{F}(\rho_3)(v_1\cdot v_4-v_2\cdot v_3)+\mathcal{F}(R_1)\left(w_1\cdot v_1-w_1\cdot v_2\right)
\\
&\quad -\mathcal{F}(R_2)\left(w_2\cdot v_1+w_2\cdot v_2\right)+O\left(t^{-\theta}\right),
\end{aligned}
\end{align}
\begin{align}\label{rho3-rho1es1}
\begin{aligned}
\dot{\rho}_3-\dot{\rho}_1&=2\left(\mathcal{F}(\rho_3)-\mathcal{F}(\rho_1)\right)+\left(\mathcal{F}(\rho_2)v_1\cdot v_2-\mathcal{F}(\rho_4)v_3\cdot v_4\right)
\\
&\quad+\left(\mathcal{F}(\rho_4)v_1\cdot v_4-\mathcal{F}(\rho_2)v_2\cdot v_3\right)+\mathcal{F}(R_1)\left(v_1\cdot w_1-v_3\cdot w_3\right)
\\
&\quad +\mathcal{F}(R_2)\left(v_3\cdot w_4-v_1\cdot w_2\right)+O\left(t^{-\theta}\right)
\\
&=2\left(\mathcal{F}(\rho_3)-\mathcal{F}(\rho_1)\right)+v_1\cdot v_2\left(\mathcal{F}(\rho_2)-\mathcal{F}(\rho_4)\right)
\\
&\quad+v_1\cdot v_4\left(\mathcal{F}(\rho_4)-\mathcal{F}(\rho_2)\right)+\mathcal{F}(\rho_4)\left(v_1\cdot v_2-v_3\cdot v_4\right)
\\
&\quad+\mathcal{F}(\rho_2)\left(v_1\cdot v_4-v_2\cdot v_3\right)+\mathcal{F}(R_1)\left(w_1\cdot v_1+w_1\cdot v_3\right)
\\
&\quad-\mathcal{F}(R_2)\left(w_2\cdot v_1+w_2\cdot v_3\right)+O\left(t^{-\theta}\right),
\end{aligned}
\end{align}
\begin{align}\label{rho4-rho1es1}
\begin{aligned}
\dot{\rho}_4-\dot{\rho}_1&=\left(2+v_1\cdot v_4\right)\left(\mathcal{F}(\rho_4)-\mathcal{F}(\rho_1)\right)-\mathcal{F}(\rho_3)v_3\cdot v_4+\mathcal{F}(\rho_2)v_1\cdot v_2
\\
&\quad+\mathcal{F}(R_1)\left(v_4\cdot w_1+v_1\cdot w_1\right)-\mathcal{F}(R_2)\left(w_4\cdot v_4+v_1\cdot w_2\right)+O\left(t^{-\theta}\right)
\\
&=\left(2+v_1\cdot v_4\right)\left(\mathcal{F}(\rho_4)-\mathcal{F}(\rho_1)\right)+v_1\cdot v_2\left(\mathcal{F}(\rho_2)-\mathcal{F}(\rho_3)\right)
\\
&\quad+\mathcal{F}(\rho_3)\left(v_1\cdot v_2-v_3\cdot v_4\right)+\mathcal{F}(R_1)\left(w_1\cdot v_1+w_1\cdot v_4\right)
\\
&\quad-\mathcal{F}(R_2)\left(w_2\cdot v_1-w_2\cdot v_4\right)+O\left(t^{-\theta}\right).
\end{aligned}
\end{align}

First, we prove $|\Delta_{\rho}|\lesssim 1$. We choose $T\gg 1$ and we introduce the bootstrap estimate 
\begin{align}\label{Deltarhonobs}
\Delta_{\rho}(t)\leq \max{(\Delta_{\rho}(T),\log{2})}.
\end{align}
We define $T_1\in [T,\infty]$ as 
\begin{align}\label{rho-Lyuukaibs2}
T_1=\sup{\{t\in[T,\infty)\ \mbox{such that}\ \eqref{Deltarhonobs}\ \mbox{holds on}\ [T,t] \}}.
\end{align}
In addition, we assume $T_1<\infty$. Then, we have at $t=T_1$, $\Delta_{\rho}(T_1)=\max{(\Delta_{\rho}(T),\log{2})}$. First, assume that
\begin{align*}
\rho_1=\min_{1\leq i\leq4}\rho_i.
\end{align*}
In view of Lemma \ref{hishinaisekilem1}, we distinguish the following three cases according to which of $\rho_2,\rho_3,\rho_4$ attains the maximum.

Case 1: $\rho_2=\max_{i=1,2,3,4}\rho_i$.

Then, by \eqref{rho2-rho1es1}, we have at $t=T_1$,
\begin{align*}
\dot{\rho}_2-\dot{\rho}_1\leq \left(2-\frac{2}{3}\right)\left(\mathcal{F}(\rho_2)-\mathcal{F}(\rho_1)\right)+\frac{2}{3}\left(\mathcal{F}(\rho_1)-\mathcal{F}(\rho_2)\right)+o\left(\mathcal{F}(D)\right)<0.
\end{align*}

Case 2: $\rho_3=\max_{i=1,2,3,4}\rho_i$.

Then, by \eqref{rho3-rho1es1}, we have at $t=T_1$,
\begin{align*}
\dot{\rho}_3-\dot{\rho}_1\leq 2\left(\mathcal{F}(\rho_3)-\mathcal{F}(\rho_1)\right)+\frac{4}{3}\left(\mathcal{F}(\rho_1)-\mathcal{F}(\rho_3)\right)+o\left(\mathcal{F}(D)\right)<0.
\end{align*}

Case 3: $\rho_4=\max_{i=1,2,3,4}\rho_i$.

Then, by \eqref{rho4-rho1es1}, we have at $t=T_1$,
\begin{align*}
\dot{\rho}_4-\dot{\rho}_1\leq \left(2-\frac{2}{3}\right)\left(\mathcal{F}(\rho_4)-\mathcal{F}(\rho_1)\right)+\frac{2}{3}\left(\mathcal{F}(\rho_1)-\mathcal{F}(\rho_4)\right)+o\left(\mathcal{F}(D)\right)<0.
\end{align*}

Consequently, in each of the three cases, if
\begin{align*}
\rho_1(T_1)=\min_{1\leq i\leq4}\rho_i(T_1),
\end{align*}
then we obtain
\begin{align*}
\dot{\Delta}_{\rho}(T_1)<0.
\end{align*}
By cyclic symmetry, the same conclusion holds when the minimum is attained by any of $\rho_2,\rho_3,\rho_4$. This contradicts the maximality of $T_1$. Hence, $T_1=\infty$, and therefore we have
\begin{align*}
\Delta_{\rho}(t)\lesssim1.
\end{align*}
Moreover, the same computations as in \eqref{rho2-rho1es1}--\eqref{rho4-rho1es1}, together with Lemma \ref{hishinaisekilem1}, show that whenever
\begin{align*}
\rho_j(t)=\min_{1\leq i\leq4}\rho_i(t),\ \rho_k(t)=\max_{1\leq i\leq4}\rho_i(t),
\end{align*}
we have
\begin{align*}
\dot{\rho}_k-\dot{\rho}_j\leq\frac{1}{3}\left(\mathcal{F}(\rho_k)-\mathcal{F}(\rho_j)\right)+O\left(t^{-\theta}\right).
\end{align*}
Thus, there exists $c>0$ such that 
\begin{align*}
\dot{\Delta}_{\rho}\leq -\frac{c\Delta_{\rho}}{t}+\frac{1}{c}t^{-\theta}.
\end{align*}
Choosing $0<\theta_8<\min{(\theta-1,c)}$ and applying Gronwall's inequality, we obtain
\begin{align*}
\Delta_{\rho}(t)\lesssim t^{-\theta_8}.
\end{align*}
Therefore, we complete the proof.

\end{proof}

We next use Lemma \ref{henchourhomblem} to obtain a sharper estimate for $D$.

\begin{lemma}\label{Dsharprhomblem}
Let $\vec{u}$ be a rhombus-type solution. Then, we have
\begin{align}\label{Dsharpes1}
\begin{aligned}
\mathcal{F}(D(t))&=\frac{1}{2t}+o\left(\frac{1}{t}\right),
\\
D(t)&=\log{t}-\frac{d-1}{2}\log{(\log{t})}+c_{\star}+\log{2}+o(1).
\end{aligned}
\end{align}
\end{lemma}
\begin{proof}
By \eqref{hishidefes}, we have
\begin{align*}
v_1\cdot v_2+v_1\cdot v_4=o(1).
\end{align*}
Furthermore, by Lemma \ref{2,2centerdynamics2} and Lemma \ref{henchourhomblem}, we obtain
\begin{align}\label{rho1odesec6ver}
\dot{\rho}_1=\left(2+o(1)\right)\mathcal{F}(\rho_1).
\end{align}
Since $|D-\rho_1|\leq\Delta_\rho=o(1)$ by Lemma \ref{henchourhomblem}, Lemma \ref{mathcalFyouode} applied to \eqref{rho1odesec6ver} yields
\begin{align*}
D(t)=\log{t}-\frac{d-1}{2}\log{(\log{t})}+c_{\star}+\log{2}+o(1).
\end{align*}
Then, we obtain 
\begin{align*}
\mathcal{F}(D(t))&=\frac{1}{2t}+o\left(\frac{1}{t}\right),
\end{align*}
and hence we obtain \eqref{Dsharpes1}.
\end{proof}

Next, using Lemma \ref{henchourhomblem} and Lemma \ref{Dsharprhomblem}, we derive a decay estimate for $z_1+z_3-z_2-z_4$.

\begin{lemma}\label{naisekirhomblem2}
Let $\vec{u}$ be a rhombus-type solution. Then, there exists $\theta_9>0$ such that 
\begin{align}\label{heimenseies1}
\left|z_1(t)+z_3(t)-z_2(t)-z_4(t)\right|\lesssim t^{-\theta_9}.
\end{align}
\end{lemma}
\begin{proof}
First, by Lemma \ref{henchourhomblem}, we have for $i=1,2,3,4$,
\begin{align}\label{henchouchousei1}
\left|\frac{\mathcal{F}(\rho_i)}{\rho_i}-\frac{\mathcal{F}(D)}{D}\right|\lesssim t^{-\theta_8-1}.
\end{align}
We define $\mathcal{W}$ as 
\begin{align*}
\mathcal{W}(t)=\left|z_1(t)+z_3(t)-z_2(t)-z_4(t)\right|.
\end{align*}
Then, by Lemma \ref{2,2centerdynamics2} and \eqref{henchouchousei1}, we have
\begin{align*}
\dot{z}_1&=\frac{\mathcal{F}(D)}{D}\left(2z_1-z_2-z_4\right)+\mathcal{F}(R_1)w_1+O\left(t^{-\theta_{\circ}}\right),
\\
\dot{z}_2&=\frac{\mathcal{F}(D)}{D}\left(2z_2-z_1-z_3\right)+\mathcal{F}(R_2)w_2+O\left(t^{-\theta_{\circ}}\right),
\\
\dot{z}_3&=\frac{\mathcal{F}(D)}{D}\left(2z_3-z_2-z_4\right)-\mathcal{F}(R_1)w_1+O\left(t^{-\theta_{\circ}}\right),
\\
\dot{z}_4&=\frac{\mathcal{F}(D)}{D}\left(2z_4-z_1-z_3\right)-\mathcal{F}(R_2)w_2+O\left(t^{-\theta_{\circ}}\right),
\end{align*}
where $\theta_{\circ}=1+\frac{\theta_8}{2}$. Then, we have
\begin{align}\label{mathcalWbibun}
\dot{\mathcal{W}}&=\frac{4\mathcal{F}(D)}{D}\mathcal{W}+O\left(t^{-\theta_{\circ}}\right).
\end{align}
Furthermore, by \eqref{Dsharpes1}, there exist $T\gg 1$ and $C>0$ such that for $t>T$
\begin{align}\label{Wodeupper}
\dot{\mathcal{W}}\geq \frac{3}{2t\log{t}}\mathcal{W}-Ct^{-\theta_{\circ}}.
\end{align}
Here, we define 
\begin{align*}
\mathcal{U}(t)=\frac{\mathcal{W}(t)}{\log{t}}.
\end{align*}
By \eqref{mathcalLsim},  \eqref{hishidefes} and \eqref{Dsharpes1}, we have $\lim_{t\to\infty} \mathcal{U}(t)=0$. Furthermore, by \eqref{Wodeupper}, we have for $t>T$
\begin{align}\label{Vupper}
\dot{\mathcal{U}}=\frac{1}{\log{t}}\left(\dot{\mathcal{W}}-\frac{1}{t\log{t}}\mathcal{W}\right)\geq \frac{1}{2t(\log{t})^2}\mathcal{W}-\frac{Ct^{-\theta_{\circ}}}{\log{t}}>-Ct^{-\theta_{\circ}}.
\end{align}
Then, we have for $T<t<t^{\prime}$,
\begin{align}\label{mathcalVhyoukayou}
\mathcal{U}(t)\leq \mathcal{U}(t^{\prime})+C\int_t^{t^{\prime}}s^{-\theta_{\circ}}ds<\mathcal{U}(t^{\prime})+\frac{C}{\theta_{\circ}-1}t^{-\theta_{\circ}+1}.
\end{align}
Letting $t^{\prime}\to\infty$ in \eqref{mathcalVhyoukayou}, we obtain
\begin{align*}
\mathcal{U}(t)\lesssim t^{-\theta_{\circ}+1},
\end{align*}
and therefore we obtain
\begin{align*}
\mathcal{W}(t)\lesssim (\log{t})t^{-\theta_{\circ}+1}\lesssim  t^{-\frac{\theta_{\circ}-1}{2}}.
\end{align*}
Therefore, we obtain \eqref{heimenseies1} by choosing $\theta_9=\frac{\theta_{\circ}-1}{2}$.
\end{proof}

\subsection{Effective equations for the centers of rhombus-type solutions}

In this subsection, we refine the equations governing the soliton centers. We define $\hat{z}_i, \hat{\rho}_i, \hat{u}_i$ for $i=1,2,3,4$, as
\begin{align*}
\hat{z}_i&=z_i-z_{\infty},
\\
\hat{\rho}_i&=\left|\hat{z}_i\right|,
\\
\hat{u}_i&=\frac{\hat{z}_i}{\hat{\rho}_i}.
\end{align*}
These quantities satisfy the following estimates.

\begin{lemma}\label{rhombodekai}
Let $\vec{u}$ be a rhombus-type solution. Then, there exist $\theta_{10}>0$ and  $\theta_{11}>1$ such that, for $i=1,2,3,4$ and $t\gg 1$,
\begin{align}\label{hayzyoujunbi}
\begin{aligned}
\left|\hat{z}_1+\hat{z}_3\right|&\lesssim t^{-\theta_{10}},
\\
\left|\hat{z}_2+\hat{z}_4\right|&\lesssim t^{-\theta_{10}},
\end{aligned}
\end{align}
\begin{align}\label{hatzioderhomb}
\begin{aligned}
\dot{\hat{z}}_i&=\left(\frac{2\mathcal{F}(D)}{D}-\frac{\mathcal{F}(2\hat{\rho}_i)}{\hat{\rho}_i}\right)\hat{z}_i+O\left(t^{-\theta_{11}}\right),
\\
\dot{\hat{\rho}}_i&=\left(\frac{2\mathcal{F}(D)}{D}-\frac{\mathcal{F}(2\hat{\rho}_i)}{\hat{\rho}_i}\right)\hat{\rho}_i+O\left(t^{-\theta_{11}}\right),
\end{aligned}
\end{align}
\end{lemma}
\begin{proof}
By Lemma \ref{zgshusoku}, Lemma \ref{henchourhomblem} and Lemma \ref{naisekirhomblem2}, there exists $\theta_{\circ}>1$ such that
\begin{align}\label{hayzyoujunbiproveyou}
\begin{aligned}
\left|\frac{\mathcal{F}(\rho_i)}{\rho_i}-\frac{\mathcal{F}(D)}{D}\right|&\lesssim t^{-\theta_{\circ}},
\\
\left|\hat{z}_1+\hat{z}_3\right|&\lesssim t^{-\theta_{\circ}+1},
\\
\left|\hat{z}_2+\hat{z}_4\right|&\lesssim t^{-\theta_{\circ}+1}.
\end{aligned}
\end{align}
By the last two estimates in \eqref{hayzyoujunbiproveyou}, the definition of $R_iw_i$, we have
\begin{align}\label{a}
\begin{aligned}
\left|2\hat{z}_i+R_iw_i\right|\lesssim t^{-\theta_{\circ}+1},
\\
\left|\frac{\mathcal{F}(2\hat{\rho}_i)}{2\hat{\rho}_i}-\frac{\mathcal{F}(R_i)}{R_i}\right|\lesssim t^{-\theta_{\circ}}.
\end{aligned}
\end{align}
Fix $1<\theta_{11}<\theta_{\bullet}<\theta_{\circ}$. Then, by Lemma \ref{2,2centerdynamics2} we have for $t\gg1$ and $i=1,2,3,4$,
\begin{align*}
\dot{\hat{z}}_i=\dot{z}_i&=-\mathcal{F}(\rho_i)v_i+\mathcal{F}(\rho_{i-1})v_{i-1}+\mathcal{F}(R_i)w_i+O\left(t^{-\theta_{\circ}}\right)
\\
&=\frac{\mathcal{F}(D)}{D}\left(2z_i-z_{i+1}-z_{i-1}\right)-\frac{2\mathcal{F}(R_i)}{R_i}\hat{z}_i+O\left(t^{-\theta_{\bullet}}\right)
\\
&=\left(\frac{2\mathcal{F}(D)}{D}-\frac{\mathcal{F}(2\hat{\rho}_i)}{\hat{\rho}_i}\right)\hat{z}_i+O\left(t^{-\theta_{11}}\right),
\\
\dot{\hat{\rho}}_i&=\dot{\hat{z}}_i\cdot \hat{u}_i
\\
&=\left(\frac{2\mathcal{F}(D)}{D}-\frac{\mathcal{F}(2\hat{\rho}_i)}{\hat{\rho}_i}\right)\hat{\rho}_i+O\left(t^{-\theta_{11}}\right).
\end{align*}
Thus, when we define $\theta_{10}=\theta_{11}-1$, we complete the proof.
\end{proof}

We next use Lemma \ref{rhombodekai} to derive estimates for the four centers $z_1,z_2,z_3,z_4$.

\begin{lemma}\label{houkourhombus}
Let $\vec{u}$ be a rhombus-type solution. Then, there exist $\theta_{12}>0$ and  $\omega_1,\omega_2\in S^{d-1}$ such that 
\begin{align}\label{houkoues1}
\begin{aligned}
&\omega_1\cdot \omega_2=0,
\\
&\hat{z}_1=\omega_1\hat{\rho}_1+O\left(t^{-\theta_{12}}\right),
\\
&\hat{z}_2=\omega_2\hat{\rho}_2+O\left(t^{-\theta_{12}}\right),
\\
&\hat{z}_3=-\omega_1\hat{\rho}_1+O\left(t^{-\theta_{12}}\right),
\\
&\hat{z}_4=-\omega_2\hat{\rho}_2+O\left(t^{-\theta_{12}}\right).
\end{aligned}
\end{align}
\end{lemma}

\begin{proof}
By Lemma \ref{rhombodekai}, we have for $i=1,2,3,4$,
\begin{align*}
\dot{\hat{u}}_i&=\frac{\dot{\hat{z}}_i}{\hat{\rho}_i}-\frac{\dot{\hat{\rho}}_i\hat{z}_i}{\hat{\rho}_i^2}=O\left(t^{-\theta_{11}}\right).
\end{align*}
Therefore, there exists $\hat{u}_{\infty,i}\in S^{d-1}$ such that
\begin{align*}
\left|\hat{u}_i-\hat{u}_{\infty,i}\right|\lesssim  \int_t^{\infty}s^{-\theta_{11}}ds\lesssim t^{-\theta_{11}+1},
\end{align*}
and we obtain
\begin{align}\label{hatzes1}
\hat{z}_i=\hat{u}_{\infty,i}\hat{\rho}_i+O\left(t^{-\theta_{\circ}}\right),
\end{align}
where $0<\theta_{\circ}<\theta_{11}-1$. Then, by \eqref{mathcalLsim}, \eqref{hishidefes}, and \eqref{Dsharpes1}, we have for $i\in\mathbb{Z}/4\mathbb{Z}$, 
\begin{align}\label{hishijushinkara}
\hat{u}_{\infty,i}\cdot \hat{u}_{\infty,i+1}=\lim_{t\to\infty} \hat{u}_i(t)\cdot \hat{u}_{i+1}(t)=0.
\end{align}
By \eqref{hayzyoujunbi}, \eqref{hatzes1}, and \eqref{hishijushinkara}, when we define
\begin{align*}
0<\theta_{12}<\min{(\theta_{\circ},\theta_{10})},\ \omega_1=\hat{u}_{\infty,1},\ \omega_2=\hat{u}_{\infty,2},
\end{align*}
we obtain \eqref{houkoues1}. Thus, we complete the proof.
\end{proof}

\subsection{Precise asymptotic geometry of the four centers}

In this subsection, we determine the precise asymptotic geometry of the four centers $z_1,z_2,z_3,z_4$. We first show that the ratio $\hat{\rho}_2/\hat{\rho}_1$ converges as $t\to\infty$.

\begin{lemma}\label{hishinagasalem2}
Let $\vec{u}$ be a rhombus-type solution. Then, there exists $\frac{\pi}{6}\leq \varphi\leq \frac{\pi}{3}$ such that
\begin{align}\label{hoshinohiws1}
\lim_{t\to\infty}\frac{\hat{\rho}_2(t)}{\hat{\rho}_1(t)}=\tan{\varphi}.
\end{align}
\end{lemma}
\begin{proof}
We define
\begin{align}\label{mathcalVdef}
\mathcal{V}(t)=\sqrt{ \left(\log{\frac{\hat{\rho}_2(t)}{\hat{\rho}_1(t)}}\right)^2+1}.
\end{align}
We note that by Proposition \ref{samesolitonhanare} and \eqref{houkoues1}, we have
\begin{align}\label{mathcalVyuukai}
\mathcal{V}(t)\lesssim 1.
\end{align}
Furthermore, by Lemma \ref{rhombodekai}, we have
\begin{align*}
\dot{\mathcal{V}}=\frac{1}{\sqrt{ \left(\log{\frac{\hat{\rho}_2(t)}{\hat{\rho}_1(t)}}\right)^2+1}}\left(\log{\frac{\hat{\rho}_2(t)}{\hat{\rho}_1(t)}} \right)\left(\frac{\mathcal{F}(2\hat{\rho}_1)}{\hat{\rho}_1}-\frac{\mathcal{F}(2\hat{\rho}_2)}{\hat{\rho}_2}\right)+O\left(t^{-\theta_{11}}\right).
\end{align*}
Moreover, since the function $\frac{\mathcal{F}(2r)}{r}$ is decreasing for all sufficiently large $r$, we have
\begin{align*}
\left(\log{\frac{\hat{\rho}_2(t)}{\hat{\rho}_1(t)}} \right)\left(\frac{\mathcal{F}(2\hat{\rho}_1)}{\hat{\rho}_1}-\frac{\mathcal{F}(2\hat{\rho}_2)}{\hat{\rho}_2}\right)\geq 0.
\end{align*}
Thus, there exists $C>0$ such that
\begin{align}\label{mathcalVbibunhutoushiki}
\dot{\mathcal{V}}\geq -Ct^{-\theta_{11}}.
\end{align}
Here, we define 
\begin{align*}
\tilde{\mathcal{V}}(t)=\mathcal{V}(t)-\frac{C}{\theta_{11}-1}t^{-\theta_{11}+1}.
\end{align*}
Then, by \eqref{mathcalVyuukai}, we have
\begin{align}\label{tildeVyuukai}
\left|\tilde{\mathcal{V}}(t)\right|\lesssim 1.
\end{align}
In addition, by \eqref{mathcalVbibunhutoushiki}, we obtain
\begin{align}\label{tildeVtanchou}
\dot{\tilde{\mathcal{V}}}=\dot{\mathcal{V}}+Ct^{-\theta_{11}}\geq 0. 
\end{align}
By \eqref{tildeVyuukai} and \eqref{tildeVtanchou}, $\tilde{\mathcal{V}}$ converges as $t\to\infty$, and hence $\mathcal{V}$ converges too. Therefore, there exists $\mathcal{V}_{\infty}$ such that
\begin{align*}
\lim_{t\to\infty}\mathcal{V}(t)=\lim_{t\to\infty}\sqrt{ \left(\log{\frac{\hat{\rho}_2(t)}{\hat{\rho}_1(t)}}\right)^2+1}=\mathcal{V}_{\infty}.
\end{align*}
Since $\hat{\rho}_1$ and $\hat{\rho}_2$ are continuous, there exists $\ell_{\infty}\in \mathbb{R}$ such that 
\begin{align}
\lim_{t\to\infty}\frac{\hat{\rho}_2(t)}{\hat{\rho}_1(t)}=\ell_{\infty}.
\end{align}
We note that by \eqref{taikakusengomi} and \eqref{houkoues1}, we have $\frac{1}{\sqrt{3}}\leq \ell_{\infty}\leq \sqrt{3}$. Define $\varphi\in[\frac{\pi}{6},\frac{\pi}{3}]$ as
\begin{align*}
\tan{\varphi}=\ell_{\infty}.
\end{align*}
Then, we obtain \eqref{hoshinohiws1}.
\end{proof}

\begin{remark}
For each rhombus-type solution $\vec{u}$, we denote by $\varphi=\varphi(\vec{u})\in[\frac{\pi}{6},\frac{\pi}{3}]$
the unique parameter determined by \eqref{hoshinohiws1}. We use this notation throughout the remainder of the paper.
\end{remark}

\subsection{Long-time behavior of rhombus-type solutions in the interior regime}

By Lemma \ref{houkourhombus} and Lemma \ref{hishinagasalem2}, the centers of every rhombus-type solution move away from a fixed point while asymptotically forming a rhombus. In this subsection, we determine their precise long-time behavior when the limiting ratio $\frac{\hat{\rho}_2}{\hat{\rho}_1}$ lies in the interior of the admissible range.

\begin{lemma}\label{hitantennagasaharplem}
Let $\vec{u}$ be a rhombus-type solution. Furthermore, we assume $\frac{\pi}{6}<\varphi<\frac{\pi}{3}$. Then, we have for $t\gg 1$
\begin{align}\label{hitantensharp}
\begin{aligned}
\hat{\rho}_1&=\cos{\varphi}\left\{\log{t}-\frac{d-1}{2}\log{(\log{t})}+c_{\star}+\log{2} \right\}+O\left(\frac{\log{(\log{t})}}{\log{t}}\right),
\\
\hat{\rho}_2&=\sin{\varphi}\left\{\log{t}-\frac{d-1}{2}\log{(\log{t})}+c_{\star}+\log{2} \right\}+O\left(\frac{\log{(\log{t})}}{\log{t}}\right).
\end{aligned}
\end{align}
\end{lemma}
\begin{proof}
When $\frac{\pi}{6}<\varphi<\frac{\pi}{3}$, there exists $\theta_{\circ}>1$ such that for $t\gg 1$,
\begin{align*}
\mathcal{F}(R_1)+\mathcal{F}(R_2)\lesssim t^{-\theta_{\circ}}.
\end{align*}
Then, by Lemma \ref{2,2centerdynamics2}, Lemma \ref{henchourhomblem}, and Lemma \ref{houkourhombus}, there exists $1<\theta_{\bullet}<\theta_{\circ}$ such that 
\begin{align}\label{rho1iiode}
\dot{\rho}_1=2\mathcal{F}(\rho_1)+O\left(t^{-\theta_{\bullet}}\right).
\end{align}
By \eqref{rho1iiode}, we have
\begin{align}\label{rho1nagasasharp}
\rho_1(t)=\log{t}-\frac{d-1}{2}\log{(\log{t})}+c_{\star}+\log{2}+O\left(\frac{\log{(\log{t})}}{\log{t}}\right).
\end{align}
By Lemma \ref{rhombodekai}, when we define $\mathfrak{V}=\log{\frac{\hat{\rho}_2}{\hat{\rho}_1}}$, there exists $1<\theta_{\heartsuit}$ such that 
\begin{align*}
\dot{\mathfrak{V}}=O\left(t^{-\theta_{\heartsuit}}\right),
\end{align*}
and hence we obtain
\begin{align}\label{kakudosharpconv}
\left|\frac{\hat{\rho}_2(t)}{\hat{\rho}_1(t)}-\tan{\varphi}\right|\lesssim t^{-\theta_{\heartsuit}+1}.
\end{align}
Furthermore, by Lemma \ref{houkourhombus}, there exists $\theta_{\diamond}>0$ such that 
\begin{align}\label{rhokeies1}
\left(\hat{\rho}_1\right)^2+\left(\hat{\rho}_2\right)^2=\rho_1^2+O\left(t^{-\theta_{\diamond}}\right).
\end{align}
By \eqref{kakudosharpconv} and \eqref{rhokeies1}, we obtain
\begin{align}\label{hatrho1keisan}
\begin{aligned}
\left(\hat{\rho}_1\right)^2-\left(\cos^2{\varphi}\right)\left(\rho_1\right)^2&=\left(1-\cos^2{\varphi}\right)\left(\hat{\rho}_1\right)^2-\cos^2\varphi\left(\hat{\rho}_2\right)^2+O\left(t^{-\theta_{\diamond}}\right)
\\
&=\left(\cos^2{\varphi}\right)\left(\hat{\rho}_1\right)^2\left\{\tan^2{\varphi}-\left(\frac{\hat{\rho}_2(t)}{\hat{\rho}_1(t)}\right)^2\right\}+O\left(t^{-\theta_{\diamond}}\right)
\\
&=O\left((\log{t})^2t^{-\theta_{\heartsuit}+1}+t^{-\theta_{\diamond}}\right).
\end{aligned}
\end{align}
By \eqref{hatrho1keisan}, we obtain
\begin{align}\label{hatrho1chantomotomeru}
\begin{aligned}
\left|\hat{\rho}_1-\left(\cos{\varphi}\right)\rho_1\right|&=\frac{\left|\left(\hat{\rho}_1\right)^2-\left(\cos^2{\varphi}\right)\left(\rho_1\right)^2\right|}{\hat{\rho}_1+\left(\cos{\varphi}\right)\rho_1}\lesssim \frac{(\log{t})^2t^{-\theta_{\heartsuit}+1}+t^{-\theta_{\diamond}}}{\log{t}}\lesssim t^{-\theta^{\prime}},
\end{aligned}
\end{align}
where $0<\theta^{\prime}<\min{(\theta_{\heartsuit}-1,\theta_{\diamond})}$. By the same argument, we have
\begin{align}\label{hatrho2chantomotomeru}
\left|\hat{\rho}_2-\left(\sin{\varphi}\right)\rho_1\right|\lesssim t^{-\theta^{\prime}}.
\end{align}
By \eqref{rho1nagasasharp}, \eqref{hatrho1chantomotomeru}, and \eqref{hatrho2chantomotomeru}, we obtain \eqref{hitantensharp}.

\end{proof}

\subsection{Asymptotic behavior in the endpoint regimes}

In this subsection, we determine the asymptotic behavior of $\hat{\rho}_1$ and $\hat{\rho}_2$ in the endpoint regimes. By symmetry,
it suffices to consider the case $\varphi=\frac{\pi}{3}$, since the case $\varphi=\frac{\pi}{6}$ follows by interchanging $\hat{\rho}_1$ and $\hat{\rho}_2$. To analyze the resulting two-dimensional system, we define
\begin{align}\label{abrdef}
\begin{aligned}
a(t)&=\hat{\rho}_1(t),
\\
b(t)&=\hat{\rho}_2(t),
\\
r(t)&=\sqrt{\hat{\rho}_1(t)^2+\hat{\rho}_2(t)^2}.
\end{aligned}
\end{align}
The effective equations for $a$, $b$, and $r$ can then be written as follows.
\begin{lemma}\label{endpointlem1}
Let $\vec{u}$ be a rhombus-type solution. Furthermore, we assume $\varphi=\frac{\pi}{3}$. Then,  there exists $\theta_{13}>1$ such that
\begin{align}\label{endpointode}
\begin{aligned}
\dot{a}&=\left(\frac{2\mathcal{F}(r)}{r}-\frac{\mathcal{F}(2a)}{a}\right)a+O\left(t^{-\theta_{13}}\right),
\\
\dot{b}&=\frac{2\mathcal{F}(r)}{r}b+O\left(t^{-\theta_{13}}\right),
\\
\dot{r}&=2\mathcal{F}(r)-\frac{a}{r}\mathcal{F}(2a)+O\left(t^{-\theta_{13}}\right).
\end{aligned}
\end{align}
\end{lemma}
\begin{proof}
By Lemma \ref{henchourhomblem} and Lemma \ref{houkourhombus}, there exists $\theta_{\circ}>0$ such that for $i=1,2,3,4$,
\begin{align*}
\left|r(t)-\rho_i(t)\right|\lesssim t^{-\theta_{\circ}}.
\end{align*}
By $\varphi=\frac{\pi}{3}$, we have for $t\gg 1$
\begin{align}\label{hatrhodekai}
\mathcal{F}(2\hat{\rho}_2)\lesssim t^{-\frac{3}{2}}.
\end{align}
Therefore, by Lemma \ref{rhombodekai} and \eqref{hatrhodekai}, choosing $\theta_{13}$ so that $1<\theta_{13}<\min{(\theta_{\circ}+1,\frac{3}{2})}$, 
we obtain the first two equations in \eqref{endpointode}. Then, by direct computation, we have
\begin{align*}
\dot{r}=\frac{a\dot{a}+b\dot{b}}{r}=2\mathcal{F}(r)-\frac{a}{r}\mathcal{F}(2a)+O\left(t^{-\theta_{13}}\right),
\end{align*}
and we complete the proof.
\end{proof}
We now analyze the system derived in Lemma \ref{endpointlem1}. We define 
\begin{align*}
x(t)&=2a(t)-r(t).
\end{align*}

\begin{lemma}\label{xodelem}
Let $\vec{u}$ be a rhombus-type solution. Furthermore, we assume $\varphi=\frac{\pi}{3}$. Then, there exists $\theta_{14}>1$ such that we have
\begin{align}\label{xode}
\dot{x}=\mathcal{F}(r)\left(\frac{2x}{r}-\frac{3}{2}e^{-x}\right)+O\left(\frac{xe^{-x}\mathcal{F}(r)}{r}+t^{-\theta_{14}}\right).
\end{align}
\end{lemma}
\begin{proof}
First, by Lemma \ref{endpointlem1}, we have \eqref{endpointode}. By Proposition \ref{samesolitonhanare} and Lemma \ref{houkourhombus}, $\varphi=\frac{\pi}{3}$, we have
\begin{align}\label{xpcondi}
\lim_{t\to\infty}x(t)=\infty,\ \lim_{t\to\infty}\frac{x(t)}{r(t)}=0.
\end{align}
Then, by \eqref{tukaeruF3}, we have for $t\gg 1$
\begin{align*}
\mathcal{F}(2a)=e^{-x}\mathcal{F}(r)+O\left(\frac{xe^{-x}}{r}\mathcal{F}(r)\right).
\end{align*}
Thus, we have
\begin{align*}
\dot{x}&=2\dot{a}-\dot{r}
\\
&=4\mathcal{F}(r)\frac{a}{r}-2\mathcal{F}(2a)-\left(2\mathcal{F}(r)-\mathcal{F}(2a)\frac{a}{r}\right)+O\left(t^{-\theta_{13}}\right)
\\
&=\frac{2x}{r}\mathcal{F}(r)-\left(\frac{3}{2}-\frac{x}{2r}\right)\mathcal{F}(2a)+O\left(t^{-\theta_{13}}\right)
\\
&=\mathcal{F}(r)\left(\frac{2x}{r}-\frac{3}{2}e^{-x}\right)+O\left(\frac{xe^{-x}\mathcal{F}(r)}{r}+t^{-\theta_{13}}\right).
\end{align*}
Therefore, when we define $\theta_{14}=\theta_{13}$, we obtain \eqref{xode}.

\end{proof}

To determine the asymptotic behavior of $a$, we first use Lemma \ref{xodelem} to derive a sharp estimate for $x$.

\begin{lemma}\label{xsharpeslem}
Let $\vec{u}$ be a rhombus-type solution. Furthermore, we assume $\varphi=\frac{\pi}{3}$. Then, we have
\begin{align}\label{xnagasa}
x(t)=\log{r(t)}-\log{(\log{r(t)})}+\log{\frac{3}{4}}+o(1).
\end{align}
\end{lemma}

\begin{proof}
By $\varphi=\frac{\pi}{3}$, we have
\begin{align}\label{xkatei}
\lim_{t\to\infty}x(t)=\infty,\ \lim_{t\to\infty}\frac{x(t)}{r(t)}=0.
\end{align}
Furthermore, by \eqref{tukaeruF3} and \eqref{xkatei}, we obtain
\begin{align}\label{F2aendpoint}
\mathcal{F}(2a)=e^{-x}\mathcal{F}(r)+O\left(\frac{xe^{-x}\mathcal{F}(r)}{r}\right).
\end{align}
Then, by \eqref{endpointode} and \eqref{F2aendpoint}, we obtain
\begin{align}\label{rdotendpoint}
\begin{aligned}
\dot{r}&=2\mathcal{F}(r)-\frac{a}{r}\mathcal{F}(2a)+O\left(t^{-\theta_{13}}\right)
\\
&=2\mathcal{F}(r)-\left(\frac{1}{2}+\frac{x}{2r}\right)\left(e^{-x}\mathcal{F}(r)+O\left(\frac{xe^{-x}\mathcal{F}(r)}{r}\right)\right)+O\left(t^{-\theta_{13}}\right)
\\
&=\left(2-\frac{e^{-x}}{2}\right)\mathcal{F}(r)+O\left(\frac{xe^{-x}\mathcal{F}(r)}{r}+t^{-\theta_{13}}\right).
\end{aligned}
\end{align}
Moreover, the estimates obtained above imply
\begin{align}\label{Frendpoint}
\mathcal{F}(r(t))\sim\frac{1}{t},\ r(t)\sim\log{t}.
\end{align}
Here, we define
\begin{align*}
\tilde{x}(t)&=\log{r(t)}-\log{(\log{r(t)})}+\log{\frac{3}{4}},
\\
y(t)&=x(t)-\tilde{x}(t).
\end{align*}
We note that by the definitions of $\tilde{x}$ and $y$, we have
\begin{align}\label{xtildeidentities}
\begin{aligned}
x&=\log{r}-\log{(\log{r})}+\log{\frac{3}{4}}+y,
\\
e^{-x}&=e^{-\tilde{x}}e^{-y}=\frac{4\log{r}}{3r}e^{-y}.
\end{aligned}
\end{align}
Then, by \eqref{rdotendpoint} and \eqref{Frendpoint}, we have
\begin{align}\label{tildexnobibun}
\dot{\tilde{x}}=\left(1-\frac{1}{\log{r}}\right)\frac{\dot{r}}{r}=\frac{2\mathcal{F}(r)}{r}+o\left(\frac{\mathcal{F}(r)}{r}\right).
\end{align}
By \eqref{xode}, \eqref{xtildeidentities}, and \eqref{tildexnobibun}, we obtain
\begin{align}\label{ydotendpoint}
\begin{aligned}
\dot{y}&=\frac{\mathcal{F}(r)}{r}\left(2x-\frac{3}{2}re^{-x}\right)-\frac{2\mathcal{F}(r)}{r}+o\left(\frac{\mathcal{F}(r)}{r}\right)
\\
&=\frac{2\mathcal{F}(r)}{r}\left\{\left(\log{r}-\log{(\log{r})}+\log{\frac{3}{4}}+y\right)-\frac{3}{4}r\left(\frac{4\log{r}}{3r}e^{-y}\right) \right\}
\\
&\quad-\frac{2\mathcal{F}(r)}{r}+o\left(\frac{\mathcal{F}(r)}{r}\right)
\\
&=\frac{2\mathcal{F}(r)}{r}\left\{\left(\log{r}\right)\left(1-e^{-y}\right)+y-\log{(\log{r})}+\log{\frac{3}{4}}-1\right\}+o\left(\frac{\mathcal{F}(r)}{r}\right).
\end{aligned}
\end{align}
Fix $0<\epsilon<1$. We first prove the upper bound for $y$. Suppose
that there exists $t^{\prime}\gg1$ such that
\begin{align*}
y(t^{\prime})\geq\epsilon.
\end{align*}
Then, we introduce the following bootstrap estimate
\begin{align}\label{yupperbootstrap}
y(t)\geq \epsilon.
\end{align}
We define $T_1\in [t^{\prime},\infty]$ by
\begin{align}\label{yupperexit}
T_1=\sup{\{t\in[t^{\prime},\infty)\ \mbox{such that}\ \eqref{yupperbootstrap}\ \mbox{holds on}\ [t^{\prime},t] \}}.
\end{align}
Furthermore, we assume $T_1<\infty$. Then, we have at $t=T_1$,
\begin{align*}
\dot{y}\geq\frac{2\mathcal{F}(r)}{r}\left\{\left(1-e^{-\epsilon}\right)\log r+y-\log(\log r)+O(1)\right\}\geq\frac{2\mathcal{F}(r)}{r}y.
\end{align*}
This contradicts the maximality of $T_1$. Hence, we conclude that $T_1=\infty$, and we have for $t>t^{\prime}$,
\begin{align}\label{ybibunhutoushikiue}
\dot{y}\geq\frac{2\mathcal{F}(r)}{r}y.
\end{align}
Then, by \eqref{rdotendpoint} and \eqref{ybibunhutoushikiue}, there exists $C_1>0$ such that, for every $t>t^{\prime}$,
\begin{align*}
\frac{d}{dt}\log\frac{y}{r}=\frac{\dot{y}}{y}-\frac{\dot{r}}{r}\geq-C_1t^{-\theta_{13}}.
\end{align*}
Since $\theta_{13}>1$, we have for $t^{\prime}<t$,
\begin{align*}
\log{\frac{y(t)}{r(t)}}-\log{\frac{y(t^{\prime})}{r(t^{\prime})}}\geq -C_1\int_{t^{\prime}}^ts^{-\theta_{13}}ds\geq -\frac{C_1}{\theta_{13}-1}\left(t^{\prime}\right)^{1-\theta_{13}},
\end{align*}
and there exists $C_2>0$ such that for all $t>t^{\prime}$
\begin{align}\label{yuemujun}
\frac{y(t)}{r(t)}>C_2.
\end{align}
By \eqref{yuemujun} and $\lim_{t\to\infty}\frac{\tilde{x}(t)}{r(t)}=0$, we have
\begin{align*}
\liminf_{t\to\infty}\frac{x(t)}{r(t)}>0,
\end{align*}
which contradicts \eqref{xkatei}. Therefore, we have for $t\gg 1$
\begin{align}\label{y<epsilon}
y(t)<\epsilon.
\end{align}
Next, we assume that there exists 
$t^{\prime}\gg1$ such that
\begin{align*}
y(t^{\prime})\leq-\epsilon.
\end{align*}
Then, we introduce the following bootstrap estimate
\begin{align}\label{rho-Lnotame2}
y(t)\leq -\epsilon.
\end{align}
We define $T_2\in [t^{\prime},\infty]$ by
\begin{align}\label{ylowerexit}
T_2=\sup{\{t\in[t^{\prime},\infty)\ \mbox{such that}\ \eqref{rho-Lnotame2}\ \mbox{holds on}\ [t^{\prime},t] \}}.
\end{align}
Furthermore, we assume $T_2<\infty$. Then, by \eqref{ydotendpoint}, we have at $t=T_2\gg 1$
\begin{align*}
\dot{y}&=\frac{2\mathcal{F}(r)}{r}\left\{\left(\log{r}\right)\left(1-e^{-y}\right)+y-\log{(\log{r})}+\log{\frac{3}{4}}-1\right\}+o\left(\frac{\mathcal{F}(r)}{r}\right)
\\
&\leq \frac{2\mathcal{F}(r)}{r}\left\{\left(\log{r}\right)\left(1-e^{\epsilon}\right)-\log{(\log{r})}+1 \right\}<0.
\end{align*}
This contradicts the maximality of $T_2$. Hence, we conclude that $T_2=\infty$. Then, we have for $t\gg t^{\prime}$,
\begin{align*}
\dot{x}&=\mathcal{F}(r)\left(\frac{2x}{r}-\frac{3}{2}e^{-x}\right)+o\left(\frac{\mathcal{F}(r)}{r}\right)
\\
&=\frac{2\mathcal{F}(r)}{r}\left\{\log{r}\left(1-e^{-y}\right)+y-\log{(\log{r})}+\log{\frac{3}{4}}\right\}+o\left(\frac{\mathcal{F}(r)}{r}\right)
\\
&<\frac{2\mathcal{F}(r)}{r}\left\{\left(\log{r}\right)\left(1-e^{\epsilon}\right)-\log{(\log{r})}+1 \right\}<0,
\end{align*}
which contradicts \eqref{xkatei}. Thus, we have for $t\gg 1$
\begin{align}\label{y>epsilon}
y(t)>-\epsilon.
\end{align}
By \eqref{y<epsilon} and \eqref{y>epsilon}, we have for $t\gg 1$
\begin{align*}
|y(t)|<\epsilon,
\end{align*}
and we obtain \eqref{xnagasa}.
\end{proof}

Next, we estimate $r$.

\begin{lemma}\label{reslem}
Let $\vec{u}$ be a rhombus-type solution. Furthermore, we assume $\varphi=\frac{\pi}{3}$. Then, we have
\begin{align}\label{res1}
r(t)=\log{t}-\frac{d-1}{2}\log{(\log{t})}+c_{\star}+\log{2}+o(1).
\end{align}
\end{lemma}
\begin{proof}
By Lemma \ref{endpointlem1}, we have
\begin{align}\label{rnoode}
\dot{r}=\left(2+o(1)\right)\mathcal{F}(r).
\end{align}
By Lemma \ref{mathcalFyouode} and \eqref{rnoode}, we obtain \eqref{res1}.
\end{proof}

Combining the estimates obtained above, we derive the asymptotic behavior of $a$ and $b$.

\begin{lemma}\label{abes}
Let $\vec{u}$ be a rhombus-type solution. Furthermore, we assume $\varphi=\frac{\pi}{3}$. Then, we have
\begin{align}\label{abessharp}
\begin{aligned}
a(t)&=\frac{1}{2}\log{t}-\frac{d-3}{4}\log{(\log{t})}-\frac{1}{2}\log{(\log{(\log{t})})}+\frac{c_{\star}}{2}+\frac{1}{2}\log{\frac{3}{2}}+o(1),
\\
b(t)&=\frac{\sqrt{3}}{2}\log{t}-\frac{(3d-1)\sqrt{3}}{12}\log{(\log{t})}+\frac{\sqrt{3}}{6}\log{(\log{(\log{t})})}+\frac{\sqrt{3}}{2}c_{\star}
\\
&\quad+\frac{\sqrt{3}}{6}\log{\frac{32}{3}}+o(1).
\end{aligned}
\end{align}
\end{lemma}
\begin{proof}
First, we prove
\begin{align}\label{abessjunbi}
\begin{aligned}
a(t)&=\frac{1}{2}r(t)+\frac{\log{r(t)}}{2}-\frac{\log{(\log{r(t)})}}{2}+\frac{1}{2}\log{\frac{3}{4}}+o(1),
\\
b(t)&=\frac{\sqrt{3}}{2}r(t)-\frac{\log{r(t)}}{2\sqrt{3}}+\frac{\log{(\log{r(t)})}}{2\sqrt{3}}-\frac{1}{2\sqrt{3}}\log{\frac{3}{4}}+o(1).
\end{aligned}
\end{align}
By Lemma \ref{xsharpeslem} and the identity $2a-r=x$, we obtain the estimate for $a$ in \eqref{abessjunbi}. Furthermore, using $a^2+b^2=r^2$, we obtain the corresponding estimate for $b$. We also have
\begin{align}\label{rnotaisuhyouka}
\log{r}=\log{(\log{t})}+o(1),\ \log{(\log{r(t)})}=\log{(\log{(\log{t})})}+o(1).
\end{align}
By \eqref{res1}, \eqref{abessjunbi}, and \eqref{rnotaisuhyouka}, we obtain \eqref{abessharp}.

\end{proof}

\begin{corollary}\label{abescoro}
Let $\vec{u}$ be a rhombus-type solution. Furthermore, we assume $\varphi=\frac{\pi}{6}$. Then, we have
\begin{align}\label{abessharpgyaku}
\begin{aligned}
a(t)&=\frac{\sqrt{3}}{2}\log{t}-\frac{(3d-1)\sqrt{3}}{12}\log{(\log{t})}+\frac{\sqrt{3}}{6}\log{(\log{(\log{t})})}+\frac{\sqrt{3}}{2}c_{\star}
\\
&\quad+\frac{\sqrt{3}}{6}\log{\frac{32}{3}}+o(1),
\\
b(t)&=\frac{1}{2}\log{t}-\frac{d-3}{4}\log{(\log{t})}-\frac{1}{2}\log{(\log{(\log{t})})}+\frac{c_{\star}}{2}+\frac{1}{2}\log{\frac{3}{2}}+o(1).
\end{aligned}
\end{align}

\end{corollary}

\subsection{Proof of Theorem \ref{maintheorem}}

In this subsection, we combine the preceding results to prove Theorem \ref{maintheorem}.

\begin{proof}[Proof of Theorem \ref{maintheorem}]

We first establish the remainder estimate common to all possible asymptotic configurations. By Proposition \ref{samesolitonhanare}, we have \eqref{Vrep}. Therefore, by Lemma \ref{epzes} and Lemma \ref{Fsimt-1lem}, we obtain
\begin{align}\label{commonremainderfinal}
\left\|u(t)-\sum_{k=1}^4(-1)^kQ(\cdot-z_k(t))\right\|_{H^1}+\|{\partial}_tu(t)\|_{L^2}\lesssim \mathcal{F}(D(t)) \lesssim t^{-1}.
\end{align}
Thus, \eqref{4linetheorem1}, \eqref{4rhombustheorem1}, and \eqref{4rhombustheorem3} hold with $\sigma=1$ for the current labeling. A cyclic relabeling of the centers changes the alternating sign pattern at most by a global sign, which can be absorbed into $\sigma=\pm1$.

We now classify the asymptotic behavior of the centers. By Proposition \ref{2,2fullconv1}, every $(2,2)$-sign soliton solution is either a line-type solution or a rhombus-type solution.

First, we assume that $\vec{u}$ is a line-type solution. After a cyclic relabeling if necessary, the center asymptotics \eqref{zestheorem} in Lemma \ref{1-linekansei} give \eqref{4linetheorem2}. Together with \eqref{commonremainderfinal}, this proves the first alternative.

Next, we assume that $\vec{u}$ is a rhombus-type solution. By Lemma \ref{zgshusoku} and Lemma \ref{houkourhombus}, there exist $z_{\infty}\in\mathbb{R}^d$, $\theta_{12}>0$, and $\omega_1,\omega_2\in S^{d-1}$ such that
\begin{align}\label{rhombusdirectionfinal}
\begin{aligned}
\omega_1\cdot\omega_2&=0,
\\
z_1(t)&=z_{\infty}+\hat{\rho}_1(t)\omega_1+O\left(t^{-\theta_{12}}\right),
\\
z_2(t)&=z_{\infty}+\hat{\rho}_2(t)\omega_2+O\left(t^{-\theta_{12}}\right),
\\
z_3(t)&=z_{\infty}-\hat{\rho}_1(t)\omega_1+O\left(t^{-\theta_{12}}\right),
\\
z_4(t)&=z_{\infty}-\hat{\rho}_2(t)\omega_2+O\left(t^{-\theta_{12}}\right).
\end{aligned}
\end{align}
Moreover, by Lemma \ref{hishinagasalem2}, there exists $\varphi\in\left[\frac{\pi}{6},\frac{\pi}{3}\right]$ such that
\begin{align*}
\lim_{t\to\infty}\frac{\hat{\rho}_2(t)}{\hat{\rho}_1(t)}=\tan\varphi.
\end{align*}

We distinguish the following three cases.

Case 1: $\frac{\pi}{6}<\varphi<\frac{\pi}{3}$.

We define 
\begin{align*}
\Lambda(t)=\log{t}-\frac{d-1}{2}\log{(\log{t})}+c_{\star}+\log{2}.
\end{align*}
Then, by Lemma \ref{hitantennagasaharplem}, we obtain
\begin{align*}
\hat{\rho}_1(t)&=\left(\cos\varphi\right)\Lambda(t)+O\left(\frac{\log(\log t)}{\log t}\right),
\\
\hat{\rho}_2(t)&=\left(\sin\varphi\right)\Lambda(t)+O\left(\frac{\log(\log t)}{\log t}\right).
\end{align*}
Combining these estimates with \eqref{rhombusdirectionfinal}, we obtain \eqref{4rhombustheorem2} with
\begin{align*}
u_{\infty}=\omega_1,\ v_{\infty}=\omega_2.
\end{align*}
Together with \eqref{commonremainderfinal}, which gives \eqref{4rhombustheorem1}, this proves the second alternative.

Case 2: $\varphi=\frac{\pi}{3}$.

With the notation in \eqref{abrdef}, we define
\begin{align*}
d_1(t)=a(t)=\hat{\rho}_1(t),\ d_2(t)=b(t)=\hat{\rho}_2(t).
\end{align*}
By Lemma \ref{abes}, we have the last two asymptotic formulas in \eqref{4rhombustheorem4}. On the other hand,
by \eqref{rhombusdirectionfinal}, we have
\begin{align*}
z_1(t)&=z_{\infty}+d_1(t)\omega_1+o(1),
\\
 z_2(t)&=z_{\infty}+d_2(t)\omega_2+o(1),
\\
z_3(t)&=z_{\infty}-d_1(t)\omega_1+o(1),
\\
z_4(t)&=z_{\infty}-d_2(t)\omega_2+o(1).
\end{align*}
Thus, \eqref{4rhombustheorem4} follows with
\begin{align*}
u_{\infty}=\omega_1,\ v_{\infty}=\omega_2.
\end{align*}
Together with \eqref{commonremainderfinal}, which gives \eqref{4rhombustheorem3}, this proves the third alternative in the
case $\varphi=\frac{\pi}{3}$.

Case 3: $\varphi=\frac{\pi}{6}$.

In this case, Corollary \ref{abescoro} gives the same two asymptotic formulas after interchanging $a$ and $b$. More precisely, cyclically relabel the centers by
\begin{align*}
(\tilde{z}_1,\tilde{z}_2,\tilde{z}_3,\tilde{z}_4)=(z_2,z_3,z_4,z_1),
\end{align*}
and set
\begin{align*}
\tilde{u}_{\infty}&=\omega_2,\ \tilde{v}_{\infty}=-\omega_1,
\\
d_1(t)&=b(t),\ d_2(t)=a(t).
\end{align*}
Then $\tilde{u}_{\infty}$ and $\tilde{v}_{\infty}$ form an orthogonal pair, and \eqref{rhombusdirectionfinal} becomes
\begin{align*}
\tilde{z}_1(t)
&=z_{\infty}
+d_1(t)\tilde{u}_{\infty}+o(1),
\\
\tilde{z}_2(t)
&=z_{\infty}
+d_2(t)\tilde{v}_{\infty}+o(1),
\\
\tilde{z}_3(t)
&=z_{\infty}
-d_1(t)\tilde{u}_{\infty}+o(1),
\\
\tilde{z}_4(t)
&=z_{\infty}
-d_2(t)\tilde{v}_{\infty}+o(1).
\end{align*}
The cyclic relabeling changes the alternating sign pattern only by a global sign, which can be absorbed into $\sigma$. After dropping the tildes, the preceding identities and Corollary \ref{abescoro} yield \eqref{4rhombustheorem4}, while
\eqref{commonremainderfinal} yields \eqref{4rhombustheorem3}. Hence, the third alternative also holds
in this endpoint case.

The above cases exhaust all possibilities, and the proof is complete.

\end{proof}

\section*{Statements and Declarations}

\textbf{Funding} No funds, grants, or other support was received.

\textbf{Competing interests} The author has no relevant financial or non-financial interests to disclose.

\textbf{Data availability} No datasets were generated or analyzed during the current study.

\textbf{Use of generative AI} During the preparation of this manuscript, the author used ChatGPT with the GPT-5.6 model (OpenAI) for English-language editing, typographical checks, and assistance in checking calculations and logical consistency. All mathematical arguments and AI-assisted suggestions were independently verified by the author. The author takes full responsibility for the content of the manuscript.

\end{document}